\documentclass[11pt]{amsart}
\usepackage{amsmath,amssymb,amsthm,mathrsfs,enumerate,bm,xcolor,multirow,pbox}
\usepackage{graphicx,color,framed,tikz,caption,subcaption}
\usepackage{enumitem}
\setlist{leftmargin=5mm}
\usepackage[colorlinks,linkcolor=red,citecolor=blue,urlcolor=blue]{hyperref}
\allowdisplaybreaks[4]
\numberwithin{equation}{section}
\newcommand{\N}{\mathbb{N}}
\newcommand{\R}{\mathbb{R}}

\newcommand{\E}{\mathbb{E}}
\newcommand{\Prob}{\mathbb{P}}

\newcommand{\pnorm}[2]{\lVert#1\rVert_{#2}}

\newcommand{\abs}[1]{\lvert#1\rvert}
\newcommand{\bigabs}[1]{\big\lvert#1\big\rvert}
\newcommand{\biggabs}[1]{\bigg\lvert#1\bigg\rvert}

\renewcommand{\epsilon}{\varepsilon}

\renewcommand{\d}[1]{\mathrm{d}#1}

\newcommand{\floor}[1]{\left\lfloor #1 \right\rfloor}

\newcommand{\limd}{\stackrel{d}{\to}}
\newcommand{\equald}{\stackrel{d}{=}}
\renewcommand{\hat}{\widehat}
\renewcommand{\tilde}{\widetilde}
\newcommand{\beq}{\begin{equation}}
\newcommand{\eeq}{\end{equation}}
\newcommand{\beqa}{\begin{equation} \begin{aligned}}
\newcommand{\eeqa}{\end{aligned} \end{equation}}
\newcommand{\beqas}{\begin{equation*} \begin{aligned}}
\newcommand{\eeqas}{\end{aligned} \end{equation*}}

\newcommand{\bit}{\begin{itemize}}
	\newcommand{\eit}{\end{itemize}}
\newcommand{\bmat}{\begin{bmatrix}}
	\newcommand{\emat}{\end{bmatrix}}
\DeclareMathOperator{\argmin}{\mathrm{argmin}}
\DeclareMathOperator{\argmax}{\mathrm{argmax}}

\theoremstyle{definition}\newtheorem{problem}{Problem}[section]
\theoremstyle{definition}
\theoremstyle{remark}\newtheorem{assumption}{Assumption}

\theoremstyle{remark}\newtheorem{remark}[problem]{Remark}
\theoremstyle{definition}
\theoremstyle{plain}\newtheorem{theorem}[problem]{Theorem}
\theoremstyle{plain}\newtheorem{question}{Question}
\theoremstyle{plain}\newtheorem{lemma}[problem]{Lemma}
\theoremstyle{plain}\newtheorem{proposition}[problem]{Proposition}
\theoremstyle{plain}\newtheorem{corollary}[problem]{Corollary}
\theoremstyle{plain}

\AtBeginDocument{%
	\def\MR#1{}
}

\begin{document}

\title[Berry-Esseen bounds for isotonic regression]{Berry-Esseen bounds for Chernoff-type non-standard asymptotics in isotonic regression}
\thanks{The research of Q. Han is partially supported by NSF Grant DMS-1916221. The research of K. Kato is partially supported by NSF grants DMS-1952306 and DMS-2014636.}

\author[Q. Han]{Qiyang Han}

\address[Q. Han]{
Department of Statistics, Rutgers University, Piscataway, NJ 08854, USA.
}
\email{qh85@stat.rutgers.edu}

\author[K. Kato]{Kengo Kato}

\address[K. Kato]{
Department of Statistics and Data Science, Cornell University, Ithaca, NY 14853, USA.
}
\email{kk976@cornell.edu}

\date{\today}


\begin{abstract}
A Chernoff-type distribution is a nonnormal distribution defined by 
the slope at zero of the greatest convex minorant of a two-sided Brownian motion with a polynomial drift. 
While a Chernoff-type distribution is known to appear as the distributional limit in many non-regular statistical estimation problems,  the accuracy of Chernoff-type approximations has remained largely unknown. In the present paper, we tackle this problem and derive Berry-Esseen bounds for  Chernoff-type limit distributions in the canonical non-regular statistical estimation problem of isotonic (or monotone) regression. The derived Berry-Esseen bounds match those of the oracle local average estimator with optimal bandwidth in each scenario of possibly different Chernoff-type asymptotics, up to multiplicative logarithmic factors. Our method of proof differs from standard techniques on Berry-Esseen bounds, and relies on new localization techniques in isotonic regression and an anti-concentration inequality for the supremum of a Brownian motion with a Lipschitz drift.  
\end{abstract}

\maketitle

%\tableofcontents

%\clearpage

\section{Introduction}

\subsection{Overview}

Non-regular statistical estimation problems constitute a class of estimation problems for which natural estimators converge at a rate different from (often slower than) the parametric rate with nonnormal limit distributions. 
Such non-regular estimation problems appear in a variety of statistical problems (cf. \cite{kim1990cube}). 
An important example of nonnormal limit  is a Chernoff-type distribution defined by the slope at zero of the greatest convex minorant of a two-sided Brownian motion with a polynomial drift \cite{chernoff1964estimation,groeneboom2014nonparametric}.
Asymptotic theory for Chernoff-type limiting distributions has been well developed so far; however, the accuracy of such Chernoff-type approximations has remained largely unknown, which poses a fundamental question regarding  the accuracy of statistical inference in non-regular estimation problems. 
Indeed, the complicated nature of the Chernoff-type limit makes the problem of establishing rates of convergence for its distributional approximation substantially challenging from a probabilistic point of view.

In the present paper, we tackle this problem and derive Berry-Esseen bounds for Chernoff-type approximations in the canonical example of monotone or isotonic regression. Estimation and inference using regression models under monotonicity constraints has a long history in  statistics, as they arise as a natural constraint in diverse application fields from economics, genetics, and  to medicine \cite{matzkin1991semiparametric,luss2012efficient,schell1997reduced,lin2012modeling}. Historical remarks and further references in statistical inference under monotonicity constraints can be found in \cite{groeneboom2014nonparametric, samworth2018editorial}. 

Formally, consider the nonparametric regression model
\begin{align}\label{model}
Y_i = f_0(X_i)+\xi_i,\quad i=1,\dots,n,
\end{align}
where $X_1,\ldots,X_n \in [0,1]$ are either fixed or random covariates and $\xi_1,\ldots,\xi_n$ are i.i.d. error variables with mean zero and variance $\sigma^{2} > 0$ (and are independent of $X_{1},\dots,X_{n}$ if random). By isotonic regression, we assume that $f_0$ is nondecreasing, i.e., $f_0 \in \mathcal{F}_{\uparrow}\equiv \{f:[0,1]\to \R :  \text{$f$ is nondecreasing} \}$, and consider the isotonic least squares estimator (LSE):
\begin{align}
\label{def:isotonic_LSE}
\hat{f}_n\equiv \argmin_{f \in \mathcal{F}_{\uparrow}} \sum_{i=1}^n (Y_i-f(X_i))^2.
\end{align}
The isotonic LSE constitutes a representative and rich example of non-regular asymptotics. Suppose that $X_{1},\dots,X_{n}$  are globally equally spaced on $[0,1]$  (i.e., $X_{i} = i/n$ for $i=1,\dots,n$) and $f_0$ is smooth enough at $x_0$ with a first non-vanishing derivative of order $\alpha$ ($\alpha$ can be $\alpha = \infty$, in which case $f_{0}$ is flat). Then,  $\alpha$ is an odd integer  with $f_0^{(\alpha)}(x_0)>0$ if $\alpha$ is finite (cf. \cite{han2020limit}). Let $c_{\alpha} \equiv (f_{0}^{(\alpha)}(x_{0})/(\alpha+1)!)^{1/(2\alpha+1)}$ if $\alpha < \infty$ and $c_{\infty} \equiv 1$ if $\alpha=\infty$, we have
\begin{align}\label{eqn:limit_theory_1d_1}
(n/\sigma^2)^{\alpha/(2\alpha+1)}\big(\hat{f}_n(x_0)-f_0(x_0)\big)\limd 
c_{\alpha} \mathbb{D}_{\alpha}.
\end{align}
%with the convention that $\alpha^{-1} = 0$ if $\alpha = \infty$. 
Here $\mathbb{D}_{\alpha}$ is the slope at zero of the greatest convex minorant of $t \mapsto \mathbb{B}(t)+t^{\alpha+1}$ for $\alpha < \infty$ where $\mathbb{B}$ is a standard two-sided Brownian motion, and $\mathbb{D}_{\infty}$ is defined in Theorem \ref{thm:berry_esseen_isoreg} ahead. 
The canonical case is the $\alpha=1$ case, where the isotonic LSE has the cube-root $n^{-1/3}$ rate and  the limit {theorem} (\ref{eqn:limit_theory_1d_1}) was first proved by \cite{brunk1970estimation}. The distribution of $\mathbb{D}_{1}$ is called the Chernoff distribution, and can be also described as {twice} the argmax of $t \mapsto \mathbb{B}(t) - t^{2}$. 
We shall call the distribution of general $\mathbb{D}_{\alpha}$ a Chernoff-type distribution. These Chernoff-type distributions are non-Gaussian and fairly complicated. For $\alpha=1$, the  detailed analytical properties of the Chernoff distribution $\mathbb{D}_{1}$ are investigated in the seminal work of \cite{groeneboom1989brownian}; see also \cite{groeneboom2015chernoff, balabdaoui2014chernoff}. 

Limit theorems akin to (\ref{eqn:limit_theory_1d_1}) with Chernoff-type limiting distributions appear in a wide range of nonparametric statistical models; see e.g. \cite{grenander1956theory,rao1969estimation,rao1970estimation,rousseeuw1984least,groeneboom1992information,huang1994estimating,huang1995estimation,groeneboom1996lectures,van1998isotonic,wellner2000two,buhlmann2002analyzing,anevski2006general,banerjee2007confidence,anevski2011monotone}, {for an incomplete list.} Further developments on limit theorems for global loss functions and the law of iterated logarithm can be found in \cite{durot2007error,durot2012limit,groeneboom1999asymptotic,jankowski2014convergence,kulikov2005asymptotic, dumbgen2016law}. 

The limit {theorem} in (\ref{eqn:limit_theory_1d_1}) showcases the intrinsic complexity of the non-standard asymptotics with Chernoff-type distributions in the isotonic regression model (\ref{model}), at least from two different angles:
(i) The rate of convergence of the LSE $\hat{f}_n$, i.e. $(n/\sigma^2)^{-\alpha/(2\alpha+1)}$ can \emph{adapt} to the local smoothness level $\alpha$ of the regression function $f_0$ at $x_0$; 
(ii) The limiting distributions $\{\mathbb{D}_\alpha\}$ are different across $\alpha$'s but with certain commonality in terms of being a non-linear and non-smooth functional of a Brownian motion with a drift (except for the case $\alpha =\infty$). 

The main result of the present paper derives Berry-Esseen bounds for the limit {theorem} (\ref{eqn:limit_theory_1d_1}) in a unified setting. 
Specifically, we prove that if the error distribution is sub-exponential and $f_0$ is smooth enough at $x_0$ with a first non-vanishing derivative of order $\alpha$, and a second non-vanishing derivative of order $\alpha^\ast$, then
\begin{equation}
\label{eqn:isotonic_Berry_Esseen}
\begin{split}
&\sup_{t \in \R} \left \lvert\Prob\big((n/\sigma^2)^{\alpha/(2\alpha+1)}\big(\hat{f}_n(x_0)-f_0(x_0)\big)\leq t\big) - \Prob \big(  c_{\alpha} \mathbb{D}_{\alpha}\leq t \big ) \right \rvert \\
&\qquad  \lesssim 
\begin{cases}
\big( n^{-\frac{\alpha^\ast-\alpha}{2\alpha+1}}\vee n^{-\frac{\alpha}{2\alpha+1}} \big)\cdot\textrm{polylog}(n) & \text{if} \ \alpha < \infty \\
n^{-1/2} \cdot \textrm{polylog}(n) & \text{if} \ \alpha = \infty
\end{cases}
\end{split}
\end{equation}
up to constants independent of $n$. In the canonical case of $\alpha = 1$, the bound in (\ref{eqn:isotonic_Berry_Esseen}) is of order $n^{-1/3}$ up to logarithmic factors. Another interesting case is the $\alpha = \infty$ case, where the bound achieves nearly the parametric rate $n^{-1/2}$. 

The rates given in the Berry-Esseen bounds (\ref{eqn:isotonic_Berry_Esseen}) are natural from an oracle perspective.
It is useful to recall that the LSE $\hat{f}_n$ has a well-known representation via the max-min formula (cf. \cite{robertson1988order}): for $x_0 \in (0,1)$,
\begin{align}\label{def:isotonic_maxmin}
\hat{f}_n(x_0) &= \max_{u\leq x_0} \min_{v \geq x_0} \frac{\sum_{i: u\leq X_i\leq v} Y_i }{ \abs{i: u\leq X_i\leq v}} \equiv \max_{u\leq x_0} \min_{v \geq x_0} \bar{Y}|_{[u,v]} = \bar{Y}|_{[u^\ast,v^\ast]}.
\end{align}
{Here $\bar{Y}_A$ is the average of $\{Y_i: i \in A\}$ as defined formally in (\ref{def:average}) ahead.} One can therefore view $\hat{f}_n(x_0)$ as a local average estimator over the sample in a data-driven \emph{random interval} $[u^\ast,v^\ast]$ around $x_0$. Heuristically, the isotonic LSE automatically learns the bias induced by the first non-vanishing derivative, in the sense that the data-driven bandwidth $\abs{v^\ast-u^\ast}$ is of the optimal  order $O_P(n^{-1/(2\alpha+1)})$ as that of an oracle local average estimator. Such oracle behavior gives rise to the rate of convergence $O_P(n^{-\alpha/(2\alpha+1)} )$ in the limit {theorem} (\ref{eqn:limit_theory_1d_1}). The second non-vanishing derivative of order $\alpha^\ast$ then quantifies the rate of convergence for the remaining bias in the standardized statistic $n^{-\alpha/(2\alpha+1)}(\hat{f}_n(x_0)-f_0(x_0))$, yielding the first term $n^{(\alpha^\ast-\alpha)/(2\alpha+1)}$ in (\ref{eqn:isotonic_Berry_Esseen}). On the other hand, the ``effective sample'' for the isotonic LSE is of order $n_e\equiv n\cdot n^{-1/(2\alpha+1)}= n^{2\alpha/(2\alpha+1)}$, and therefore the speed for the noise $\bar{\xi}|_{[u^\ast,v^\ast]}$ to converge in distribution is  of order $(n_e)^{-1/2}=n^{-\alpha/(2\alpha+1)}$. This yields the second term in (\ref{eqn:isotonic_Berry_Esseen}).  These heuristic interpretations on the Berry-Esseen bounds (\ref{eqn:isotonic_Berry_Esseen}) also indicate that the adaptation of the isotonic LSE occurs not only at the level of the rate of convergence of $\hat{f}_n$, but also at the level of the speed of this distributional approximation.

The proof of the Berry-Esseen bounds (\ref{eqn:isotonic_Berry_Esseen}) is highly nontrivial reflecting the complexity of the limit {theorem} (\ref{eqn:limit_theory_1d_1}), and our proof strategies differ substantially from existing techniques on Berry-Esseen bounds (see a literature review below). 
Importantly, in contrast to regular $M$-estimation problems, the isotonic LSE does not admit an asymptotic linear expansion, nor can be approximated by a simple statistic for which existing techniques on Berry-Esseen bounds are applicable.  Our method of proof to establish (\ref{eqn:isotonic_Berry_Esseen}) builds on localization techniques in isotonic regression and an \emph{anti-concentration} inequality (Theorem \ref{thm:anti_concentration}) for the supremum of a Brownian motion with a Lipschitz drift on a compact interval including the origin. Informally, localization shows that (i) $| n^{\alpha/(2\alpha+1)} (\hat{f}_{n}(x_{0}) - f_{0}(x_{0}) ) | \leq O(\sqrt{\log n})$ and (ii)  $n^{1/(2\alpha+1)} \max \{ | x_{0} - u^{*} |, | v^{*} - x_{0} | \} \leq O(\sqrt{\log n})$ with high enough probability.  The former (i) enables us to restrict the range of $t$ in (\ref{eqn:isotonic_Berry_Esseen}) to $|t| \le O(\sqrt{\log n})$, while the latter (ii) enables us to restrict the range of $(u,v)$ in the max-min formula (\ref{def:isotonic_maxmin}) to $O(n^{-1/(2\alpha+1)})$ neighborhoods of $x_{0}$ up to logarithmic factors. Such localization makes possible the application of  the anti-concentration inequality that quantifies the rates of convergence of the bias and the noise to the limit, which are shown to be of the same order as  the desired rate in the Berry-Essen bound (\ref{eqn:isotonic_Berry_Esseen}), up to multiplicative logarithmic factors. 
The prescribed proof techniques can be extended to further Chernoff-type limiting distributions in isotonic regression, allowing both interior and boundary points $x_0$ (cf. \cite{kulikov2006behavior}); and both fixed and random design covariates.

As discussed before, a key technical ingredient of our proof for the Berry Esseen bounds is an explicit anti-concentration inequality for the supremum of a standard Brownian motion with a Lipschitz drift, $T = \sup_{t \in [0,1]}(\mathbb{B}(t) + P(t))$, which is  of independent interest. The anti-concentration inequality quantifies the modulus of continuity of the distribution function of  a random variable, and we need an explicit quantitative anti-concentration inequality of the form $\sup_{u \in \R} \Prob (|T-u| \le \epsilon) \lesssim \epsilon$ up to logarithmic factors to derive the desired Berry-Esseen bounds. The difficulty lies in the fact that the variance of the drifted Brownian motion can be arbitrarily close to zero, so that existing results such as \cite[Lemma 2.2]{chernozhukov2015empirical} are not applicable, at least directly (in addition, it is highly nontrivial to obtain a density formula for $T$ in this generality). To circumvent this problem, we use a carefully designed blocking argument; see the proof of Theorem \ref{thm:anti_concentration}.

The literature on Berry-Essen bounds is broad. 
Berry-Esseen bounds for the classical central limit theorem (CLT) and its various generalizations to multivariate, high-dimensional, and dependent settings can be found in e.g. \cite{bolthausen1982berry,bolthausen1982exact,bentkus1986dependence,gotze1991rate,goldstein1996multivariate,
	rinott1996multivariate,rio1996sur,bentkus1997berry,rinott1997coupling,bentkus2003dependence,chen2004normal,
	chatterjee2006generalization,chatterjee2007multivariate,meckes2009stein,lahiri2009berry,
	reinert2009multivariate,chernozhukov2013gaussian,jirak2016berry,chernozukov2014central,chernozhukov2019improved,chernozhukov2020nearly,fang2020large,fang2020high,fang2020new}, {just to name a few.}
The techniques developed in those references can not be applied to our problem since the isotonic LSE does not admit an asymptotic linear expansion (and thus has a nonnormal limit). 
Stein's method \cite{stein1986approximate,chen2011normal} is known to be a powerful method to derive rates of convergence of distribution approximations. Recent contributions (e.g. \cite{chatterjee2011nonnormal,shao2019berry}) showcase the possibility of using Stein's method for deriving Berry-Esseen bounds with nonnormal limits that admit explicit and easy-to-handle densities; however,  it seems unclear  if the complicated Chernoff distribution is within the reach of such methods. To the best of our knowledge, this is the first paper that derives Berry-Esseen bounds for an important class of Chernoff-type limit distributions.

The rest of the paper is organized as follows. In Section \ref{section:main_result}, we first consider the problem of accuracy of distributional approximation in isotonic regression from an oracle perspective, and then derive the main Berry-Esseen bounds for the isotonic LSE in a unified setup. In  Section \ref{section:anti_concentration}, we prove the key technical result of  anti-concentration inequality, and in Section \ref{section:localization}, we develop the localization techniques in isotonic regression. Building on the techniques developed in Sections \ref{section:anti_concentration} and \ref{section:localization}, we prove the main Berry-Esseen bounds in Section \ref{section:proof_berry_esseen_isotonic}. In Section \ref{section:final_remark}, we conclude the paper and outline a few open questions.
The appendix contains proofs of some auxiliary results and technical tools used in the proofs.

\subsection{Notation}\label{section:notation}

%Let $(\mathcal{F},\pnorm{\cdot}{})$ be a subset of the normed space of real functions $f:\mathcal{X}\to \R$.
{For $\epsilon>0$ and a subset $\mathcal{F}$ of a normed space with norm $\pnorm{\cdot}{}$, let $\mathcal{N}(\epsilon,\mathcal{F},\pnorm{\cdot}{})$ denote the $\epsilon$-covering number of $\mathcal{F}$; see page 83 of \cite{van1996weak} for more details. For the regression model (\ref{model}), for any $A \subset [0,1]$, define
	\begin{align}\label{def:average}
	\bar{Y}|_A \equiv \frac{1}{n_A }\sum_{i: X_i \in A} Y_i, \bar{f_0}|_A \equiv \frac{1}{ n_A }\sum_{i: X_i \in A} f_0(X_i), \bar{\xi}|_A \equiv \frac{1}{ n_A }\sum_{i: X_i \in A} \xi_i
	\end{align}
	where $n_A\equiv \abs{\{i: X_i \in A\}}$ and $0/0 = 0$ by convention.} For two real numbers $a,b$, $a\vee b\equiv \max\{a,b\}$ and $a\wedge b\equiv\min\{a,b\}$. The notation $C_{x}$ will denote a generic constant that depends only on $x$, whose numeric value may change from line to line unless otherwise specified. The notation $a\lesssim_{x} b$ and $a\gtrsim_x b$ mean $a\leq C_x b$ and $a\geq C_x b$ respectively, and $a\asymp_x b$ means $a\lesssim_{x} b$ and $a\gtrsim_x b$ [$a\lesssim b$ means $a\leq Cb$ for some absolute constant $C$].  The notation $\limd$ is {reserved} for convergence in distribution.

\section{Main results}\label{section:main_result}

\subsection{Assumptions}

We first consider local smoothness assumptions on the regression function $f_0$ at $x_0$. We consider both interior ($x_0 \in (0,1)$) and boundary ($x_0 = 0$) points.

\begin{assumption}\label{assump:local_smoothness}
	Let $x_0 \in [0,1)$ be a fixed point of interest. Let $\alpha, \alpha^\ast \in \mathbb{Z}_{\geq 1} \cup \{\infty\}, \alpha+1\leq \alpha^\ast \leq \infty$ be such that $f_0^{(\alpha)}(x_0)\neq 0$ and $f_0^{(\alpha^\ast)}(x_0)\neq 0$ if $\alpha, \alpha^\ast\neq \infty$, and the Taylor expansion
	\begin{align*}
	f_0(x) &= f_0(x_0)+\frac{f_0^{(\alpha)}(x_0)}{\alpha!}(x-x_0)^\alpha \bm{1}_{\alpha<\infty}\\
	&\quad\quad + \frac{f_0^{(\alpha^\ast)}(x_0)}{\alpha^\ast!}(x-x_0)^{\alpha^\ast}\bm{1}_{\alpha^\ast<\infty}+ R\big((x-x_0)^{\alpha^\ast}\bm{1}_{\alpha^\ast<\infty}\big) 
	\end{align*}
	holds for all $x \in [0,1]$ for some function $R: \R\to \R$ such that $R(0)=0$ and $R(\epsilon) = o(\epsilon)$ as $\epsilon \to 0$. 
\end{assumption}

If $x_0 = 0 $, then the derivatives are understood as one-side limits. Assumption \ref{assump:local_smoothness} essentially says that $f_0$ has a first non-vanishing derivative at $x_0$ of order $\alpha$, and a second one of order $\alpha^\ast$. If $x_0 \in (0,1)$, by \cite[Lemma 1]{han2020limit}, $\alpha$ must be an odd integer, and $f_0^{(\alpha)}(x_0)>0$ under Assumption \ref{assump:local_smoothness}. If $x_0= 0$, $\alpha$ need not be an odd integer, but $f_0^{(\alpha)}(x_0) > 0$. We do not consider $x_0=1$ as the situation is similar to $x_0=0$.

The following are some examples satisfying Assumption \ref{assump:local_smoothness}. 
\begin{enumerate}
	\item[(i)] $f_0(x)=x$. Then $\alpha =1$ and $\alpha^\ast =\infty$ at $x_0=1/2$.
	\item[(ii)] $f_0(x)=e^x$. Then $\alpha=1$ and $\alpha^\ast=2$ at $x_0=1/2$.
	\item[(iii)] $f_0(x)=(x-1/2)^3$. Then $\alpha = 3$ and $\alpha^\ast =\infty$ at $x_0=1/2$.
	\item[(iv)] $f_0(x)=(x-1/2)^3+(x-1/2)^5$. Then $\alpha=3$ and  $\alpha^\ast=5$ at $x_0=1/2$.
	\item[(v)] $f_0(x)=x^2+x^4$. Then $\alpha=2$ and $\alpha^\ast=4$ at $x_0=0$.
\end{enumerate}

When $x_0=0$, we consider limit distribution theory at $x_n = n^{-\rho}$ where $\rho \in (0,1)$. Namely, we estimate $f_{0}(0)$ by $\hat{f}_{n}(x_{n})$. For notational convenience, let
\begin{align}\label{def:x_star}
x^\ast = x_0\bm{1}_{x_0 \in (0,1)}+x_n \bm{1}_{x_0=0}=x_0\bm{1}_{x_0 \in (0,1)}+n^{-\rho} \bm{1}_{x_0=0}.
\end{align}

Next we state assumptions on the design points.

\begin{assumption}\label{assumption:design_points}
	Suppose that the design points $\{X_i\}_{i=1}^{n}$ satisfy either of the following conditions. 
	\begin{itemize}
		\item \textit{(Fixed design)} $X_1,\ldots,X_n \in [0,1]$ are deterministic, and {there exists some $\Lambda_0>0$} such that for some $\delta_0>0$, the design points restricted to $I_0\equiv [0\vee (x_0-\delta_0),1\wedge (x_0+\delta_0)]$,  $\{X_i: X_i \in I_{0}, i=1,\dots,n \}$, are equally spaced with distance $1/(\Lambda_0 n)$.\footnote{In other words, $X_1\leq X_2\leq \ldots\leq X_{n-1}\leq X_n$ and $X_{i+1}-X_i = X_{i}-X_{i-1} = 1/(\Lambda_0 n)$ whenever $X_{i-1},X_i,X_{i+1} \in I_0$.}
		
		In the case $\alpha=\infty$, or $x_0=0$ and $\rho \in [1/(2\alpha+1),1)$, we assume that the design points are globally equally spaced on $[0,1]$ (i.e. $X_i=i/n, i=1,\ldots,n$, so $\Lambda_0=1$).
		\item \textit{(Random design)} $X_1,\ldots,X_n$ are i.i.d. with law $P$ on $[0,1]$, and $P$ admits a Lebesgue density $\pi$ that is continuous around $x_0$ and is bounded and bounded away from $0$  on $[0,1]$. Further assume that for some $1\leq \beta\leq \infty$,
		\begin{align*}
		\pi(x)-\pi(x_0) = \frac{\pi^{(\beta)}(x_0)}{\beta!}(x-x_0)^\beta\bm{1}_{\beta<\infty} + R_\pi \big((x-x_0)^\beta \bm{1}_{\beta<\infty}\big)
		\end{align*}
		holds for all $x \in [0,1]$ for some function $R_\pi : \R\to \R$ such that $R_\pi(0)=0$ and $R_\pi(\epsilon)=o(\epsilon)$ as $\epsilon \to 0$. {Let $\Lambda_0=\pi(x_0)$}. 
		
		In the case $\alpha=\infty$ or $x_0=0, \rho \in [1/(2\alpha+1),1)$, we assume that $P$ is the uniform distribution on $[0,1]$.
	\end{itemize}
\end{assumption}

{The canonical case is the globally equally spaced fixed design with $X_i = i/n, i=1,\ldots,n$ (so $\Lambda_0 = 1$).} Furthermore, we have made more specific assumptions on the designs of the covariates when $\alpha=\infty$, or $x_0=0$ and $\rho \in [1/(2\alpha+1),1)$ due to the non-local nature of the limit distribution theory in such scenarios. This helps us to develop unified Berry-Esseen bounds for the isotonic LSE.

\subsection{Oracle considerations}
To gain some insights into what should be expected for a Berry-Esseen bound for the non-standard limit {theorem} (\ref{eqn:limit_theory_1d_1}), we shall first look at the problem from an oracle perspective. Suppose that Assumption \ref{assump:local_smoothness} holds, and the regularity of $f_0$ at $x_0$ is known. Consider the local average estimator 
\begin{align}\label{def:local_average}
\bar{f}_n(x^\ast)\equiv \bar{f}_n(x^\ast;r_n,h) = \bar{Y}|_{[x^\ast-h_1 r_n,x^\ast+h_2 r_n]}	
\end{align}
with a tuning parameter $r_n>0$ and constants $h_1,h_2>0$. The isotonic least squares estimator $\hat{f}_n$, defined via the max-min formula (\ref{def:isotonic_maxmin}), can be viewed as a local average estimator (\ref{def:local_average}) with automatic data-driven choices of the tuning parameters $h_1,h_2,r_n$.

An oracle local average estimator $\bar{f}_n$ knows the regularity of $f_0$ at $x_0$ and chooses the bandwidth $r_n$ of the following optimal order:
\begin{align}\label{def:r_n_oracle}
r_n\equiv 
\begin{cases}
n^{-1/(2\alpha+1)}, &\quad \text{if} \ x_0 \in (0,1)\\
n^{-(1-2\rho(\alpha-1))/3}, &\quad \text{if} \ x_0=0 \ \text{and} \ \rho \in (0,1/(2\alpha+1))\\
n^{-\rho}, &\quad \text{if} \ x_0=0 \ \text{and} \ \rho \in  [1 /(2\alpha+1),1)
\end{cases}
,
\end{align}
and hence the local rate of convergence of the oracle estimator is given by
\begin{align}\label{def:w_n_oracle}
\omega_n^{-1}\equiv (nr_n)^{1/2} =  
\begin{cases}
n^{\alpha/(2\alpha+1)}, &\quad \text{if} \ x_0 \in (0,1)\\
n^{(1+\rho(\alpha-1))/3}, &\quad \text{if} \ x_0=0 \ \text{and} \ \rho \in (0,1/(2\alpha+1))\\
n^{(1-\rho)/2}, &\quad \text{if} \ x_0=0 \ \text{and} \ \rho \in  [1 /(2\alpha+1),1)
\end{cases}
.
\end{align}
For instance, in the canonical case {where $\alpha = 1$ and $x_{0} \in (0,1)$}, then $r_{n} = \omega_{n} = n^{-1/3}$. To describe the limiting distribution of the oracle estimator, further define
\begin{align}\label{def:Q}
Q(h) &  \equiv 
\begin{cases}
\frac{f_0^{(\alpha)}(x_0)}{(\alpha+1)!}\cdot h^{\alpha+1} \bm{1}_{\alpha<\infty},&\quad \text{if} \ x_0 \in (0,1)\\
\frac{f_0^{(\alpha)}(0)}{2(\alpha-1)!} \cdot h^2, &\quad  \text{if} \ x_0 =0 \ \text{and} \ \rho \in (0,1/(2\alpha+1))\\
\sum_{\ell=1}^\alpha \frac{f_0^{(\alpha)}(0)}{(\alpha-\ell)!(\ell+1)!}\cdot h^{\ell+1}  ,&\quad  \text{if} \  x_0 = 0 \ \text{and} \  \rho = 1/(2\alpha+1) \\
0, &\quad \text{if} \ x_0=0 \ \text{and} \  \rho \in (1/(2\alpha+1),1)
\end{cases}
,
\end{align}
and
\begin{align}\label{ineq:def_B}
& \mathbb{B}_{\sigma,\Lambda_0,Q}(h_1,h_2) \equiv (\sigma/\Lambda_0^{1/2})\cdot \frac{\mathbb{B}(h_2)-\mathbb{B}(-h_1)}{h_1+h_2} + \frac{Q(h_2)-Q(-h_1)}{h_1+h_2}
\end{align}
where $\mathbb{B}$ is a  standard two-sided Brownian motion starting from $0$.

\begin{proposition}
	[Berry-Esseen bounds: Oracle considerations]
	\label{prop:oracle_CLT}
	Let $\xi_i$'s be i.i.d. errors with finite third moment and $\E\xi_1^2=\sigma^2$. Suppose Assumptions \ref{assump:local_smoothness} and \ref{assumption:design_points} hold. Then with $\omega_n^{-1}$ defined in (\ref{def:w_n_oracle}) and $\mathbb{B}_{\sigma,\Lambda_0,Q}$ defined in (\ref{ineq:def_B}), the local average estimator $\bar{f}_n$ defined in (\ref{def:local_average}) with oracle bandwidth $r_n$ defined in (\ref{def:r_n_oracle}) satisfies 
	\[
	\sup_{t \in \R} \bigg\lvert\Prob\bigg(\omega_n^{-1}\big(\bar{f}_n(x^\ast;r_n,h)-f_0(x^\ast) \big)\leq t\bigg)-\Prob\bigg( \mathbb{B}_{\sigma,\Lambda_0,Q}(h_1,h_2)\leq t\bigg) \bigg\lvert \leq K\cdot \mathcal{B}_n.
	\]
	The constant $K>0$ does not depend on $n$, and with $\bm{1}^r$ denoting the indicator for the random design case, 
	\begin{align}\label{def:B_n}
	\mathcal{B}_n&  \equiv 
	\begin{cases}
	\max\Big\{n^{-\frac{\alpha}{2\alpha+1}}(\log n)^{\bm{1}_{\alpha<\infty}\cdot \bm{1}^r}, \\
	\qquad \qquad n^{-\frac{ \alpha^\ast-\alpha}{2\alpha+1}}\bm{1}_{\alpha^\ast<\infty}, \\ 
	\qquad \qquad n^{-\frac{\beta}{2\alpha+1}}\bm{1}_{\alpha\vee \beta<\infty} \bm{1}^r\Big\}, & \text{if} \ x_0 \in (0,1)\\
	\max \Big\{ n^{-(1-(2\alpha+1)\rho)/3}, \\
	\qquad\qquad n^{-\rho(\alpha^\ast-\alpha)}\bm{1}_{\alpha^\ast<\infty} \Big\}, & \text{if} \ x_0 =0 \ \text{and} \  \rho \in (0,1/(2\alpha+1))\\
	\max\Big\{n^{-\frac{\alpha}{2\alpha+1}}(\log n)^{\bm{1}_{\alpha<\infty}\cdot \bm{1}^r}, \\
	\qquad\qquad n^{-\frac{ \alpha^\ast-\alpha}{2\alpha+1}}\bm{1}_{\alpha^\ast<\infty},\\
	\qquad\qquad n^{-\frac{\beta}{2\alpha+1}}\bm{1}_{\alpha\vee \beta<\infty} \bm{1}^r \Big\}, & \text{if} \ x_0 = 0 \ \text{and} \  \rho = 1/(2\alpha+1)\\
	\max \Big\{ n^{-(1-\rho)/2 }(\log n)^{\bm{1}_{\alpha<\infty}\cdot \bm{1}^r},\\
	\qquad\qquad  n^{-((2\alpha+1)\rho-1)/2}\Big\}, & \text{if} \ x_0 = 0 \ \text{and} \ \rho \in (1/(2\alpha+1),1)
	\end{cases}
	.
	\end{align}
	Furthermore, the above Berry-Esseen bound cannot be improved in general, except for the logarithmic factors in the random design case.
\end{proposition}

In general, the rate $\mathcal{B}_n$ above is determined by the order of the leading term in the remainders of (\ref{def:local_average}) after centering and normalization at the rate $\omega_n^{-1}$. In particular, different terms in the rate $\mathcal{B}_n$ come from different sources in different scenarios:
\begin{itemize}
	\item For $x_0 \in (0,1)$, $n^{-\frac{\alpha}{2\alpha+1}}$ is the rate for the noise to approximate its Gaussian limit, while $n^{-\frac{\alpha^\ast-\alpha}{2\alpha+1}}$ is the rate induced by the second non-vanishing derivative of $f_0$ of order $\alpha^\ast$ at $x_0$.
	\item For $x_0 = 0$ and $\rho \in (0,1/(2\alpha+1))$, $n^{-(1-(2\alpha+1)\rho)/3}$ is the rate induced by the second order bias (since in this case the first order bias contributes to the limiting distribution), while $n^{-\rho(\alpha^\ast-\alpha)}$ is the rate induced by the second non-vanishing derivative of $f_0$ of order $\alpha^\ast$ at $0$. The rate for the noise to approximate its Gaussian limit is dominated by the maximum of the two rates.
	\item For $x_0=0$ and $\rho = 1/(2\alpha+1)$, $n^{-\frac{\alpha}{2\alpha+1}}$ is the rate for the noise to approximate its Gaussian limit, while $n^{-\frac{\alpha^\ast-\alpha}{2\alpha+1}}$ is the rate induced by the second non-vanishing derivative of $f_0$ of order $\alpha^\ast$ at $x_0$.
	\item For $x_0=0$ and $\rho \in (1/(2\alpha+1),1)$, $n^{-(1-\rho)/2 }$ is the rate for the noise to approximate its Gaussian limit, while $n^{-((2\alpha+1)\rho-1)/2}$ is the rate induced by the first non-vanishing derivative of $f_0$ of order $\alpha^\ast$ at $x_0$ (since in this case $Q \equiv 0$).
	\item The rates involving $\beta$ come from the regularity of the design density in the random design setting. They appear when $x_0 \in (0,1)$ or $x_0=0,\rho = 1/(2\alpha+1)$.
\end{itemize}

In the next subsection we will show that the isotonic least squares estimator $\hat{f}_n$ converges to the limiting Chernoff distribution at a rate no slower than the oracle rate $\mathcal{B}_n$, up to logarithmic factors.

\begin{proof}[Proof of Proposition \ref{prop:oracle_CLT}]
	First consider the fixed design case {with the additional assumption that $x^\ast \in \{X_i\}$. Applying Lemma \ref{lem:bias_calculation} below in Section \ref{section:localization} with any fixed positive real number $\tau_n\geq h_1\vee h_2$}, for $x_0 \in (0,1)$,
	\begin{align*}
	& \bar{f_0}|_{[x_0-h_1 r_n,x_0+h_2 r_n]  }-f_0(x_0)\\
	& = \frac{f_0^{(\alpha)}(x_0)}{(\alpha+1)!}\cdot \frac{h_2^{\alpha+1}-h_1^{\alpha+1}}{h_1+h_2}\cdot r_n^{\alpha}\bm{1}_{\alpha<\infty} + O\bigg(r_n^{\alpha^\ast}\bm{1}_{\alpha^\ast<\infty}\bigvee r_n^\alpha (nr_n)^{-1} \bm{1}_{\alpha<\infty} \bigg).
	\end{align*}
	For $x_0 = 0$,
	\begin{align*}
	& \bar{f_0}|_{[x_n-h_1 r_n,x_n+h_2 r_n]  }-f_0(0)\\
	& = f_0^{(\alpha)}(0) \sum_{\ell=1}^{\alpha} \frac{1}{(\alpha-\ell)!(\ell+1)!} \cdot  \frac{h_2^{\ell+1}-(-h_1)^{\ell+1}}{h_1+h_2}\cdot x_n^{\alpha-\ell} r_n^\ell \bm{1}_{\alpha<\infty} \\
	&\quad\quad +O\bigg(\max_{1\leq \ell\leq \alpha^\ast} x_n^{\alpha^\ast-\ell}r_n^\ell \bm{1}_{\alpha^\ast<\infty}\bigvee \max_{1\leq \ell\leq \alpha} x_n^{\alpha-\ell}r_n^\ell (nr_n)^{-1} \bm{1}_{\alpha<\infty}\bigg)\\
	& = 
	\begin{cases}
	\frac{f_0^{(\alpha)}(0)}{2(\alpha-1)!}\cdot \frac{h_2^2-h_1^2}{h_1+h_2}\cdot x_n^{\alpha-1} r_n \bm{1}_{\alpha<\infty}\\
	\quad+ O\big(x_n^{\alpha-2}r_n^2 \bm{1}_{\alpha<\infty}\vee x_n^{\alpha^\ast-1}r_n\bm{1}_{\alpha^\ast<\infty}\\
	\quad\quad\quad\qquad\vee x_n^{\alpha-1}r_n(nr_n)^{-1}\bm{1}_{\alpha<\infty}\big) ,&\quad \text{if} \ x_n\gg r_n\\
	f_0^{(\alpha)}(0) \big(\sum_{\ell=1}^{\alpha} \frac{1}{(\alpha-\ell)!(\ell+1)!} \cdot  \frac{h_2^{\ell+1}-(-h_1)^{\ell+1}}{h_1+h_2}\big)\cdot r_n^{\alpha} \bm{1}_{\alpha<\infty}\\
	\quad + O\big(r_n^{\alpha^\ast}\bm{1}_{\alpha^\ast<\infty}\vee r_n^{\alpha}(nr_n)^{-1} \bm{1}_{\alpha<\infty}\big), &\quad \text{if} \ x_n = r_n
	\end{cases}
	.
	\end{align*}
	Let $W_n\equiv \sqrt{nr_n}\cdot\big(\bar{\xi}|_{[x^\ast-h_1r_n,x^\ast+h_2r_n]}\big)$,
	\begin{align*}
	Z_h \equiv (\sigma/\Lambda_0^{1/2})\cdot \frac{\mathbb{B}(h_2)-\mathbb{B}(-h_1)}{h_1+h_2},\quad \text{and} \quad  \mu \equiv \frac{Q(h_2)-Q(-h_1)}{h_1+h_2}.
	\end{align*}
	Note that $\mathbb{B}_{\sigma,\Lambda_0,Q}(h_1,h_2) \stackrel{d}{=} Z_{h} + \mu \sim \mathcal{N}(\mu,\sigma^{2}/ (\Lambda_{0}(h_1+ h_2)))$. Further, let
	\begin{align}\label{def:R_n_oracle}
	\mathcal{R}_n^{f}\equiv 
	\begin{cases}
	\Big [ \sqrt{nr_n^{2\alpha^\ast+1}}\bm{1}_{\alpha^\ast<\infty} \vee r_n^\alpha (nr_n)^{-1/2}\bm{1}_{\alpha<\infty} \Big ],& \text{if} \ x_0 \in (0,1)\\
	\Big [ x_n^{\alpha-2}\sqrt{nr_n^5} \bm{1}_{\alpha<\infty}\vee x_n^{\alpha^\ast-1}\sqrt{nr_n^3}\bm{1}_{\alpha^\ast<\infty}\\
	\quad\quad \vee x_n^{\alpha-1}r_n(n r_n)^{-1/2}\bm{1}_{\alpha<\infty} \Big ], &  \text{if} \ x_0 =0 \ \text{and} \  \rho\in (0,1/(2\alpha+1))\\
	\Big [ \sqrt{n r_n^{2\alpha^\ast+1}}\bm{1}_{\alpha^\ast<\infty} \vee r_n^\alpha (n r_n)^{-1/2}\bm{1}_{\alpha<\infty} \\
	\quad\quad \vee \sqrt{nr_n^{2\alpha+1}}\bm{1}_{\rho>1/(2\alpha+1),\alpha<\infty} \Big ],& \text{if}\ x_0 = 0 \ \text{and} \ \rho \in  [1/(2\alpha+1),1)
	\end{cases}
	.
	\end{align}
	Then, uniformly in $t \in \R$,
	\begin{align*}
	&\Prob\bigg(\sqrt{ nr_n}\big(\bar{f}_n(x^\ast;r_n,h)-f_0(x^\ast) \big)\leq t\bigg)\\
	&= \Prob\bigg(\sqrt{ n r_n}\big(\bar{\xi}|_{[x^\ast-h_1r_n,x^\ast+h_2r_n]}\big)+ \sqrt{n r_n}\big(\bar{f_0}|_{[x^\ast-h_1r_n,x^\ast+h_2r_n]}-f_0(x_0)\big)\leq t\bigg)\\
	& =\Prob\bigg( W_n + \mu +O\big(\mathcal{R}_n^f\big) \leq t\bigg)= \Prob\bigg( Z_h + \mu +O\big(\mathcal{R}_n^f\big) \leq t\bigg) {+}O\big((n r_n)^{-1/2}\big)\\
	& = \Prob\big(Z_h+\mu \leq t\big) {+}  O\big( \mathcal{R}_n^f \vee (n r_n)^{-1/2} \big).
	\end{align*}
	The second last line follows from the classical Berry-Esseen bound, and the last line follows from the anti-concentration of a standard normal random variable: it holds that $\sup_{t \in \R}\Prob(\abs{Z-t}\leq \epsilon)\leq  \epsilon \sqrt{2/\pi}$ where $Z\sim \mathcal{N}(0,1)$. The remainder term cannot be improved in general by the sharpness of the Berry-Esseen bound for the central limit theorem, cf. \cite{hall1984reversing}. Calculations show that $\mathcal{R}_n^f \vee (n r_n)^{-1/2}=\mathcal{B}_n$ in the fixed design case {with $x^\ast \in \{X_i\}$. For $x^\ast$ in general position, using Remark \ref{rmk:removal_x_on_design}, the error bound is of order at most $(\mathcal{R}_n^f\vee (\omega_n^{-1}\cdot n^{-1})) \vee (n r_n)^{-1/2}=\mathcal{R}_n^f \vee (n r_n)^{-1/2}=\mathcal{B}_n$.} 
	
	For the random design case, let 
	\begin{align}\label{def:R_n_oracle_random}
	\mathcal{R}_n^{r}\equiv 
	\begin{cases}
	\Big [ \sqrt{n r_n^{2\alpha^\ast+1}}\bm{1}_{\alpha^\ast<\infty} \bigvee \sqrt{n r_n^{2(\alpha+\beta)+1}}\bm{1}_{\alpha\vee \beta<\infty}\\
	\quad \bigvee \sqrt{r_n^{2\alpha}\log n \vee \frac{\log^2 n}{nr_n} }\cdot \bm{1}_{\alpha<\infty} \Big ],& \text{if} \ x_0 \in (0,1)\\
	\Big [ x_n^{\alpha-2}\sqrt{n r_n^5} \bm{1}_{\alpha<\infty}\bigvee x_n^{\alpha^\ast-1}\sqrt{n r_n^3}\bm{1}_{\alpha^\ast<\infty}\\
	\quad\quad \bigvee x_n^{\alpha-1}\sqrt{n r_n^{2\beta+3}}\bm{1}_{\alpha\vee \beta<\infty} \\
	\quad \quad \bigvee x_n^{\alpha-1} \sqrt{ r_n^2 \log n\vee \frac{\log^2 n}{nr_n}}\cdot \bm{1}_{\alpha<\infty} \Big ], &  \text{if}  \ x_0 =0 \ \text{and} \ \rho\in (0,1/(2\alpha+1))\\
	\Big [ \sqrt{n r_n^{2\alpha^\ast+1}}\bm{1}_{\alpha^\ast<\infty} \bigvee \sqrt{n r_n^{2(\alpha+\beta)+1}}\bm{1}_{\alpha\vee \beta<\infty}\\
	\quad \bigvee \sqrt{r_n^{2\alpha}\log n \vee \frac{\log^2 n}{nr_n} }\cdot \bm{1}_{\alpha<\infty}\\
	\quad \bigvee \sqrt{n r_n^{2\alpha+1}} \bm{1}_{\rho>1/(2\alpha+1),\alpha<\infty} \Big ],& \text{if} \  x_0 = 0 \ \text{and} \ \rho \in  [1/(2\alpha+1),1)
	\end{cases}
	.
	\end{align}
	Tedious and patient calculations show that $\mathcal{R}_n^r \vee (n r_n)^{-1/2}=\mathcal{B}_n$ in the random design case. For terms involving $\log n$, the bounds cannot be improved by considering $\alpha=\infty$. 
\end{proof}

\subsection{Berry-Esseen bounds}

Some further definitions for $H_1,H_2$:

\vspace{0.5ex}
\setlength{\tabcolsep}{4pt} 
\renewcommand{\arraystretch}{1.2} 
\begin{center}
	\begin{tabular}{|c||c|c|}
		\hline 
		& $H_1$ & $H_2$\\
		\hline\hline
		$\alpha<\infty$ & $\begin{cases}
		(0,\infty),& x_0 \in (0,1)\\
		(0,\infty) , &  x_0 =0, \rho \in (0,1/(2\alpha+1))\\
		(0,1] ,& x_0 = 0, \rho \in [1/(2\alpha+1),1)
		\end{cases}$ & $[0,\infty)$\\
		\hline
		$\alpha=\infty$ & $\begin{cases}
		(0,x_0], & x_0 \in (0,1)\\
		(0,1],& x_0 = 0
		\end{cases}$ & $\begin{cases}
		[0,1-x_0], & x_0 \in (0,1)\\
		[0,\infty),&x_0 = 0
		\end{cases}$ \\
		\hline
	\end{tabular}
	\vspace{1ex}
	\captionof{table}{Definitions of $H_1,H_2$.}\label{table:def_H}
\end{center}
Now we present the main results of this paper, i.e., Berry-Esseen bounds for (\ref{eqn:limit_theory_1d_1}) and its generalizations in isotonic regression. 

\begin{theorem}[Berry-Esseen bounds for isotonic LSE]
	\label{thm:berry_esseen_isoreg}
	Let $\xi_i$'s be i.i.d. mean-zero sub-exponential errors, i.e.,  $\E \xi_1 = 0$ and $\E e^{\theta \xi_1}<\infty$ for all $\theta$ in a neighborhood of the origin.  Let $\sigma^2\equiv \E \xi_1^2$. Suppose Assumptions \ref{assump:local_smoothness} and \ref{assumption:design_points} hold, and $\rho \in (0,1/(2\alpha+1)]\cup [2/3,1)$. Then with $\omega_n^{-1}$ defined in (\ref{def:w_n_oracle}), $\mathbb{B}_{\sigma,\Lambda_0,Q}$ defined in (\ref{ineq:def_B}) and $H_1,H_2$ defined in Table \ref{table:def_H}, the isotonic least squares estimator $\hat{f}_n$ defined in (\ref{def:isotonic_maxmin}) satisfies
	\begin{align*}
	&\sup_{t \in \R} \bigg\lvert\Prob\bigg(\omega_n^{-1}\big(\hat{f}_n(x^\ast)-f_0(x^\ast) \big)\leq t\bigg) \\
	&\quad\quad\quad\quad\quad\quad  -\Prob\bigg( \sup_{h_1 \in H_1}\inf_{h_2 \in H_2} \mathbb{B}_{\sigma,\Lambda_0,Q}(h_1,h_2)\leq t\bigg) \bigg\lvert \leq K\cdot \mathcal{B}_n (\log n)^{\zeta_{\alpha,\alpha^\ast,\beta}}.
	\end{align*}
	The constant $K>0$ does not depend on $n$, $\mathcal{B}_n$ is defined in (\ref{def:B_n}) in the statement of Proposition \ref{prop:oracle_CLT}, and {$\zeta_{\alpha,\alpha^\ast,\beta}>0$ is a constant depending only on $\alpha,\alpha^\ast,\beta$.}
	%\begin{align*}
	%\zeta_{\alpha,\alpha^\ast,\beta} = 
	%\begin{cases}
	% 3+\frac{1\vee \alpha^\ast \bm{1}_{\alpha^\ast<\infty}}{2},&\textrm{fixed design}\\
	% \frac{5}{2}+ \bm{1}_{\alpha<\infty}\big[ \frac{1+\alpha^\ast\bm{1}_{\alpha^\ast<\infty}\vee (\beta+1)\bm{1}_{\beta<\infty}}{2}\vee \frac{5}{4}\big]+\bm{1}_{\alpha=\infty},& \textrm{random design}
	%\end{cases}
	%.
	%\end{align*}
\end{theorem}

\begin{proof}
	See Section \ref{section:proof_berry_esseen_isotonic}.
\end{proof}

{It is  possible to track the numerical value of $\zeta_{\alpha,\alpha^\ast,\beta}$ in the proofs, but its value may not be optimal. For brevity, we omit the numerical value of $\zeta_{\alpha,\alpha^\ast,\beta}$ in the statement of the theorem. }

\begin{remark}[Limit distributions]
	The limiting distribution in Theorem \ref{thm:berry_esseen_isoreg} is written in a compact and unified form which may not be familiar in the literature. We will recover the more familiar forms using the following switching relation: Let $H$ be an (open or closed) interval contained in $\R$, and $\mathrm{LCM}_H$ (resp. $\mathrm{GCM}_H$) be the least concave majorant (resp. greatest convex minorant) operator on $H$ and $\mathrm{LCM}_H(\cdot)'$ (resp. $\mathrm{GCM}_H(\cdot)'$) be its left derivative. Then for any $F: H \to \R$, and $t,a \in \R$, we have (cf. \cite[Lemma 3.2]{groeneboom2014nonparametric})
	\begin{align*}
	\mathrm{LCM}_H(F)'(t)\geq a\quad &\Leftrightarrow \quad \mathrm{GCM}_{H} (-F)' (t) \leq -a \\
	&\Leftrightarrow \quad \argmin_{u \in H}  \{ -F(u) + au \} \geq t \\
	&\Leftrightarrow \quad \argmax_{u \in H} \{F(u)-au\}\geq t.
	\end{align*}
	If there are multiple maxima (resp. minima) in the map $u\mapsto F(u)-au$ (resp. $u\mapsto -F(u)+au$) , then the argmax (resp. argmin) is defined to be the location of the first maximum (resp. minimum).

	\begin{itemize}
		\item Let $x_0 \in (0,1), \alpha<\infty$. Then $Q(h) = \frac{f_0^{(\alpha)}(x_0)}{(\alpha+1)!} h^{\alpha+1}, H_1=(0,\infty), H_2 = [0,\infty)$, and we have
		\begin{align*}
		&\sup_{h_1 \in (0,\infty)}\inf_{h_2 \in [0,\infty)} \bigg[\frac{\mathbb{B}(h_2)-\mathbb{B}(-h_1)}{h_1+h_2}+\frac{Q(h_2)-Q(-h_1)}{h_1+h_2}\bigg]\leq t\\
		\Leftrightarrow\quad& \forall h_1 \in (-\infty,0),\exists h_2 \in [0,\infty),\\
		&\quad\quad -\mathbb{B}(h_2)-Q(h_2)+th_2\geq -\mathbb{B}(h_1)-Q(h_1)+th_1\\
		\Leftrightarrow\quad & \argmax_{u \in \R}\big(-\mathbb{B}(u)-Q(u)-(-t)u\big)\geq 0\\
		\Leftrightarrow\quad & \mathrm{LCM}_{\R}\big(-\mathbb{B}(u)-Q(u)\big)'(0)\geq -t \\
		\Leftrightarrow\quad & \mathrm{GCM}_{\R}\big(\mathbb{B}(u)+Q(u)\big)'(0)\leq t \\
		\stackrel{d}{\Leftrightarrow} \quad & \bigg(\frac{f_0^{(\alpha)}(x_0)}{(\alpha+1)!}\bigg)^{1/(2\alpha+1)}\cdot \mathbb{D}_{\alpha}\leq t,
		\end{align*}
		where $\mathbb{D}_{\alpha}$ is the slope at zero of the greatest convex minorant of $t\mapsto \mathbb{B}(t)+t^{\alpha+1}$, and the last equivalence in distribution follows from a standard Brownian scaling argument. In particular, for $\alpha = 1$, we have 
		\[
		\mathbb{D}_{1} \stackrel{d}{=} 2 \cdot  \argmax_{h \in \R} \{ \mathbb{B}(h) -h^{2} \},
		\]
		where the argmax on the right hand side is a.s. uniquely defined by \cite[Lemma 2.6]{kim1990cube}. 
		See \cite[Problem 3.12]{groeneboom2014nonparametric}. The case for $x_0=0, \rho \in (0,1/(2\alpha+1))$ is similar as $H_1=(0,\infty), H_2 = [0,\infty)$ as above. 
		\item Let $x_0 = 0, \rho \in (1/(2\alpha+1),1)$. Then $Q(h)=0, H_1 = (0,1], H_2 = [0,\infty)$, and we have
		\begin{align*}
		&\sup_{h_1 \in (0,1]}\inf_{h_2 \in [0,\infty)} \bigg[\frac{\mathbb{B}(h_2)-\mathbb{B}(-h_1)}{h_1+h_2}\bigg]\leq t\\
		\Leftrightarrow\quad& \forall h_1 \in [-1,0),\exists h_2 \in [0,\infty),  -\mathbb{B}(h_2)+th_2\geq -\mathbb{B}(h_1)+th_1\\
		\Leftrightarrow\quad & \argmax_{u \in [-1,\infty)}\big(-\mathbb{B}(u)-(-t)u\big)\geq 0\\
		\Leftrightarrow\quad & \mathrm{LCM}_{[-1,\infty)}\big(-\mathbb{B}(u)\big)'(0)\geq -t \\
		\Leftrightarrow\quad & \mathrm{GCM}_{[-1,\infty)}\big(\mathbb{B}(u)\big)'(0)\leq t,
		\end{align*}
		which takes a similar form as the limiting distribution found in \cite[Theorem 3.1-(i)]{kulikov2006behavior} (up to a shift and a re-centering of the Brownian motion).
		\item Let $x_0 = 0, \rho = 1/(2\alpha+1)$. Then $Q(h) = f_0^{(\alpha)}(0)\sum_{\ell=1}^{\alpha} \frac{h^{\ell+1}}{(\alpha-\ell)!(\ell+1)!} = \frac{f_0^{(\alpha)}(0)}{(\alpha+1)!} \big((1+h)^{\alpha+1}-1-(\alpha+1)h\big), H_1 = (0,1], H_2 = [0,\infty)$, and we have
		\begin{align*}
		&\sup_{h_1 \in (0,1]}\inf_{h_2 \in [0,\infty)} \bigg[\frac{\mathbb{B}(h_2)-\mathbb{B}(-h_1)}{h_1+h_2}+\frac{Q(h_2)-Q(-h_1)}{h_1+h_2}\bigg]\leq t\\
		\Leftrightarrow\quad& \forall h_1 \in [-1,0),\exists h_2 \in [0,\infty),\\
		&\quad\quad-\mathbb{B}(h_2)-Q(h_2)+th_2\geq -\mathbb{B}(h_1)-Q(h_1)+th_1\\
		\Leftrightarrow\quad & \argmax_{u \in [-1,\infty)}\bigg(-\mathbb{B}(u)- \frac{f_0^{(\alpha)}(0)}{(\alpha+1)!} (1+u)^{\alpha+1} -\bigg(-\frac{f_0^{(\alpha)}(0)}{\alpha!} -t\bigg)u\bigg)\geq 0\\
		\Leftrightarrow\quad & \mathrm{LCM}_{[-1,\infty)}\bigg(-\mathbb{B}(u)- \frac{f_0^{(\alpha)}(0)}{(\alpha+1)!} (1+u)^{\alpha+1}\bigg)'(0)\geq -\frac{f_0^{(\alpha)}(0)}{\alpha!} -t \\
		\Leftrightarrow\quad & \mathrm{GCM}_{[-1,\infty)}\bigg(\mathbb{B}(u)+ \frac{f_0^{(\alpha)}(0)}{(\alpha+1)!} (1+u)^{\alpha+1}\bigg)'(0)\leq \frac{f_0^{(\alpha)}(0)}{\alpha!}+t,
		\end{align*}
		which resembles the limiting distribution found in \cite[Theorem 3.1-(ii)]{kulikov2006behavior} (again up to a shift and a re-centering of the Brownian motion).
	\end{itemize}
\end{remark}

The Berry-Esseen bound in Theorem \ref{thm:berry_esseen_isoreg} matches the oracle rate in Proposition \ref{prop:oracle_CLT} up to multiplicative logarithmic factors, and the normal distribution  therein is replaced by the generalized Chernoff distribution. In this sense, the isotonic least squares estimator $\hat{f}_n$ mimics the behavior of the oracle local average estimator in Proposition \ref{prop:oracle_CLT} in terms of the speed of distributional approximation to the limiting random variable.

Theorem \ref{thm:berry_esseen_isoreg} immediately yields the following Berry-Esseen bound in a canonical setting for isotonic regression.

\begin{corollary}[Berry-Esseen bound for canonical case]
	\label{cor:berry_esseen_canonical}
	Let $x_0 \in (0,1)$ and $\xi_i$'s be as in Theorem \ref{thm:berry_esseen_isoreg}. Suppose Assumption \ref{assump:local_smoothness} holds with $\alpha=1, \alpha^\ast \geq 2$, i.e. $f_0$ is locally $C^2$ at $x_0$ with $f_0'(x_0)>0$,  and that $\{X_i: i=1,\dots,n \}$ are globally equally spaced design points on $[0,1]$ or i.i.d. $\mathrm{Unif}[0,1]$ random variables independent of $\xi_i$'s. Then
	\begin{align*}
	&\sup_{t \in \R} \bigg\lvert \Prob\big((n/\sigma^2)^{1/3}\big(\hat{f}_n(x_0)-f_0(x_0)\big)\leq t\big)\\
	&\quad\quad\quad\quad\quad\quad\quad - \Prob \big(\big(f_0'(x_0)/2\big)^{1/3} \cdot \mathbb{D}_{1} \leq t\big) \bigg\rvert \leq K\cdot n^{-1/3}(\log n)^{\zeta_{1,\alpha^\ast,\infty}}.
	% \frac{K\cdot (\log n)^{\zeta_{1,\alpha^\ast,\infty}}}{n^{1/3}}. 
	\end{align*}
	The constant $K>0$ does not depend on $n$.
\end{corollary}
\begin{proof}
	Apply Theorem \ref{thm:berry_esseen_isoreg} with $\Lambda_0=1$ and $\alpha=1$ with arbitrary $\alpha^\ast$. Here $\beta = \infty$ in the random design case.
\end{proof}

\begin{remark}[Simulation experiment]
	We present a simulation result (cf. Figure \ref{fig:sim1}) in support of the  $n^{-1/3}$ rate (modulo logarithmic factors) in the Berry-Esseen bound in Corollary \ref{cor:berry_esseen_canonical}. In this simulation we consider $f_1(x) = 2x^2$ and $f_2(x) = 4x^4$, and the fixed design as in Corollary \ref{cor:berry_esseen_canonical}. We use i.i.d. Rademacher errors, i.e. $\Prob(\xi_i=\pm 1)=1/2$. The choice of error distribution is motivated by the fact that the worst-case Berry-Esseen bound for the central limit theorem of sample mean is attained by the Rademacher mean. Under this setup, we have 
	\begin{align*}
	n^{1/3}\big(\hat{f}_n(1/2)-f_i(1/2)\big)\limd \mathbb{D}_1,\quad i=1,2.
	\end{align*}
	By (limiting) symmetric considerations, we only compute the values of $\Prob\big(n^{1/3}\big(\hat{f}_n(1/2)-f_i(1/2)\big)\leq t\big)$ for $t \in \{\ell/5: 1\leq \ell\leq 10\}$ based on  $5\times 10^5$ simulations. The values of $\{\Prob(\mathbb{D}_1\leq t): t \in \{\ell/5: 1\leq \ell\leq 10\}\}$ are taken from \cite{groeneboom2001computing} (note that our $\mathbb{D}_1=2Z$ in their notation). The simulations provide overwhelming evidence that the Berry-Esseen bound in Corollary \ref{cor:berry_esseen_canonical} is sharp modulo logarithmic factors.
\end{remark}

\begin{figure}
	\centering
	\includegraphics[width=1\textwidth]{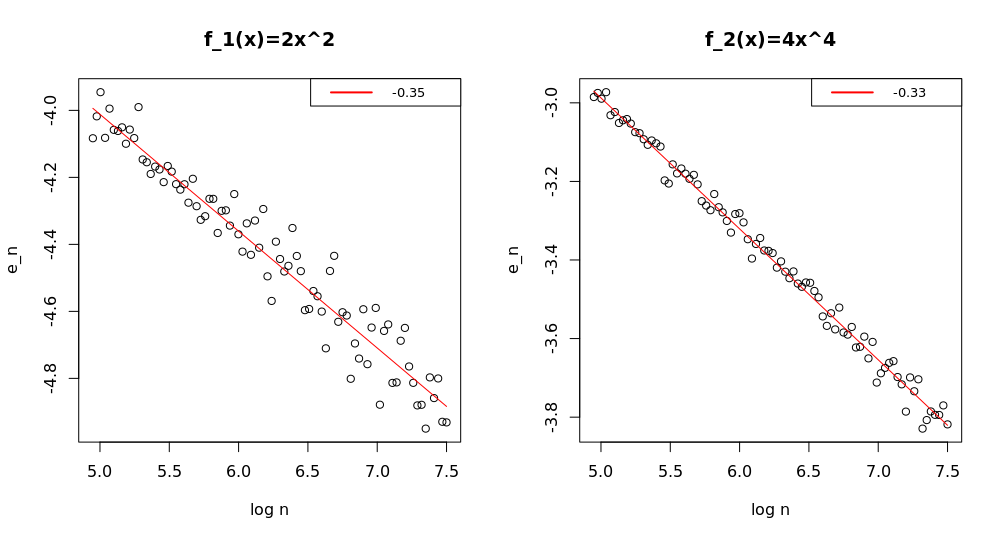}
	\caption{$E_n \equiv \max_{t \in \{\ell/5: 1\leq \ell\leq 10\}} \abs{\Prob^\ast \big(n^{1/3}\big(\hat{f}_n(1/2)-f_i(1/2)\big)\leq t\big)-\Prob\big(\mathbb{D}_1\leq t) }$ and $e_n\equiv \log E_n$, where $\Prob^\ast$ denotes the empirical average  based on $5\times 10^5$ simulations. The number in the legend of the figure indicates the slope for linear regression fit of $(\log n,e_n)$.  }
	\label{fig:sim1}
\end{figure}

Another interesting consequence of Theorem \ref{thm:berry_esseen_isoreg} is the following: If $f_0$ is flat (i.e. equals a constant), then a parametric rate (up to logarithmic factors) in the Berry-Esseen bound is possible. We formalize this result as follows.

\begin{corollary}[Berry-Esseen bound for constant function]
	\label{cor:berry_esseen_constant}
	Let $x_0 \in (0,1)$ and $\xi_i$'s be as in Theorem \ref{thm:berry_esseen_isoreg}. Suppose $f_0\equiv c$ for some constant $c \in \R$,  and that $\{X_i: i =1,\dots,n\}$ are globally equally spaced design points on $[0,1]$ or i.i.d. $\mathrm{Unif}[0,1]$ random variables independent of $\xi_i$'s. Then
	\begin{align*}
	&\sup_{t \in \R} \bigg\lvert\Prob\left((n/\sigma^2)^{1/2}\big(\hat{f}_n(x_0)-f_0(x_0)\right)\leq t\big)\\
	&\quad\quad - \Prob \bigg(\sup_{h_1 \in (0,x_0]}\inf_{h_2 \in [0,1-x_0]} \mathbb{B}_{\sigma,1,0}(h_1,h_2) \leq t\bigg) \bigg\rvert \leq K\cdot n^{-1/2}(\log n)^{\zeta_{\infty,\infty,\infty}}.
	% \frac{K\cdot (\log n)^{\zeta_{\infty,\infty,\infty}}}{n^{1/2}}. 
	\end{align*}
	The constant $K>0$ does not depend on $n$. 
	%Here $\zeta_{\infty,\infty,\infty}=7/2$ in both fixed and random designs.
\end{corollary}
\begin{proof}
	Apply Theorem \ref{thm:berry_esseen_isoreg} with $\Lambda_0=1$ and $\alpha=\infty$ (so $\alpha^\ast=\infty$). Here $\beta =\infty$ in the random design. 
\end{proof}

\begin{remark}[Boundary case]
	When $x_0=0$, the range of $\rho$ in Theorem \ref{thm:berry_esseen_isoreg} is restricted to $(0,1/(2\alpha+1)]\cup [2/3,1)$. The main reason for this restriction is an abrupt phase transition in the limit distribution theory. For instance, consider $f_0\equiv 0$ (i.e. $\alpha=\infty$) with noise level $\sigma=1$. If $x_0 \in (0,1)$, $\sqrt{n}\hat{f}_n(x_0)$ converges in distribution to 
	\begin{align*}
	Y_0\equiv \sup_{h_1 \in (0,x_0]}\inf_{h_2 \in [0,1-x_0]} \mathbb{B}_{1,1,0}(h_1,h_2),
	\end{align*}
	with a Berry-Esseen bound on the order of $O(n^{-1/2})$ up to logarithmic factors. However, as soon as $x_n \to 0$, $\sqrt{n x_n}\hat{f}_n(x_n)$ converges in distribution to a completely different limiting random variable 
	\begin{align*}
	Y_1\equiv \sup_{h_1 \in (0,1]}\inf_{h_2 \in [0,\infty)} \mathbb{B}_{1,1,0}(h_1,h_2),
	\end{align*}
	in the sense that $Y_1\leq 0$ a.s. It is therefore natural to expect that for $x_n$ converging slowly enough, a near $O((nx_n)^{-1/2})$ rate cannot be attained in the Berry-Esseen bound due to the inherent difference between $Y_0$ and $Y_1$. Our Theorem \ref{thm:berry_esseen_isoreg} here guarantees a near $O((nx_n)^{-1/2})$ rate for the Berry-Esseen bound when $x_n=n^{-\rho}$ converges fast enough with $\rho \in [2/3,1)$.
\end{remark}

\subsection{Proof sketch}
\label{sec: proof sketch}

In this subsection, we give a sketch of proof for Theorem \ref{thm:berry_esseen_isoreg} in the canonical case (\ref{eqn:isotonic_Berry_Esseen}), where $X_i=i/n, i=1,\ldots,n$ are globally equally spaced fixed design points on $[0,1]$, $f_0$ is locally $C^2$ at $x_0 \in (0,1)$ with $f_0'(x_0)>0$, and the errors $\xi_i$'s are i.i.d. mean zero with $\E e^{\theta \xi_1}<\infty$ for $\theta$ in a neighborhood of the origin. For simplicity of discussion, we assume that $\E \xi_1^2=1$. We reparametrize the max-min formula (\ref{def:isotonic_maxmin}) by
\begin{align}
\label{def:isotonic_maxmin_repara}
\hat{f}_n(x_0) &= \max_{h_1> 0}\min_{h_2\geq 0} \bar{Y}|_{[x_0-h_1 n^{-1/3},x_0+h_2 n^{-1/3}]} \\
&\equiv \bar{Y}|_{[x_0-h_1^\ast n^{-1/3},x_0+h_2^\ast n^{-1/3}]}.\nonumber
\end{align}
The first step in the proof of (\ref{eqn:isotonic_Berry_Esseen}) is to localize the isotonic LSE $\hat{f}_n$ in the sense that for some slowly growing sequences $\{t_n\}, \{\tau_n\}$, 
\begin{itemize}
	\item $\abs{n^{1/3}(\hat{f}_n(x_0)-f_0(x_0))}\leq t_n$ and
	\item $\abs{h_1^\ast}\vee \abs{h_2^\ast}\leq \tau_n$
\end{itemize}
hold with overwhelming probability. In fact, we may take $t_n,\tau_n$ on the order of $\sqrt{\log n}$ for this purpose; see Lemmas \ref{lem:large_deviation_0} and \ref{lem:large_deviation} ahead. 

Next, note that by the Kolm\'os-Major-Tusn\'ady strong embedding theorem (see Lemma \ref{lem:strong_embedding} ahead), with overwhelming probability,  
\begin{align*}
&\bar{\xi}|_{[x_0-h_1 n^{-1/3},x_0+h_2 n^{-1/3}] } \approx \frac{\sum_{X_i \in [x_0-h_1 n^{-1/3},x_0+h_2 n^{-1/3}]} \xi_i}{ (h_1+h_2) n^{2/3}}\\
&\quad \approx \frac{ \mathbb{B}(h_2n^{2/3})+\mathbb{B}(-h_1n^{2/3})}{ (h_1+h_2) n^{2/3}} \stackrel{d}{=} n^{-1/3}\cdot \frac{ \mathbb{B}(h_2)+\mathbb{B}(-h_1) }{h_1+h_2 },
\end{align*}
and by a calculation of the bias via Taylor expansion (see Lemma \ref{lem:bias_calculation} ahead),
\begin{align*}
\bar{f_0}|_{[x_0-h_1 n^{-1/3},x_0+h_2 n^{-1/3}] } -f_0(x_0) \approx n^{-1/3}\bigg[ \frac{f_0'(x_0)}{2}\cdot \frac{h_2^2-h_1^2}{h_1+h_2}+ R_n\bigg],
\end{align*}
where $R_n$ is roughly of order $n^{-1/3}$. Now using the alternative max-min formula (\ref{def:isotonic_maxmin_repara}), with $\gamma_0\equiv f_0'(x_0)/2$, uniformly in  $\abs{t}\leq t_n$, 
\begin{align*}
&\Prob\big(n^{1/3}\big(\hat{f}_n(x_0)-f_0(x_0)\big)\leq t\big) \\
& \approx \Prob \bigg(\max_{0< h_1\leq \tau_n}\min_{0\leq h_2\leq \tau_n}\big( \mathbb{B}(h_2)+\mathbb{B}(-h_1)+ \gamma_0(h_2^2-h_1^2)-t(h_1+h_2)\big)\leq \tilde{O}(R_n)\bigg),
\end{align*}
where $\tilde{O}(R_n)$ stands for a term of order $R_n$ up to poly-logarithmic factors. Let $T_{n,1}\equiv \max_{0< h_1\leq \tau_n} \big(\mathbb{B}(-h_1)-\gamma_0 h_1^2-th_1\big)$, $T_{n,2}\equiv \min_{0\leq h_2\leq \tau_n}\big(\mathbb{B}(h_2)+\gamma_0 h_2^2-th_2\big)$, and $\mathcal{L}_i(\epsilon)\equiv \sup_{u \in \R}\Prob\big(\abs{T_{n,i}-u}\leq \epsilon \big)$. 
Note that $T_{n,1}$ and $T_{n,2}$ are independent. 
Then the above display equals
\begin{align*}
&\Prob\big(T_{n,1}+T_{n,2}\leq \tilde{O}(R_n)\big)\\
&\leq \Prob\big(T_{n,1}+T_{n,2}\leq 0\big)+ \min_{i=1,2} \mathcal{L}_i \big(\tilde{O}(R_n)\big)\\
& = \Prob \bigg(\max_{0< h_1\leq \tau_n}\min_{0\leq h_2\leq \tau_n}\bigg( \frac{ \mathbb{B}(h_2)+\mathbb{B}(-h_1)}{h_1+h_2}+ \gamma_0\cdot \frac{h_2^2-h_1^2}{h_1+h_2}\bigg)\leq t\bigg)+ \min_{i=1,2} \mathcal{L}_i \big(\tilde{O}(R_n)\big)\\
&\approx  \Prob \bigg(\max_{ h_1 > 0}\min_{h_2\geq 0}\bigg( \frac{ \mathbb{B}(h_2)+\mathbb{B}(-h_1)}{h_1+h_2}+ \gamma_0\cdot \frac{h_2^2-h_1^2}{h_1+h_2}\bigg)\leq t\bigg)+ \min_{i=1,2} \mathcal{L}_i \big(\tilde{O}(R_n)\big).
\end{align*}
The last approximation follows from a similar localization property as in the first step for the isotonic LSE.  The first term in the above display is exactly the desired quantity
\begin{align*}
\Prob\big( (f_0'(x_0)/2)^{1/3}\cdot \mathbb{D}_1\leq t\big),
\end{align*}
so it remains to derive a sharp control of 
\begin{align*}
\min_{i=1,2} \mathcal{L}_i \big(\tilde{O}(R_n)\big).
\end{align*}
This is the \emph{anti-concentration} problem that will be studied in the next Section \ref{section:anti_concentration}. In particular, Theorem \ref{thm:anti_concentration} below shows that $\min_{i=1,2} \mathcal{L}_i \big(\tilde{O}(R_n)\big) = \tilde{O}(R_n) = \tilde{O}(n^{-1/3})$, by noting that $t_n \asymp \sqrt{\log n}$ and $\tau_n\asymp \sqrt{\log n}$ in the localization step (see also Remark \ref{rmk:anti_conc_grow_interval} below). This completes the proof of (\ref{eqn:isotonic_Berry_Esseen}) in the regime $\abs{t}\leq t_n$.  The regime $\abs{t}>t_n$ is already handled by the localization property of the isotonic LSE $\hat{f}_n$ in the first step.

\section{Anti-concentration}\label{section:anti_concentration}

\subsection{The anti-concentration problem}

As discussed in Section \ref{sec: proof sketch}, the proof of our main Berry-Esseen bounds in the canonical case builds  on the \emph{anti-concentration} of the random variable $T_n\equiv \sup_{0\leq h\leq \tau_n}\big(\mathbb{B}(h)+b h^{2}+th\big)$ for certain $\tau_n \uparrow \infty$, i.e.,  an estimate of $\mathcal{L}_{T_n}(\epsilon)\equiv \sup_{u \in \R} \Prob\big(\abs{T_n-u}\leq \epsilon)$, with certain uniformity in $t$. We note that \cite{groeneboom2010maximum} and \cite{jason2010maximum} derive analytical expressions of the density function of $\sup_{h \geq 0} (\mathbb{B}(h) - \gamma h^{2})$ for $\gamma > 0$, but their results are not applicable to our problem since we need anti-concentration bounds on the supremum of a Brownian motion with a linear-quadratic drift on a compact interval. In addition, the proof for the general case in Theorem \ref{thm:berry_esseen_isoreg} requires, as one of the key technical results, uniform anti-concentration bounds on the supremum of a Brownian motion with a general polynomial drift. 
Theorem \ref{thm:anti_concentration} below derives such anti-concentration bounds in a more general context for Brownian motion with a Lipschitz drift. 

\begin{theorem}
	[Anti-concentration of sup of BM plus a Lipschitz drift]
	\label{thm:anti_concentration}
	Let $\mathbb{B}$ be a  standard Brownian motion starting from $0$. Let $P: [0,1]\to \R$ be $b$-Lipschitz in that $\abs{P(h_1)-P(h_2)}\leq b\abs{h_1-h_2}$ for all $h_1,h_2 \in [0,1]$, and 
	\begin{align}\label{def:T}
	T\equiv \sup_{0\leq h\leq 1}\big(\mathbb{B}(h)+P(h)\big).
	\end{align}
	Then the following anti-concentration holds: there exists some absolute constant $K>0$ such that for any $\epsilon>0$,
	\begin{align}\label{ineq:anti_concentration_0}
	\sup_{u \in \R}\Prob\big(\abs{T-u}\leq \epsilon\big)\leq K \epsilon \mathscr{L}_{\bar{b}}(\epsilon),
	\end{align}
	where $
	\mathscr{L}_{\bar{b}}(\epsilon) \equiv \bar{b} \log_+(\bar{b}/\epsilon) \big(1\vee \bar{b}\epsilon \log_+^{-1}(1/\epsilon)\big)\big(\bar{b}\vee \log_+(\bar{b}/\epsilon)\big)$. 
	Here $\log_+(\cdot)\equiv 1\vee \log(\cdot)$ and $\bar{b}\equiv 1\vee b$. 
\end{theorem}
\begin{proof}
	See the next subsection.
\end{proof}

\begin{remark}
	From log-concavity of Gaussian measures, the distribution of $T$ is absolutely continuous on $(r_{0},\infty)$, where $r_{0}$ is the left end point of the support of $T$; see, e.g., \cite[Theorem 11.1]{davydov1998local}. This also shows that the density of $T$ is bounded on $(r,\infty)$ for any $r > r_{0}$.  This theorem, however, does not guarantee global boundedness of the density of $T$, and thus does not lead to a quantitative anti-concentration inequality of the form (\ref{ineq:anti_concentration_0}) (since the variance of the process $h \mapsto \mathbb{B}(h) + P(h)$ attains zero at $h=0$, \cite[Proposition 11.4]{davydov1998local} is also not applicable). 
	Indeed, as we will discuss in Remark \ref{rmk:anti_conc_grow_interval} ahead, in our application, we need to know how the drift term $P(\cdot)$ quantitatively affects the anti-concentration inequality, and such quantitative information {does not follow} from \cite[Theorem 11.1]{davydov1998local} or its proof. 
\end{remark}

\begin{remark}[Case with uniformly bounded coefficients]
	If $\bar{b}\lesssim 1$, then (\ref{ineq:anti_concentration_0}) in Theorem \ref{thm:anti_concentration} reduces to 
	\begin{align*}
	\sup_{u \in \R}\Prob\big(\abs{T-u}\leq \epsilon\big)\leq K\epsilon \log^{2}_+(1/\epsilon).
	\end{align*}
	The above bound holds for any Lipschitz function $P$. If $P=0$,  then by the reflection principle  for a Brownian motion, $T = \sup_{0 \le h \le 1} \mathbb{B}(h) \stackrel{d}{=} |Z|$ for $Z \sim \mathcal{N}(0,1)$, so that the logarithmic factor in the above display can be removed. 
\end{remark}

\begin{remark}[Suprema over slowly expanding intervals]\label{rmk:anti_conc_grow_interval}
	In the proof of Theorem \ref{thm:berry_esseen_isoreg}, we will need anti-concentration for random variables of the form $T_n = \sup_{0\leq h\leq \tau_n} \big(\mathbb{B}(h)+ \sum_{\ell=1}^{\alpha} b_\ell' h^{\ell+1}-th\big)$, where $\tau_n \uparrow \infty$ is some slowly growing sequence, and $\bar{b}'\equiv 1\vee \max_{1\leq \ell \leq \alpha} b_\ell' $ (typically) does not grow with $n$. Note that
	\begin{align*}
	T_n  &\equald \tau_n^{1/2} \sup_{0\leq h'\leq 1} \bigg(\mathbb{B}(h')+ \sum_{\ell=1}^{\alpha} b_\ell' \tau_n^{\ell+1/2} (h')^{\ell+1}-t \tau_n^{1/2} h'\bigg)\equiv \tau_n^{1/2}\cdot T_n'.
	\end{align*} 
	Hence uniformly in $\abs{t}\leq t_n$, where $t_n$ is potentially a slowly growing sequence, we have by Theorem \ref{thm:anti_concentration}
	\begin{align*}
	&\sup_{u \in \R}\Prob\big(\abs{T_n-u}\leq \epsilon\big) = \sup_{u \in \R}\Prob\big(\abs{T_n'-u}\leq \epsilon/\tau_n^{1/2}\big)\\
	&\leq K_{\bar{b}'} \cdot  \epsilon(\tau_n^\alpha \vee t_n) \log_+\bigg( \frac{\tau_n\big(\tau_n^\alpha \vee t_n\big)}{\epsilon} \bigg)\\
	&\quad\quad\quad\quad\times \bigg(1\bigvee \frac{(\tau_n^\alpha \vee t_n) \epsilon }{\log_+(\tau_n^{1/2}/\epsilon)}\bigg)\bigg(\tau_n^{1/2}(\tau_n^\alpha \vee t_n) \bigvee \log_+\bigg(\frac{\tau_n(\tau_n^\alpha \vee t_n)}{\epsilon}\bigg)\bigg).
	\end{align*}
	For the canonical case $\alpha =1$, we will take $\tau_n\asymp t_n \asymp \sqrt{\log n}$ as described in Section \ref{sec: proof sketch}, and $\epsilon=\epsilon_n$ such that $\log_+(1/\epsilon_n)\asymp \log n$. Then the above bound reduces to
	\begin{align*}
	\sup_{u \in \R}\Prob\big(\abs{T_n-u}\leq \epsilon_n\big) \leq K_{\bar{b}'}\cdot \epsilon_n\cdot \log^{5/2} n.
	\end{align*}
	For a general $\alpha$, we will typically take $\tau_n^\alpha \asymp t_n \asymp \sqrt{\log n}$, so the above bound still holds.
\end{remark}

\begin{remark}[Comparison with small ball problem]
	The anti-concentration problem  considered in Theorem \ref{thm:anti_concentration} is qualitatively different from the small ball problem, cf. \cite{li2001gaussian}. For instance, \cite[Theorem 3.1]{li2001gaussian} shows that as $\epsilon \downarrow 0$,
	\begin{align*}
	\Prob\bigg(\sup_{0\leq h\leq 1}\bigabs{\mathbb{B}(h)+P(h)}\leq \epsilon\bigg)\sim e^{-\pnorm{P'}{L_2}^2/2}\Prob\bigg(\sup_{0\leq h\leq 1}\abs{\mathbb{B}(h)}\leq \epsilon\bigg).
	\end{align*}
	Using the well-known fact that $\log\Prob\big(\sup_{0\leq h\leq 1}\abs{\mathbb{B}(h)}\leq \epsilon\big) \sim -(\pi^2/8)\epsilon^{-2}$ (cf. \cite[Theorem 6.3]{li2001gaussian}), we have
	\begin{align*}
	\epsilon^2 \log \Prob\bigg(\sup_{0\leq h\leq 1}\bigabs{\mathbb{B}(h)+P(h)}\leq \epsilon\bigg)\sim -\pi^2/8,
	\end{align*}
	as $\epsilon \downarrow 0$, an estimate exhibiting a completely different behavior compared with the anti-concentration bound in Theorem \ref{thm:anti_concentration}.
\end{remark}

\begin{remark}[Anti-concentration inequalities]
	The anti-concentration inequalities that are in similar in nature to Theorem \ref{thm:anti_concentration} play a pivotal role in establishing  Berry-Esseen bounds for central limit theorems on the class of convex sets in the multivariate setting \cite{bentkus2003dependence} and on  hyperrectangles in the high-dimensional setting  \cite{chernozukov2014central,chernozhukov2014comparison,chernozhukov2014gaussian}. In the latter problem, one main ingredient is the anti-concentration for the maximum  of  jointly Gaussian random variables with uniformly positive variance; cf. Nazarov's inequality \cite{nazarov2003maximal,chernozhukov2017detailed}.
\end{remark}

\subsection{Proof of Theorem \ref{thm:anti_concentration}}

The proof of Theorem \ref{thm:anti_concentration} relies on several technical results. One is the anti-concentration lemma (\cite[Lemma 2.2]{chernozhukov2015empirical}) for the supremum of a non-centered Gaussian process with \emph{uniformly positive variance}. 
\begin{lemma}
	\label{lem:anti_conc_CCK}
	Let $\{X(t):t \in T\}$ be a possibly noncentered tight Gaussian random variable in $\ell^\infty(T)$. Let $\underline{\sigma}^2\equiv \inf_{t \in T} \mathrm{Var}(X(t))$. Let $d:T\times T\to \R_{\geq 0}$ be a pseudometric defined by $d^2(s,t)\equiv \E \big(X(s)-X(t)\big)^2$ for $s,t \in T$. Then for any $\epsilon>0$, 
	\begin{align*}
	&\sup_{u \in \R}\Prob\bigg(\biggabs{\sup_{t\in T}X(t)-u}\leq \epsilon\bigg)\\
	&\leq \inf_{\delta,r>0}\bigg[\frac{2}{\underline{\sigma}}\big(\epsilon+ \varpi_X(\delta)+r\delta \big)  \big(\sqrt{2\log \mathcal{N}(\delta,T,d)}+2\big)+e^{-r^2/2}\bigg],
	\end{align*}
	where $\varpi_X(\delta)\equiv \E\sup_{s,t \in T, d(s,t)\leq \delta} \abs{X(s)-X(t)}$.
\end{lemma}

Unfortunately, we can not directly apply the above anti-concentration bound to our problem since  the supremum in (\ref{def:T}) necessarily involves the Brownian motion at small times, whose the variance can be arbitrarily close to zero. In the proof below we will use a carefully designed blocking argument to compensate the large estimate due to small variance incurred by Lemma \ref{lem:anti_conc_CCK}, with small estimate for the anti-concentration of the supremum of a Brownian motion with a linear drift. To this end, we will use the following lemma, the proof of which can be found in the appendix.

\begin{lemma}[Density of sup of BM with linear drift]\label{lem:BM_linear_drift}
	Let $\mathbb{B}$ be a standard Brownian motion starting from $0$, and $\mu \in \R$. Let $M_{\mu} \equiv \sup_{0 \leq h \le 1} (\mathbb{B}(h) + \mu h) \equiv \sup_{0\leq h\leq 1} \mathbb{B}_\mu(h)$. Then the Lebesgue density of $M_{\mu}$, denoted by $p_{M_{\mu}}$, is given by
	\begin{align}\label{eqn:density_drifted_BM}
	p_{M_{\mu}}(y)= \big[2\varphi(y-\mu)-2\mu e^{2\mu y}\big(1-\Phi\big(y+\mu)\big)\big] \bm{1}_{y\geq 0},
	\end{align}
	where $\varphi(\cdot) = (2\pi)^{-1/2} e^{-(\cdot)^2/2}$ and $\Phi(\cdot)$ are the probability density function and cumulative distribution function of the standard normal distribution, respectively. Consequently, $\pnorm{p_{M_{\mu}}}{\infty}\lesssim (\mu \vee 1)$.
\end{lemma}

%We are now in position to prove Theorem \ref{thm:anti_concentration}. 

\begin{proof}[Proof of Theorem \ref{thm:anti_concentration}]
	Let $N \equiv  \max\big\{\floor{K_1 \bar{b}^2\epsilon^{-2}\log_+(\bar{b}/\epsilon)}+1,4\}$ for some constant $K_1>0$ to be chosen later. Let $h_\ell\equiv \ell/N$ for $1\leq \ell \leq N$. Assume without loss of generality $\log_2 (N/2+1) \in \N$. Let $L(h)\equiv \mathbb{B}(h)+P(h)$. For $1\leq j\leq \log_2 (N/2+1)$, let $\Omega_j\equiv \{\argmax_{1\leq \ell \leq N}\big(\mathbb{B}(h_\ell)+P(h_\ell)\big) \in \{2^{j-1}, \ldots,2^j-1\}\}$. Then
	\begin{align*}
	&\Prob\big(u-\epsilon\leq T\leq u+\epsilon\big)\\
	&\leq \Prob\bigg(\max_{1\leq \ell\leq N/2} L(h_\ell) \in [u-2\epsilon,u+2\epsilon]\bigg)+ \Prob\bigg(\sup_{ \substack{h_1,h_2 \in [0,1],\\\abs{h_1-h_2}\leq 1/N}} \abs{L(h_1)-L(h_2)}>\epsilon\bigg)\\
	&\quad\quad +\Prob\bigg( \sup_{h \in [1/2,1]} L(h)\in[u-\epsilon,u+\epsilon]\bigg)\equiv (I)+(II)+(III).
	\end{align*}
	
	We first handle relatively easy terms $(II)$ and $(III)$.
	
	For $(II)$, note that for some absolute constant $K_2>1$, we may choose $K_1>800$ large enough such that
	\begin{align*}
	\E \sup_{\abs{h_1-h_2}\leq 1/N}\abs{L(h_1)-L(h_2)}&\leq \E\sup_{\abs{h_1-h_2}\leq 1/N}\abs{\mathbb{B}(h_1)-\mathbb{B}(h_2)} +\abs{P(h_1)-P(h_2)}\\
	&\leq K_2\big[\sqrt{\log N/N}+\bar{b}/N\big]\leq \epsilon/10.
	\end{align*}
	The first inequality in the above display uses entropy integral (cf. Lemma \ref{lem:dudley_entropy_integral}) to evaluate the expected supremum. Since
	\begin{align*}
	&\sup_{\abs{h_1-h_2}\leq 1/N} \mathrm{Var}(L(h_1)-L(h_2))\leq 1/N \leq \frac{\epsilon^2}{\bar{b}^2 K_1\log_+(\bar{b}/\epsilon)}\leq \frac{\epsilon^2}{K_1\log_+(1/\epsilon)},
	\end{align*}
	it follows by the Gaussian concentration (cf. Lemma \ref{lem:Gaussian_concentration}) that
	\begin{align*}
	(II)&\leq \Prob\bigg(\sup_{\abs{h_1-h_2}\leq 1/N}\abs{\mathbb{B}(h_1)-\mathbb{B}(h_2)}\\
	&\quad\quad-\E\sup_{\abs{h_1-h_2}\leq 1/N}\abs{\mathbb{B}(h_1)-\mathbb{B}(h_2)}>\epsilon-\bar{b}/N-\epsilon/10\bigg)\\
	&\leq \Prob\bigg(\sup_{\abs{h_1-h_2}\leq 1/N}\abs{\mathbb{B}(h_1)-\mathbb{B}(h_2)}-\E\sup_{\abs{h_1-h_2}\leq 1/N}\abs{\mathbb{B}(h_1)-\mathbb{B}(h_2)}>\epsilon/2\bigg)\\
	&\leq \exp\bigg(-\frac{\epsilon^2/4}{2 \epsilon^2/K_1\log_+(1/\epsilon)}\bigg) = \exp\big(-(K_1/8)\log_+(1/\epsilon)\big)\leq \epsilon^{100},
	\end{align*}
	by choosing $K_1>800$.
	
	For $(III)$, as the minimum standard deviation of $L(h)$ for $h \in [1/2,1]$ is strictly bounded from below by $1/\sqrt{2}$, we may use the anti-concentration inequality for non-centered Gaussian process as in Lemma \ref{lem:anti_conc_CCK}:
	\begin{align*}
	(III)&\lesssim \inf_{r,\delta>0}\bigg[\bigg(\epsilon+\E\sup_{ \substack{1/2\leq h_i\leq 1: i=1,2\\ d(h_1,h_2)\leq \delta }} \abs{L(h_1)-L(h_2)}+r\delta \bigg)\\
	&\quad\quad\quad \quad\quad \quad\quad \times \big(1\vee \sqrt{ \log \mathcal{N}(\delta,[1/2,1],d)}\big)+e^{-r^2/2}\bigg],
	\end{align*}
	where 
	\begin{align*}
	d^2(h_1,h_2)&=\E \big(L(h_1)-L(h_2)\big)^2 = \E\big(\mathbb{B}(h_1)-\mathbb{B}(h_2)\big)^2+ (P(h_1)-P(h_2))^2\\
	&\leq \abs{h_1-h_2}+\bar{b}^2 (h_1-h_2)^2\leq K_3\bar{b}^2\cdot \abs{h_1-h_2}.
	\end{align*}
	Hence by Lemma \ref{lem:dudley_entropy_integral},
	\begin{align*}
	&\E\sup_{ \substack{1/2\leq h_i\leq 1: i=1,2\\ d(h_1,h_2)\leq \delta }} \abs{L(h_1)-L(h_2)} \lesssim \big[\delta\sqrt{\log(1/\delta)}+\bar{b}\delta^2\big],\\
	&\log \mathcal{N}(\delta,[1/2,1],d) \leq \log \mathcal{N}\bigg(\frac{\delta^2}{K_3\bar{b}^2}, [1/2,1],\abs{\cdot}\bigg) \lesssim \log_+\bigg(\frac{\bar{b}}{\delta}\bigg).
	\end{align*}
	Collecting the estimates, by choosing $r\equiv 2\log_+^{1/2}(1/\epsilon)$ and $\delta\equiv \epsilon/\sqrt{\log_+(1/\epsilon)} $, we arrive at
	\begin{align*}
	(III)&\lesssim \inf_{r,\delta>0} \bigg[\big(\epsilon+ \delta\sqrt{\log(1/\delta)}+\bar{b}\delta^2+r\delta\big)\cdot \log_+^{1/2}\big(\bar{b}/\delta\big) +e^{-r^2/2}\bigg]\\
	&\lesssim  \epsilon\big(1 \vee \bar{b} \epsilon\log^{-1}_+(1/\epsilon)\big) \log_+^{1/2}(\bar{b}/\epsilon). 
	\end{align*}

	Finally we handle the most difficult term $(I)$. 
	For each $1\leq j\leq \log_2(N/2+1)$, let $h_{\ell_{j}^\ast} \in \{ h_{\ell} : 2^{j-1} \leq \ell < 2^{j} \}$ be defined by $L(h_{\ell_{j}^\ast}) =  \max_{2^{j-1}\leq \ell< 2^{j}} L(h_\ell)$. 
	By blocking through the events $\{\Omega_j:1\leq j\leq \log_2(N/2+1)\}$, we have
	\begin{align*}
	(I)&\leq \sum_{j=1}^{\log_2 (N/2+1)} \Prob\bigg(L(h_{\ell_{j}^\ast})  \in [u-2\epsilon,u+2\epsilon],  L(h_k)\leq u+2\epsilon,  2^j<\forall k\leq N\bigg)\\
	&\leq \sum_{j=1}^{\log_2 (N/2+1)} \Prob\bigg(L(h_{\ell_{j}^\ast})  \in [u-2\epsilon,u+2\epsilon],  L(h_k)-L(h_{\ell_{j}^\ast})\leq 4\epsilon,  2^j<\forall k\leq N\bigg)\\
	&\leq \sum_{j=1}^{\log_2 (N/2+1)} \Prob\bigg(\max_{2^{j-1}\leq \ell< 2^{j}} L(h_\ell) \in [u-2\epsilon,u+2\epsilon],\\
	&\quad \quad L(h_k)-L(h_{2^j})\leq 4\epsilon+\max_{2^{j-1}\leq \ell< 2^j} \abs{L(h_{2^{j}})-L(h_{\ell})},  2^j<\forall k\leq N\bigg).
	\end{align*}
	It is not hard to show that, using similar arguments above by calculating the first moment via the entropy integral (cf. Lemma \ref{lem:dudley_entropy_integral}) and Gaussian concentration (cf. Lemma \ref{lem:Gaussian_concentration}), for some large constant $K_4=K_4(m)>0$,
	\begin{align*}
	\Prob\bigg(\max_{2^{j-1}\leq \ell< 2^j} \abs{L(h_{2^{j}})-L(h_{\ell})}>K_4 \bigg[ \sqrt{\frac{2^j}{N}\log_+\bigg(\frac{N}{\epsilon 2^j}\bigg)}+ \frac{\bar{b} 2^j}{N}\bigg]\bigg)\leq \epsilon^{100}.
	\end{align*}
	Hence, we may continue bounding $(I)$ as follows:
	\begin{align*}
	(I)
	%&\leq \ldots\\
	&\leq \sum_{j=1}^{\log_2 (N/2+1)} \Prob\bigg( \max_{2^{j-1}\leq \ell< 2^{j}} L(h_\ell) \in [u-2\epsilon,u+2\epsilon],\\
	&\quad \quad L(h_k)-L(h_{2^j})\leq 4\epsilon+K_4 \bigg[ \sqrt{\frac{2^j}{N}\log_+\bigg(\frac{N}{\epsilon 2^j}\bigg)}+ \frac{\bar{b} 2^j}{N}\bigg],  2^j< \forall k \leq N\bigg)\\
	&\quad\quad + \log_2 (N/2+1)\cdot \epsilon^{100}\\
	& \leq  \sum_{j=1}^{\log_2 (N/2+1)} \bigg[\Prob\bigg(\max_{2^{j-1}\leq \ell< 2^{j}} L(h_\ell) \in [u-2\epsilon,u+2\epsilon] \bigg)\\
	&\quad\quad\times \Prob\bigg(\mathbb{B}(h_k-h_{2^j})\leq 4\epsilon+\bar{b}(h_k-h_{2^j})\\
	&\quad\quad\quad +K_4 \bigg[ \sqrt{\frac{2^j}{N}\log_+\bigg(\frac{N}{\epsilon 2^j}\bigg)}+ \frac{\bar{b} 2^j}{N}\bigg],  2^j<\forall k\leq N\bigg)\bigg] + \log_2 (N/2+1)\cdot \epsilon^{100}\\
	&\leq \sum_{j=1}^{\log_2 (N/2+1)} \bigg[\Prob\bigg(\max_{2^{j-1}\leq \ell< 2^{j}} L(h_\ell) \in [u-2\epsilon,u+2\epsilon] \bigg)\\
	& \quad\quad \times \Prob \bigg(\sup_{0\leq h\leq 1-2^j/N} \big(\mathbb{B}(h)-\bar{b} h\big)\leq K_5 \bigg[\epsilon+ \sqrt{\frac{2^j}{N}\log_+\bigg(\frac{N}{\epsilon 2^j}\bigg)}+ \frac{\bar{b} 2^j}{N}\bigg]  \bigg)\bigg] \\
	&\quad\quad + 2\log_2 (N/2+1)\cdot \epsilon^{100}\\
	&\equiv \sum_{j=1}^{\log_2(N/2+1)} \mathfrak{p}_{j,1}\cdot \mathfrak{p}_{j,2}  +2\log_2(N/2+1)\cdot \epsilon^{100}.
	\end{align*}
	In the last inequality we have expanded the supremum from the discrete set $\{1/N,2/N,\ldots,1-2^j/N\}$ to $0\leq h \leq 1-2^j/N$ at the cost of a larger constant $K_5$ and a larger residual probability estimate. 
	
	Now following similar calculations as in the derivation of $(III)$ using Lemma \ref{lem:anti_conc_CCK}, we have
	\begin{align*}
	\mathfrak{p}_{j,1}\lesssim  \frac{\epsilon\big(1 \vee \bar{b} \epsilon\log^{-1}_+(1/\epsilon)\big) \log_+^{1/2}(\bar{b}/\epsilon)}{\sqrt{2^j/N}}.
	\end{align*}
	On the other hand, as the supremum in $\mathfrak{p}_{j,2}$ can be restricted to $[0,1/4]$ (by noting that $\min_{1\leq j\leq \log_2(N/2+1)} (1-2^j/N)\geq 1-(N/2+1)/N = 1/2-1/N\geq 1/4$ for $N\geq 4$) and is always non-negative, we have
	\begin{align*}
	\mathfrak{p}_{j,2}&\leq \Prob \bigg(\sup_{0\leq h\leq 1/4} \big(\mathbb{B}(h)-\bar{b} h\big)\in \bigg[0, K_5 \bigg\{ \epsilon+ \sqrt{\frac{2^j}{N}\log_+\bigg(\frac{N}{\epsilon 2^j}\bigg)}+ \frac{\bar{b} 2^j}{N}\bigg\} \bigg]\bigg)\\
	& =\Prob \bigg(\sup_{0\leq h\leq 1} \big(\mathbb{B}(h)-(\bar{b}/2) h\big)\in \bigg[0, 2K_5 \bigg\{ \epsilon+ \sqrt{\frac{2^j}{N}\log_+\bigg(\frac{N}{\epsilon 2^j}\bigg)}+ \frac{\bar{b} 2^j}{N}\bigg\} \bigg]\bigg).
	\end{align*}
	By Lemma \ref{lem:BM_linear_drift}, the density of  $\sup_{0\leq h\leq 1} \big(\mathbb{B}(h)-(\bar{b}/2) h\big)$ is bounded by $\bar{b}$ up to a constant depending only on $m$, i.e.,  $\pnorm{p_{\sup_{0\leq h\leq 1} (\mathbb{B}(h)-(\bar{b}/2) h)}}{\infty}\lesssim \bar{b}$, and hence 
	\begin{align*}
	\mathfrak{p}_{j,2}\lesssim \bar{b} \bigg\{\epsilon+ \sqrt{\frac{2^j}{N}\log_+\bigg(\frac{N}{\epsilon 2^j}\bigg)}+ \frac{\bar{b} 2^j}{N}\bigg\}.
	\end{align*}
	Collecting the estimates, it follows that with
	\begin{align*}
	L_b(\epsilon)\equiv \epsilon \big(1\vee \bar{b} \epsilon\log^{-1}_+(1/\epsilon)\big) \log_+^{1/2}(\bar{b}/\epsilon),
	\end{align*}
	we have
	\begin{align*}
	(I)
	%&\leq \ldots\\
	&\lesssim \sum_{j=1}^{\log_2 (N/2+1)} \frac{L_b(\epsilon)}{\sqrt{2^j/N}}\cdot  \bar{b} \bigg\{\epsilon+ \sqrt{\frac{2^j}{N}\log_+\bigg(\frac{N}{\epsilon 2^j}\bigg)}+ \frac{\bar{b} 2^j}{N}\bigg\} +2\log_2 (N/2+1)\cdot \epsilon^{100} \\
	&\lesssim \bar{b}\epsilon L_b(\epsilon) \sum_{j=1}^{\log_2(N/2+1)} \frac{1}{\sqrt{2^j/N}} + \bar{b} L_b(\epsilon) \sum_{j=1}^{\log_2(N/2+1)} \log_+^{1/2}\big(N/\epsilon 2^j\big)\\
	&\quad\quad + \bar{b}^2 L_b(\epsilon) \sum_{j=1}^{\log_2(N/2+1)} (2^j/N)^{1/2} + \log_2 (N/2+1)\cdot \epsilon^{100}\\ 
	&\lesssim \bar{b}\sqrt{N}\epsilon L_b(\epsilon)+ \bar{b}L_b(\epsilon)\log_+^{3/2}(N/\epsilon) + \bar{b}^2 L_b(\epsilon) + \log_2 (N/2)\cdot \epsilon^{100}\\
	& \lesssim \bar{b}L_b(\epsilon) \log_+^{1/2}(\bar{b}/\epsilon) \big(\bar{b}\vee \log_+(\bar{b}/\epsilon)\big)+ \log_+(\bar{b}/\epsilon)\cdot \epsilon^{100}\\
	& \lesssim \bar{b}\epsilon \log_+(\bar{b}/\epsilon) \big(1\vee \bar{b}\epsilon \log_+^{-1}(1/\epsilon)\big)\big(\bar{b}\vee \log_+(\bar{b}/\epsilon)\big).
	\end{align*}
	The calculation above uses that $N\asymp \bar{b}^2\epsilon^{-2}\log_+(\bar{b}/\epsilon)$, as chosen in the beginning of the proof.
\end{proof}

\section{Localization}\label{section:localization}

\subsection{Preliminary estimates}
\label{sec:preliminary}

We make a few definitions:
\begin{itemize}
	\item Let $r_n$ be defined in (\ref{def:r_n_oracle}) and $\omega_n= (nr_n)^{-1/2}$ be as in (\ref{def:w_n_oracle}).
	\item Let $h_1^\ast, h_2^\ast>0$ be random variables defined by $\hat{f}_n(x_0)\equiv \bar{Y}|_{[x_0-h_1^\ast r_n, x_0+ h_2^\ast r_n]}$.
	\item Let $\Omega_n\equiv \{h_1^\ast \vee h_2^\ast \leq \tau_n\}$ for some $\tau_n>0$ to be specified below.
	\item  Let $\tilde{h}_1,\tilde{h}_2>0$ be random variables defined by 
	\begin{align}\label{def:h_tilde}
	&\sup_{h_1 \in H_1}\inf_{h_2 \in H_2} \mathbb{B}_{\sigma,\Lambda_0,Q}(h_1,h_2)\equiv \mathbb{B}_{\sigma,\Lambda_0,Q}(\tilde{h}_1,\tilde{h}_2),
	\end{align}
	{where $H_1,H_2$ are defined in Table \ref{table:def_H}.}
	Note that $\tilde{h}_1,\tilde{h}_2$ are a.s. well-defined (cf. Lemma \ref{lem:h_tilde_well_defined}).
	\item  Let $\tilde{\Omega}_n\equiv \{\tilde{h}_1\vee \tilde{h}_2\leq \tau_n\}$ for some $\tau_n>0$ to be specified below. 
	\item For some $t_n>0$ to be specified below, let $\mathcal{E}_n\equiv \{\abs{\omega_n^{-1}\big(\hat{f}_n(x_0)-f_0(x_0))}\leq t_n\}$ and $\tilde{\mathcal{E}}_n\equiv \{\abs{\sup_{h_1 \in H_1}\inf_{h_2 \in H_2} \mathbb{B}_{\sigma,\Lambda_0,Q}(h_1,h_2)}\leq t_n\}$.
\end{itemize}

For simplicity of notation, we assume $\sigma=1$ throughout the proof.

The following lemma explicitly calculates the order of the bias.

\begin{lemma}[Bias calculation]\label{lem:bias_calculation}
	{Suppose Assumptions \ref{assump:local_smoothness} and \ref{assumption:design_points} hold. In the fixed design setting, further assume that $x^\ast \in \{X_i\}_{i=1}^n$.} Let $r_n \downarrow 0$ for $\alpha<\infty$. Then for $\tau_n\geq 1$ such that $r_n\tau_n^b \downarrow 0$ for any $b>0$, in both fixed and random designs, the following holds with probability at least $1-O(n^{-11})$, uniformly in $h_1, h_2\leq \tau_n$:
	\begin{enumerate}
		\item If $x_0 \in (0,1)$,
		\begin{align*}
		& (n r_n)^{-1}\sum_{x_0-h_1r_n\leq X_i\leq x_0+h_2 r_n} \big(f_0(X_i)-f_0(x_0)\big)\\
		& =  \frac{f_0^{(\alpha)}(x_0)}{(\alpha+1)!}\cdot \big(h_2^{\alpha+1}-h_1^{\alpha+1}\big)\cdot  \Lambda_0 r_n^{\alpha} \bm{1}_{\alpha<\infty}  \\
		& \quad\quad +
		\begin{cases}
		O \bigg(  \tau_n ^{\alpha^\ast+1} r_n^{\alpha^\ast}  \bm{1}_{\alpha^\ast<\infty}\vee \tau_n^\alpha r_n^{\alpha} (n r_n)^{-1} \bm{1}_{\alpha<\infty}\bigg), &\quad \textrm{fixed design}\\
		O\bigg(\tau_n^{\alpha^\ast+1}r_n^{\alpha^\ast}\bm{1}_{\alpha^\ast<\infty} \vee  \tau_n^{\alpha+\beta+1} r_n^{\alpha+\beta}\bm{1}_{\alpha\vee \beta<\infty}\\
		\quad \bigvee \sqrt{\tau_n^{2\alpha+1} r_n^{2\alpha} \frac{\log n}{nr_n} \vee \left(\frac{\log n}{n r_n }\right)^2 }\cdot \bm{1}_{\alpha<\infty}\bigg),&\quad \textrm{random design}
		\end{cases}
		.
		\end{align*}
		\item If $x_0 =0, x_n \downarrow 0$, 
		\begin{align*}
		& (n r_n)^{-1}\sum_{x_n-h_1r_n\leq X_i\leq x_n+h_2 r_n} \big(f_0(X_i)-f_0(x_n)\big)\\
		& =  \sum_{\ell=1}^{\alpha} \frac{f_0^{(\alpha)}(0)}{(\alpha-\ell)!(\ell+1)!}\cdot \big(h_2^{\ell+1}-(-h_1)^{\ell+1} \big)\cdot \Lambda_0 x_n^{\alpha-\ell} r_n^{\ell}\bm{1}_{\alpha<\infty} \\
		& \quad\quad +
		\begin{cases}
		O\bigg(\max_{1\leq \ell \leq \alpha^\ast} \tau_n^{\ell+1} x_n^{\alpha^\ast-\ell} r_n^{\ell} \bm{1}_{\alpha^\ast<\infty} \\
		\quad \bigvee \max_{1\leq \ell \leq \alpha } \tau_n^\ell x_n^{\alpha-\ell} r_n^\ell (n r_n)^{-1} \bm{1}_{\alpha<\infty} \bigg), & \textrm{fixed design}\\
		O\bigg( \max_{1\leq \ell\leq \alpha^\ast} \tau_n^{\ell+1}x_n^{\alpha^\ast-\ell} r_n^\ell \bm{1}_{\alpha^\ast<\infty}\\
		\quad  \bigvee \max_{1\leq \ell \leq \alpha} \bigg\{\tau_n^{\ell+\beta+1} x_n^{\alpha-\ell} r_n^{\ell+\beta} \bm{1}_{\alpha\vee \beta<\infty}\\
		\quad \bigvee  x_n^{\alpha-\ell}\sqrt{\tau_n^{2\ell+1} r_n^{2\ell} \frac{\log n}{nr_n} \vee \left(\frac{\log n}{nr_n}\right)^2 }\cdot \bm{1}_{\alpha<\infty}\bigg)\bigg\} ,& \textrm{random design}
		\end{cases}
		.
		\end{align*}
		
	\end{enumerate}
\end{lemma}

{
	\begin{remark}\label{rmk:removal_x_on_design}
		In the fixed design setting, the assumption $x^\ast \in \{X_i\}$ can be removed with an additional term of order at most $O(\tau_n/n)$. See the comments on this point in the proof of Lemma \ref{lem:bias_calculation} in the appendix. 
		%		 by controlling the error of $f_0(x^\ast)$ and $f_0(\bar{x}^\ast)$, where $\bar{x}^\ast$ is the nearest point in $\{X_i\}$ to $x^\ast$.
	\end{remark}
}

The following lemma gives exponential bounds for the supremum of a weighted partial sum process.

\begin{lemma}\label{lem:tail_bound_noise}
	Suppose $\xi_i$'s are i.i.d. mean-zero sub-exponential random variables. Then for both fixed and random design cases, there exists some constant $K>0$ such that for $t\geq 1$, 
	\begin{align*}
	&\Prob\bigg(\sup_{h\geq 0}\abs{ \bar{\xi}|_{[x^\ast-hr_n,x^\ast+r_n]}}>t\omega_n\bigg) \bigvee \Prob\bigg(\sup_{h\geq 0}\abs{ \bar{\xi}|_{[x^\ast-r_n,x^\ast+hr_n]}}>t\omega_n\bigg)  \\
	&\leq K \big(e^{-\{t^2 \wedge (n r_n)^{1/2} t\}/K}+n^{-11}\big).
	\end{align*}
\end{lemma}

Proofs for the above lemmas can be found in the appendix.

\subsection{Localization}

Recall the events $\mathcal{E}_{n}$ and $\tilde{\mathcal{E}}_{n}$ defined in Section \ref{sec:preliminary}. The following lemma shows that each of these events has probability $1-O(n^{-11})$ for $t_n \asymp \sqrt{\log n}$.

\begin{lemma}\label{lem:large_deviation_0}
	Suppose the conditions in Theorem \ref{thm:berry_esseen_isoreg} hold. For $t_n = K\sqrt{\log n}$ with some large $K>0$, we have $\Prob\big(\mathcal{E}_n^c\big)\vee \Prob\big(\tilde{\mathcal{E}}_n^c\big)\leq O(n^{-11})$.
\end{lemma}
\begin{proof}
	First consider $x_0 \in (0,1)$. Note that by the max-min formula and monotonicity of $f_0$,  
	\begin{align}\label{ineq:large_deviation_1}
	\hat{f}_n(x_0)-f_0(x_0)&\leq \big(\bar{f_0}|_{[x_0-h^\ast_1 r_n, x_0+r_n]}-f_0(x_0)\big)+ \bar{\xi}|_{[x_0-h^\ast_1r_n,x_0+r_n]}\\
	&\leq \big(\bar{f_0}|_{[x_0, x_0+r_n]}-f_0(x_0)\big)+\sup_{h\geq 0}\abs{ \bar{\xi}|_{[x_0-hr_n,x_0+r_n]}}.\nonumber
	\end{align}
	By Lemma \ref{lem:bias_calculation}, in both fixed and random design cases, for $n$ large enough, with probability at least $1-O(n^{-11})$,
	\begin{align*}
	\bar{f_0}|_{[x_0, x_0+r_n]}-f_0(x_0) = O(r_n^{\alpha}\bm{1}_{\alpha<\infty}).
	\end{align*} 
	On the other hand, by Lemma \ref{lem:tail_bound_noise} we have for some constant $K>0$, in both fixed and random design cases,
	\begin{align*}
	\Prob\bigg(\sup_{h>0}\abs{ \bar{\xi}|_{[x_0-hr_n,x_0+r_n]}}>K \omega_n\sqrt{\log n}\bigg)\leq O(n^{-11})
	\end{align*}
	holds for $n$ large enough. Hence with probability at least $1-O(n^{-11})$, $\big(\omega_n^{-1}(\hat{f}_n(x_0)-f_0(x_0))\big)_+\leq K_1 \sqrt{\log n}$. The other direction can be argued similarly. This proves $\Prob(\mathcal{E}_n^c)\leq O(n^{-11})$. The analogous claim also holds for its limit version by using \cite[Lemma 5]{han2020limit}. We omit the details.
	
	Next suppose $x_0=0$ and $\rho \in (0,1)$. Using (\ref{ineq:large_deviation_1}) and Lemma \ref{lem:bias_calculation}, we have with the same probability estimate, it holds that
	\begin{align*}
	\hat{f}_n(x_n)-f_0(x_n)\leq O\bigg(\max_{1\leq \ell \leq \alpha} x_n^{\alpha-\ell}r_n^\ell \bm{1}_{\alpha<\infty} \bigvee \omega_n \sqrt{\log n}\bigg).
	\end{align*}
	The reverse direction is similar.
\end{proof}

Next we will show that each of the events $\Omega_{n}$ and $\tilde{\Omega}_{n}$ defined in Section \ref{sec:preliminary} has probability $1-O(n^{-11})$ for some slowly growing sequence $\tau_n$. 

\begin{lemma}\label{lem:h_tilde_well_defined}
	The random variables $\tilde{h}_1$ and $\tilde{h}_2$ in (\ref{def:h_tilde}) are a.s. well-defined. 
\end{lemma}	

The proof of the above technical lemma can be found in the appendix.

\begin{lemma}\label{lem:large_deviation}
	Suppose the conditions in Theorem \ref{thm:berry_esseen_isoreg} hold. For $\alpha<\infty$ and
	\begin{align*}
	\tau_n \equiv K\cdot 
	\begin{cases}
	(\log n)^{1/2\alpha}, &\quad \text{if} \ x_0 \in (0,1)\\
	\log^{1/2} n, &\quad \text{if} \ x_0=0 \ \text{and} \ \rho \in (0,1/(2\alpha+1))\\
	(\log n)^{1/2\alpha}, &\quad \text{if} \ x_0=0 \ \text{and} \ \rho = 1/(2\alpha+1)
	\end{cases}
	\end{align*}
	with a sufficiently large $K>0$, we have $\Prob(\Omega_n^c)\vee\Prob(\tilde{\Omega}_n^c) \leq O(n^{-11})$.
\end{lemma}
\begin{proof}
	First consider $x_0 \in (0,1)$. Let $K_1>0$ be the constant in Lemma \ref{lem:large_deviation_0}. Consider the event $\{h_2^\ast \geq \tau_n\}$. On this event, by max-min formula, we have
	\begin{align*}
	\hat{f}_n(x_0)-f_0(x_0)&\geq \bar{f_0}|_{[x_0-r_n,x_0+h_2^\ast r_n]}-f_0(x_0)+ \bar{\xi}|_{[x_0-r_n,x_0+h_2^\ast r_n]}\\
	&\geq \big(\bar{f_0}|_{[x_0-r_n,x_0+\tau_n r_n]}-f_0(x_0)\big)-\sup_{h\geq 0}\abs{\bar{\xi}|_{[x_0-r_n,x_0+h r_n]}}.
	\end{align*}
	The bias term is easy to compute: by Lemma \ref{lem:bias_calculation}, in both fixed and random design cases, with probability at least $1-O(n^{-11})$, 
	\begin{align*}
	&\bar{f_0}|_{[x_0-r_n,x_0+\tau_n r_n]}-f_0(x_0)= O\bigg( \frac{(\tau_n^{\alpha+1}-1)}{\tau_n+1} r_n^\alpha\bigg) \geq c_0 \tau_n^{\alpha} r_n^\alpha
	\end{align*}
	holds for some constant $c_0=c_0(\alpha,f_0,x_0)>0$ and $n$ large enough. On the other hand, by using again Lemma \ref{lem:tail_bound_noise}, we conclude that with probability at least $1-O(n^{-11})$, $\sup_{h\geq 0}\abs{\bar{\xi}|_{[x_0-r_n,x_0+h r_n]}}\leq K_2\omega_n\sqrt{\log n}$. We choose $
	\tau_n\equiv \big({(K_1+K_2)\sqrt{\log n}}/{c_0}\big)^{1/\alpha}$. 
	Combining the above estimates, on the intersection of $\{h_2^\ast \geq \tau_n\}$ and an event with probability at least $1-O(n^{-11})$, we have
	\begin{align*}
	\omega_n^{-1}\big(\hat{f}_n(x_0)-f_0(x_0)\big)&\geq c_0\tau_n^{\alpha}-K_2\sqrt{\log n}\geq K_1\sqrt{\log n},
	\end{align*}
	which occurs with probability at most $O(n^{-11})$ by Lemma \ref{lem:large_deviation_0}. Hence $\Prob\big(h_2^\ast \geq \tau_n\big)\leq O(n^{-11})$ for $n$ large enough. Similar considerations apply to $h_1^\ast$, and the limit versions. Details are omitted.
	
	Next consider $x_0=0$ with $\rho \in (0,1/(2\alpha+1)]$. Using Lemma \ref{lem:bias_calculation}, we have
	\begin{align*}
	&\bar{f_0}|_{[x^\ast-r_n,x^\ast+\tau_n r_n]}-f_0(x^\ast) = O\bigg(\sum_{\ell=1}^\alpha \frac{\tau_n^{\ell+1}-(-1)^{\ell+1}}{\tau_n+1}\cdot x_n^{\alpha-\ell}r_n^\ell\bigg) \\
	&\geq c_1 \max_{1\leq \ell \leq \alpha} x_n^{\alpha-\ell} \tau_n^\ell r_n^\ell = c_1
	\begin{cases}
	x_n^{\alpha-1}\tau_n r_n,&\quad \rho \in (0,1/(2\alpha+1))\\
	\tau_n^{\alpha} r_n^{\alpha},&\quad \rho = 1/(2\alpha+1)
	\end{cases}
	\end{align*}
	for some $c_1=c_1(\alpha,f_0)$, which holds in both fixed and random design settings with probability at least $1-O(n^{-11})$. The claim now follows from similar arguments above. 
\end{proof}

\section{Proof of Theorem \ref{thm:berry_esseen_isoreg}}\label{section:proof_berry_esseen_isotonic}

\subsection{Proof for the fixed design}

In addition to the anti-concentration inequality and localization, 
the  Kolm\'os-Major-Tusn\'ady (KMT) strong embedding theorem \cite{komlos1975approximation,komlos1976approximation} will play an important role. The formulation below is taken from \cite{chatterjee2012new}. 
\begin{lemma}
	[KMT strong embedding]
	\label{lem:strong_embedding}
	Let $\xi_1,\ldots,\xi_n$ be i.i.d. mean-zero, unit variance, and sub-exponential random variables, i.e. $\E \xi_1 =0, \E \xi_1^2 =1$, and $\E e^{\theta \xi_{1}}<\infty$ for all $\theta$ in a neighborhood of the origin. Then for each $n$, a version of $\big(S_k\equiv \sum_{i=1}^k \xi_i\big)_{1\leq k\leq n}$ and a standard Brownian motion $\big(\mathbb{B}_n(t)\big)_{0\leq t\leq n}$ can be constructed on the same probability space such that for all $x\geq 0$,
	\begin{align*}
	\Prob\bigg(\max_{1\leq k\leq n}\abs{S_k-\mathbb{B}_n(k)}\geq C\log n+x\bigg)\leq K\exp(-x/K).
	\end{align*}
	Here the constants $C,K>0$ depend on the distribution of $\xi_1$ only.
\end{lemma}

\begin{proof}[Proof of Theorem \ref{thm:berry_esseen_isoreg}: $x_0 \in (0,1)$ or $x_0=0$, $0<\rho\leq 1/(2\alpha+1)$, $\alpha<\infty$]
	{As in the proof of Proposition \ref{prop:oracle_CLT}, we first work with the extra condition that $x^\ast \in \{X_i\}$.  Then} for any $\abs{t}\leq t_n$, by max-min formula we have
	\begin{align*}
	&\Prob\big(\omega_n^{-1}(\hat{f}_n(x^\ast)-f_0(x^\ast)) \leq t \big)\\
	& = \Prob\bigg(\max_{h_1> 0} \min_{h_2\geq 0} \omega_n^{-1}\big(\bar{\xi}|_{[x^\ast-h_1r_n,x^\ast+h_2 r_n]  } + \bar{f_0}|_{[x^\ast-h_1 r_n,x^\ast+h_2 r_n]  }-f_0(x^\ast) \big)\leq t\bigg)\\
	&\leq \Prob\bigg(\max_{ \substack{0< h_1\leq \tau_n\\ h_1 \in H_1}}\min_{\substack{0\leq h_2\leq \tau_n\\ h_2 \in H_2}} \bigg[\omega_n \sum_{x^\ast-h_1r_n\leq X_i\leq x^\ast+h_2 r_n} \xi_i   \\
	&\quad\quad + \big(\floor{h_1r_n \cdot  n}+\floor{h_2 r_n\cdot  n}+1\big) \omega_n \big(\bar{f_0}|_{[x^\ast-h_1 r_n,x^\ast+h_2 r_n]  }-f_0(x^\ast)\big) \\
	& \quad\quad - \big(\floor{h_1r_n \cdot  n}+\floor{h_2 r_n\cdot n}+1\big) \omega_n^2 t\bigg]\leq 0\bigg) +\Prob(\Omega_n^c)\\
	&\leq \Prob\bigg(\max_{ \substack{0< h_1\leq \tau_n\\ h_1 \in H_1}}\min_{\substack{0\leq h_2\leq \tau_n\\ h_2 \in H_2}} \bigg[\omega_n \sum_{x^\ast-h_1r_n\leq X_i\leq x^\ast+h_2 r_n} \xi_i + Q(h_2)-Q(-h_1)\\
	&\quad\quad -t(h_1+h_2)\bigg]\leq  O\bigg(\omega_n^2 t_n\bigvee \mathcal{R}_n^f \cdot \tau_n^{(\alpha^\ast+1)\bm{1}_{\alpha^\ast<\infty}+\alpha \bm{1}_{\alpha<\infty}} \bigg) \bigg) +\Prob(\Omega_n^c).
	\end{align*}
	Here $\mathcal{R}_n^f$ is defined in (\ref{def:R_n_oracle}). The inequality in the last line of the above display follows since by Lemma \ref{lem:bias_calculation},
	\begin{align*}
	&\big(\floor{h_1r_n \cdot n}+\floor{h_2 r_n\cdot n}+1\big) \omega_n \big(\bar{f_0}|_{[x^\ast-h_1 r_n,x^\ast+h_2 r_n]  }-f_0(x^\ast)\big)\\
	& = (n r_n)^{1/2}\cdot (n r_n)^{-1} \sum_{x^\ast-h_1r_n\leq X_i\leq x^\ast+h_2 r_n} \big(f_0(X_i)-f_0(x^\ast))\\
	& = Q(h_2)-Q(-h_1) + O\big(\mathcal{R}_n^f \cdot \tau_n^{(\alpha^\ast+1)\bm{1}_{\alpha^\ast<\infty}+\alpha \bm{1}_{\alpha<\infty}} \big),
	\end{align*}
	and $\big(\floor{h_1r_n \cdot n}+\floor{h_2 r_n\cdot n}+1\big) \omega_n^2 t=t(h_1+h_2) +O(\omega_n^2 t_n)$. By the KMT strong embedding (cf. Lemma \ref{lem:strong_embedding}), there exist independent Brownian motions $\mathbb{B}_n,\mathbb{B}_n'$ such that for some constant $C_0>0$ that does not depend on $n$, with probability $1-O(n^{-11})$, uniformly in $h_1,h_2\geq 0$,
	\begin{align*}
	&\biggabs{\sum_{x^\ast\leq X_i \leq x^\ast+h_2 r_n} \xi_i - \mathbb{B}_n\big(\floor{h_2r_n\cdot n}+1\big)} \\
	&\quad\quad\quad\quad\quad\quad\quad  \bigvee \biggabs{\sum_{x^\ast-h_1 r_n\leq X_i < x^\ast} \xi_i - \mathbb{B}_n'\big(\floor{h_1r_n\cdot  n}\big)}\leq C_0 \log n.
	\end{align*}
	This means that, with
	\begin{align}\label{ineq:R_n}
	\mathscr{R}_n^f\equiv \max\bigg\{\mathcal{R}_n^f \cdot \tau_n^{(\alpha^\ast+1)\bm{1}_{\alpha^\ast<\infty}\vee \alpha \bm{1}_{\alpha<\infty}}, \omega_n^2\big(t_n\vee\sqrt{\log n}\big), \omega_n \log n\bigg\},
	\end{align}
	we have
	\begin{equation}\label{ineq:proof_fixed_design_1}
	\begin{split}
	&\Prob\big(\omega_n^{-1}(\hat{f}_n(x^\ast)-f_0(x^\ast)) \leq t \big)\\
	&\leq \Prob\bigg(\max_{ \substack{0< h_1\leq \tau_n\\ h_1 \in H_1}}\min_{\substack{0\leq h_2\leq \tau_n\\ h_2 \in H_2}} \bigg[ \omega_n \mathbb{B}\big(\floor{h_2r_n\cdot  n}+1\big)-\omega_n \mathbb{B}\big(-\floor{h_1 r_n\cdot  n}\big)\\
	&\quad\quad\quad\quad +Q(h_2)-Q(-h_1) -  t(h_1+h_2)\bigg]\leq K_0 \mathscr{R}_n^f\bigg) +\Prob(\Omega_n^c)+O(n^{-11})\\
	&\leq \Prob\bigg(\max_{ \substack{0< h_1\leq \tau_n\\ h_1 \in H_1}}\min_{\substack{0\leq h_2\leq \tau_n\\ h_2 \in H_2}} \bigg[\big(\mathbb{B}(h_2)-\mathbb{B}(-h_1)\big)\\
	&\quad\quad\quad\quad  +Q(h_2)-Q(-h_1) - t(h_1+h_2)  \bigg]\leq K_1 \mathscr{R}_n^f\bigg)+\Prob(\Omega_n^c)+O(n^{-11}).
	\end{split}
	\end{equation}
	The last inequality follows since by Lemma \ref{lem:dudley_entropy_integral},
	\begin{align*}
	&\E \sup_{\substack{0\leq h_i\leq \tau_n, i=1,2\\ \abs{h_1-h_2}\leq \omega_n^2} } \abs{\mathbb{B}(h_1)-\mathbb{B}(h_2)}\lesssim \omega_n\sqrt{\log n}, \quad \sup_{\substack{0\leq h_i\leq \tau_n, i=1,2\\ \abs{h_1-h_2}\leq \omega_n^2} } \mathrm{Var}\big(\mathbb{B}(h_1)-\mathbb{B}(h_2)\big)\leq \omega_n^2
	\end{align*}
	and hence by  the Gaussian concentration (cf. Lemma \ref{lem:Gaussian_concentration}),
	we have for a large enough constant $C_1>0$,
	\begin{align*}
	&\Prob\bigg(\sup_{\substack{0\leq h_i\leq \tau_n, i=1,2\\ \abs{h_1-h_2}\leq \omega_n^2} } \abs{\mathbb{B}(h_1)-\mathbb{B}(h_2)}> C_1 \omega_n\sqrt{\log n} \bigg) \leq e^{-C_1^2 \omega_n^2\log n/8 \omega_n^2}\leq O(n^{-11}).
	\end{align*}
	Let 
	\begin{align*}
	T_{n,1}
	&\equiv \max_{0\leq  h\leq \tau_n, h\in H_1}\bigg[-\mathbb{B}(-h)-Q(-h)-th\bigg],\\
	T_{n,2}& \equiv \min_{0\leq h\leq\tau_n, h \in H_2} \bigg[\mathbb{B}(h)+Q(h)-th\bigg] = - \max_{0\leq h\leq \tau_n, h \in H_2} \bigg[-\mathbb{B}(h)-Q(h)+th\bigg].
	\end{align*}
	By (\ref{ineq:proof_fixed_design_1}), we have
	\begin{align*}
	&\Prob\big(\omega_n^{-1}\big(\hat{f}_n(x^\ast)-f_0(x^\ast)\big)\leq t\big)\leq \Prob\big(T_{n,1}+T_{n,2} \leq K_1 \mathscr{R}_n^f\big) +\Prob(\Omega_n^c)+O(n^{-11}).
	\end{align*}
	Now apply the anti-concentration Theorem \ref{thm:anti_concentration} with the following choices of $(t_n,\tau_n,\bar{b},\epsilon)$: 
	\begin{itemize}
		\item $x_0 \in (0,1)$:  $t_n\asymp \sqrt{\log n}, \tau_n \asymp (\log n)^{1/2\alpha}, \bar{b}\asymp \tau_n^{\alpha+1/2}\asymp \log^{(\alpha+1/2)/2\alpha} n, \epsilon \asymp \mathscr{R}_n^f/\sqrt{\tau_n}$,
		\item $x_0 =0, \rho \in (0,1/(2\alpha+1))$: $t_n,\tau_n \asymp \sqrt{\log n}, \bar{b}\asymp \tau_n^{3/2}, \epsilon\asymp \mathscr{R}_n^f/\sqrt{\tau_n}$,
		\item $x_0=0, \rho = 1/(2\alpha+1)$: $t_n\asymp \sqrt{\log n}, \tau_n \asymp (\log n)^{1/2\alpha}, \bar{b}\asymp \tau_n^{\alpha+1/2}\asymp \log^{(\alpha+1/2)/2\alpha} n, \epsilon \asymp \mathscr{R}_n^f/\sqrt{\tau_n}$.
	\end{itemize}
	Then, in view of Remark \ref{rmk:anti_conc_grow_interval}, we see that for any $\abs{t}\leq t_n$,
	\begin{align*}
	&\Prob\big(\omega_n^{-1}\big(\hat{f}_n(x^\ast)-f_0(x^\ast)\big)\leq t\big)\\
	&= \Prob\big(T_{n,1}+T_{n,2}\leq 0\big)+ K_2 \mathscr{R}_n^f \log^{5/2} n +\Prob(\Omega_n^c)+O(n^{-11})\\
	&\leq \Prob\bigg(\max_{ \substack{0< h_1\leq \tau_n\\ h_1 \in H_1}}\min_{\substack{0\leq h_2\leq \tau_n\\ h_2 \in H_2}} \bigg[\mathbb{B}(h_2)-\mathbb{B}(-h_1) +Q(h_2)-Q(-h_1) - t(h_1+h_2)  \bigg]\leq 0\bigg)\\
	&\quad\quad + K_2 \mathscr{R}_n^f \log^{5/2} n+ \Prob\big(\Omega_n^c\big) +O(n^{-11}) \\
	&\leq \Prob\bigg(\max_{ \substack{0< h_1\leq \tau_n\\ h_1 \in H_1}}\min_{\substack{0\leq h_2\leq \tau_n\\ h_2 \in H_2}} \frac{\mathbb{B}(h_2)-\mathbb{B}(-h_1)+Q(h_2)-Q(-h_1)}{h_1+h_2}\leq t\bigg)\\
	&\quad\quad + K_2 \mathscr{R}_n^f \log^{5/2} n + \Prob\big(\Omega_n^c\big)+O(n^{-11})\\
	&\leq \Prob \bigg(\sup_{h_1 \in H_1}\inf_{h_2 \in H_2} \mathbb{B}_{\sigma,\Lambda_0,Q}(h_1,h_2) \leq t \bigg)+ K_2 \mathscr{R}_n^f \log^{5/2} n + \Prob\big(\Omega_n^c\big)+\Prob(\tilde{\Omega}_n^c)+O(n^{-11}).
	\end{align*}
	Recalling the definitions of $\mathcal{E}_n$ and $\tilde{\mathcal{E}}_n$ and arguing the reverse direction similarly, we have
	\begin{align*}
	&\sup_{t \in \R} \biggabs{\Prob\big(\omega_n^{-1} \big(\hat{f}_n(x^\ast)-f_0(x^\ast)\big)\leq t\big)- \Prob \bigg( \sup_{h_1 \in H_1}\inf_{h_2 \in H_2} \mathbb{B}_{\sigma,\Lambda_0,Q}(h_1,h_2) \leq t\bigg) } \\
	&\leq K_2 \mathscr{R}_n^f \log^{5/2} n+ \Prob\big(\Omega_n^c\big)+\Prob(\tilde{\Omega}_n^c)+ \Prob(\mathcal{E}_n^c)+\Prob(\tilde{\mathcal{E}}_n^c)+O(n^{-11}).
	\end{align*}
	The claim of the theorem {under $x^\ast \in \{X_i\}$} now follows from Lemmas \ref{lem:large_deviation_0} and \ref{lem:large_deviation}. {For $x^\ast$ in general position, by Remark \ref{rmk:removal_x_on_design}, in the definition (\ref{ineq:R_n}) of $\mathscr{R}_n^f$, the quantity $\mathcal{R}_n^f$ need be replaced by $\mathcal{R}_n^f \vee (\omega_n^{-1}\cdot (\tau_n/n))$. The contribution of the additional maximum can be assimilated into the $\omega_n \log n$ term in (\ref{ineq:R_n}), so the above display remains valid.}
\end{proof}

\begin{proof}[Proof of Theorem \ref{thm:berry_esseen_isoreg}: $x_0 =0, 1/(2\alpha+1)<\rho<1$, $\alpha<\infty$.] 
	
	{We only consider the case $x^\ast \in \{X_i\}$.} In this regime, Lemma \ref{lem:large_deviation} does not apply so we do not have exponential localization in $h_i^\ast, \tilde{h}_i, i=1,2$. However, Lemma \ref{lem:large_deviation_0} still applies, and we do have sub-Gaussian localization of the statistics $\omega_n^{-1}\big(\hat{f}_n(x_n)-f_0(x_n))$ and the limiting distribution 
	\begin{align*}
	\sup_{h_1 \in (0,1]} \inf_{h_2 \in [0,\infty)} \mathbb{B}_{1,1,0}(h_1,h_2)= \sup_{h_1 \in (0,1]} \inf_{h_2 \in [0,\infty)} \frac{\mathbb{B}(h_2)-\mathbb{B}(-h_1)}{h_1+h_2}.
	\end{align*}
	Hence, for any $\abs{t}\leq t_n$, repeating the arguments in the previous proof, with $\bar{T}_{n,1}\equiv \max_{h \in [0,1]}\big(-\mathbb{B}(-h)-th\big)$ and $\bar{T}_{n,2}\equiv \min_{h \in [0,(1-x_n)/x_n)}\big(\mathbb{B}(h_2)-t h_2\big)$,
	\begin{align}\label{ineq:berry_esseen_boundary_1}
	&\Prob\big(\omega_n^{-1}(\hat{f}_n(x_n)-f_0(x_n))\leq t)\\
	& \leq \Prob\bigg(\max_{h_1 \in (0,1]} \min_{h_2 \in [0,(1-x_n)/x_n)} \bigg[\mathbb{B}(h_2)-\mathbb{B}(-h_1)-t(h_1+h_2)\bigg]\leq K_1 \mathscr{R}_n^f \bigg) +O(n^{-11}) \nonumber\\
	& = \Prob \big( \bar{T}_{n,1} \leq -\bar{T}_{n,2}+K_1 \mathscr{R}_n^f \big)+O(n^{-11})\nonumber \\
	&\leq \Prob \big(\bar{T}_{n,1} +\bar{T}_{n,2} \leq 0 \big)+K_2 \mathscr{R}_n^f \log^{5/2} n +O(n^{-11})\nonumber\\
	& = \Prob\bigg(\sup_{h_1 \in (0,1]} \inf_{h_2 \in [0,(1-x_n)/x_n)} \mathbb{B}_{1,1,0}(h_1,h_2)\leq t\bigg)+ K_2 \mathscr{R}_n^f \log^{5/2} n+ O(n^{-11}),\nonumber
	\end{align}
	where in the first inequality we used independence between $\bar{T}_{n,1} $ and $\bar{T}_{n,2} $, and anti-concentration Theorem \ref{thm:anti_concentration} for $\bar{T}_{n,1}$ with $\bar{b}\asymp \sqrt{\log n}$. Let $\mathcal{E}_{n,2}\equiv \{h_2^\ast \geq (1-x_n)/x_n\}$. Then on $\mathcal{E}_{n,2}$, 
	\begin{align*}
	&\sup_{h_1 \in (0,1]} \inf_{h_2 \in [0,\infty)} \mathbb{B}_{1,1,0}(h_1,h_2)\geq \inf_{h_2\geq (1-x_n)/x_n} \frac{\mathbb{B}(h_2)-\mathbb{B}(-1)}{h_2+1}\\
	&\geq -\sup_{h_2\geq n^\rho/2} \biggabs{\frac{\mathbb{B}(h_2)}{h_2}}- \frac{\abs{\mathbb{B}(-1)}}{n^\rho/2+1} \equald-\frac{Y_1}{(n^\rho/2)^{1/2}}-\frac{Y_2}{n^\rho/2+1},
	\end{align*}
	where $Y_1 \equiv \sup_{h\geq 1}\abs{\mathbb{B}(h)/h}$ and $Y_2 \equiv \abs{\mathbb{B}(-1)}$ are non-negative and have sub-Gaussian tails. Hence on the intersection of $\mathcal{E}_{n,2}$ and an event with probability at least $1-O(n^{-11})$, we have
	\begin{align*}
	\sup_{h_1 \in (0,1]} \inf_{h_2 \in [0,\infty)} \mathbb{B}_{1,1,0}(h_1,h_2) \geq - O\big(n^{-\rho/2}\sqrt{\log n}\big).
	\end{align*}
	By Lemma \ref{lem:anti_concentration_oneside}, $\Prob(\mathcal{E}_{n,2})\leq O(n^{-\rho/4}\log^{1/2}n)$. Combined with (\ref{ineq:berry_esseen_boundary_1}), this means that 
	\begin{align*}
	&\Prob\big(\omega_n^{-1}(\hat{f}_n(x_n)-f_0(x_n))\leq t)\\
	& \leq \Prob\bigg(\sup_{h_1 \in (0,1]} \inf_{h_2 \in [0,\infty)} \mathbb{B}_{1,1,0}(h_1,h_2)\leq t\bigg)+ K_2 \mathscr{R}_n^f \log^{5/2} n+ O(n^{-\rho/4}\log^{1/2}n).
	\end{align*}
	The inequality above can be reversed (note here that from (\ref{ineq:berry_esseen_boundary_1}) one may directly enlarge the range of inf to $h_2 \in [0,\infty)$; but this argument does not work for the reversed direction). The claim now follows as the last term can be assimilated when $\rho \in [2/3,1)$.
\end{proof}

\begin{proof}[Proof of Theorem \ref{thm:berry_esseen_isoreg}: $\alpha=\infty$.]
	{We only consider the case $x^\ast \in \{X_i\}$.} First consider $x_0 \in (0,1)$. This case follows quite straightforwardly: with ${T}_{n,1}\equiv \max_{h \in [0,x_0]}\big(-\mathbb{B}(-h)-th\big)$ and ${T}_{n,2}\equiv \min_{h \in [0,1-x_0]}\big(\mathbb{B}(h_2)-t h_2\big)$, for any $\abs{t}\leq t_n$, $t_n\asymp \sqrt{\log n}$, we have
	\begin{align*}
	&\Prob\big(\omega_n^{-1}(\hat{f}_n(x_0)-f_0(x_0))\leq t)\\
	& \leq \Prob\bigg(\max_{h_1 \in (0,x_0]} \min_{h_2 \in [0,1-x_0]} \bigg[\mathbb{B}(h_2)-\mathbb{B}(-h_1)-t(h_1+h_2)\bigg]\leq K_1 \mathscr{R}_n^f \bigg)+O(n^{-11}) \nonumber\\
	& = \Prob \big( {T}_{n,1}  +{T}_{n,2}\leq K_1 \mathscr{R}_n^f \big)+O(n^{-11})\nonumber \\
	&\leq \Prob \big({T}_{n,1} +{T}_{n,2} \leq 0 \big)+K_2 \mathscr{R}_n^f \log^{5/2} n +O(n^{-11})\nonumber\\
	& = \Prob\bigg(\sup_{h_1 \in H_1} \inf_{h_2 \in H_2} \mathbb{B}_{1,1,0}(h_1,h_2)\leq t\bigg)+ K_2 \mathscr{R}_n^f \log^{5/2} n+ O(n^{-11}).\nonumber
	\end{align*}
	Next consider $x_0 = 0$. By similar arguments as in the previous proof for the case $\alpha<\infty, x_0=0, 1/(2\alpha+1)<\rho<1$, we have
	\begin{align*}
	&\Prob\big(\omega_n^{-1}(\hat{f}_n(x_n)-f_0(x_n))\leq t)\\
	& = \Prob\bigg(\sup_{h_1 \in (0,1]} \inf_{h_2 \in [0,(1-x_n)/x_n)} \mathbb{B}_{1,1,0}(h_1,h_2)\leq t\bigg)+ K_2 \mathscr{R}_n^f \log^{5/2} n+ O(n^{-11})\\
	&\leq \Prob\bigg(\sup_{h_1 \in (0,1]} \inf_{h_2 \in [0,\infty)} \mathbb{B}_{1,1,0}(h_1,h_2)\leq t\bigg)+ K_2 \mathscr{R}_n^f \log^{5/2} n+ O(n^{-\rho/4}\log^{1/2}n),
	\end{align*}
	the last term of which can be assimilated for $\rho \in [2/3,1)$.
\end{proof}

\subsection{Proof for the random design}

\begin{proof}[Proof of Theorem \ref{thm:berry_esseen_isoreg}: random design]
	The proof strategy is broadly similar to the fixed design case, but differs quite substantially in technical details due to the randomness of $\{X_i\}$.

	First consider the case $x_0 \in (0,1)$ or $x_0=0$, $0<\rho\leq 1/(2\alpha+1)$, $\alpha<\infty$. Note that
	\begin{align}\label{ineq:proof_berry_esseen_1}
	&\Prob\big(\omega_n^{-1}(\hat{f}_n(x^\ast)-f_0(x^\ast)) \leq t \big)\\
	&\leq \Prob\bigg(\max_{0< h_1\leq \tau_n}\min_{0\leq h_2\leq \tau_n} \bigg[\omega_n \sum_{x^\ast-h_1r_n\leq X_i\leq x^\ast+h_2 r_n} \xi_i  \nonumber \\
	&\quad\quad + \big(n\Prob_n \bm{1}_{[x^\ast-h_1 r_n,x^\ast+h_2r_n]}\big) \omega_n \big(\bar{f_0}|_{[x^\ast-h_1 r_n,x^\ast+h_2 r_n]  }-f_0(x^\ast)\big)\nonumber \\
	& \quad\quad - \big(\omega_n^2 n\Prob_n \bm{1}_{[x^\ast-h_1 r_n,x^\ast+h_2r_n]}\big)  t\bigg]\leq 0\bigg) +\Prob(\Omega_n^c).\nonumber
	\end{align}
	By Lemma \ref{lem:bias_calculation}, with probability at least $1-O(n^{-11})$,
	\begin{align}\label{ineq:proof_berry_esseen_3}
	&\big(n\Prob_n \bm{1}_{[x^\ast-h_1 r_n,x^\ast+h_2r_n]}\big) \omega_n \big(\bar{f_0}|_{[x^\ast-h_1 r_n,x^\ast+h_2 r_n]  }-f_0(x^\ast)\big) \\
	&= Q(h_2)-Q(-h_1) + O\big( \mathcal{R}_n^r \cdot \tau_n^{\zeta^r}\big). \nonumber
	\end{align}
	Here $\mathcal{R}_n^r$ is defined in (\ref{def:R_n_oracle_random}) and
	\begin{align}\label{def:zeta_random_design}
	\zeta^r\equiv \zeta^r_{\alpha,\alpha^\ast,\beta}\equiv (\alpha^\ast+1)\bm{1}_{\alpha^\ast<\infty}\bigvee (\alpha+\beta+1) \bm{1}_{\alpha\vee\beta<\infty}\bigvee (\alpha+1/2).
	\end{align}
	Combining (\ref{ineq:proof_berry_esseen_1})-(\ref{ineq:proof_berry_esseen_3}), we have
	\begin{align}\label{ineq:proof_berry_esseen_4}
	&\Prob\big(\omega_n^{-1}(\hat{f}_n(x^\ast)-f_0(x^\ast)) \leq t \big)\\
	&\leq \Prob\bigg(\max_{0< h_1\leq \tau_n}\min_{0\leq h_2\leq \tau_n} \bigg[\omega_n \sum_{x^\ast-h_1r_n\leq X_i\leq x^\ast+h_2 r_n} \xi_i + Q(h_2)-Q(-h_1)\nonumber\\
	&\quad\quad -\big(\omega_n^2 n\Prob_n \bm{1}_{[x^\ast-h_1 r_n,x^\ast+h_2r_n]}\big)t\bigg]\leq O\big( \mathcal{R}_n^r \cdot \tau_n^{\zeta^r}\big)\bigg)  +\Prob(\Omega_n^c) +O(n^{-11}).\nonumber
	\end{align}
	By the KMT strong embedding, conditionally on $\{X_i\}$'s, there exist independent Brownian motions $\mathbb{B}_n,\mathbb{B}_n'$ (which may depend on $\{X_i\}$) such that for some constant $C_0>0$ that does not depend on $n$ or $\{X_i\}$, 
	\begin{align*}
	&\Prob\bigg(\sup_{h_2>0}\biggabs{\sum_{x^\ast\leq X_i \leq x^\ast+h_2 r_n} \xi_i - \mathbb{B}_n\big(n\Prob_n \bm{1}_{[x^\ast,x^\ast+h_2r_n]}\big)}\\
	&\quad\quad \bigvee \sup_{h_1>0} \biggabs{\sum_{x^\ast-h_1 r_n\leq X_i < x^\ast} \xi_i - \mathbb{B}_n'\big(n\Prob_n \bm{1}_{[x^\ast-h_1r_n,x^\ast)}\big)}\leq C_0 \log n\bigg\lvert \{X_i\}\bigg)\\
	&\geq 1-O(n^{-11}).
	\end{align*}
	We do not compare directly $\mathbb{B}_n(n\Prob_n \bm{1}_{[x^\ast,x^\ast+h_2r_n]})$ with $\mathbb{B}(h_2 nr_n)$ as in the fixed design case, as the standard deviation of $n\Prob_n\bm{1}_{[x^\ast,x^\ast+h_2r_n]}$ is of order $\sqrt{n r_n}=\omega_n^{-1}$ and therefore the comparison of Brownian motions leads to sub-optimal error bounds. We use a different re-parametrization idea as follows. Let ${h}_{1,n}\equiv \omega_n^2 n\Prob_n \bm{1}_{[x^\ast-h_1r_n,x^\ast)}$ and  ${h}_{2,n}\equiv \omega_n^2 n\Prob_n \bm{1}_{[x^\ast,x^\ast+h_2r_n]}$. Let 
	\begin{align*}
	\mathcal{E}_{n,1}&\equiv \bigg\{\sup_{ \substack{0\leq h_i\leq \tau_n,\\ i=1,2} }\abs{({h}_{1,n} +{h}_{2,n})-(h_1+h_2)}  \leq C_1 \omega_n^2 \sqrt{n\tau_n r_n \log n}= C_1 \omega_n \sqrt{\tau_n \log n} \bigg\}.
	\end{align*}
	Then for $C_1>0$ large enough, $\Prob(\mathcal{E}_{n,1}^c)\leq O(n^{-11})$.  Let $\tau_{1,n}\equiv \omega_n^2 n\Prob_n \bm{1}_{[x^\ast-\tau_n r_n,x^\ast)}$ and ${\tau}_{2,n}\equiv \omega_n^2 n\Prob_n \bm{1}_{[x^\ast,x^\ast+\tau_n r_n]}$. On the event $\mathcal{E}_{n,1}$, we have $\tau_{1,n}\geq \tau_n - C_1 \omega_n\sqrt{\tau_n \log n} \geq \tau_n/2$ and $\tau_{2,n}\leq \tau_n +C_1 \omega_n \sqrt{\tau_n \log n}\leq 2\tau_n$ for $n$ large enough. Therefore, by (\ref{ineq:proof_berry_esseen_4}), we have
	\begin{align*}
	&\Prob\big(\omega_n^{-1}(\hat{f}_n(x^\ast)-f_0(x^\ast)) \leq t \big)\\
	&\leq \Prob\bigg(\max_{ \substack{0< h_{1,n}\leq \tau_{n,1}, \\h_{1,n} \in \omega_n^2 \mathbb{Z}} }\min_{ \substack{0\leq h_{2,n}\leq \tau_{n,2},\\ h_{2,n} \in \omega_n^2 \mathbb{Z} } }\bigg[ \mathbb{B}(h_{2,n})-\mathbb{B}(-h_{1,n})+Q(h_{2,n})-Q(-h_{1,n})\\
	&\quad\quad -t(h_{1,n}+h_{2,n})\bigg]\leq O\big( \mathcal{R}_n^r \cdot \tau_n^{\zeta^r}\bigvee  \omega_n (\tau_n \log n)^{1/2} \tau_n^{\alpha} \bigvee \omega_n \log n\big), \mathcal{E}_{n,1}\bigg) \\
	&\quad\quad +\Prob(\Omega_n^c) +O(n^{-11})\\
	&\leq \Prob\bigg(\max_{ \substack{0< h_{1,n}\leq \tau_{n}/2, \\h_{1,n} \in \omega_n^2 \mathbb{Z}} }\min_{ \substack{0\leq h_{2,n}\leq 2\tau_{n},\\ h_{2,n} \in \omega_n^2 \mathbb{Z} } }\bigg[ \mathbb{B}(h_{2,n})-\mathbb{B}(-h_{1,n})+Q(h_{2,n})-Q(-h_{1,n})\\
	&\quad\quad -t(h_{1,n}+h_{2,n})\bigg]\leq O\big( \mathcal{R}_n^r \cdot \tau_n^{\zeta^r}\bigvee  \omega_n (\tau_n \log n)^{1/2} \tau_n^{\alpha} \big)\bigg) +\Prob(\Omega_n^c) +O(n^{-11}).
	\end{align*}
	The discretization effect in the above max-min formula can be handled in the $O$ term up to a further probability estimate on the order of $O(n^{-11})$ (for Brownian motion), so we obtain
	\begin{align*}
	&\Prob\big(\omega_n^{-1}(\hat{f}_n(x^\ast)-f_0(x^\ast)) \leq t \big)\\
	&\leq \Prob\bigg(\max_{ \substack{0< h_{1,n}\leq \tau_{n}/2 } }\min_{ \substack{0\leq h_{2,n}\leq 2\tau_{n} } }\bigg[ \mathbb{B}(h_{2,n})-\mathbb{B}(-h_{1,n})+Q(h_{2,n})-Q(-h_{1,n})\\
	&\quad\quad -t(h_{1,n}+h_{2,n})\bigg]\leq  O\big( \mathcal{R}_n^r \cdot \tau_n^{\zeta^r}\bigvee  \omega_n (\tau_n \log n)^{1/2}\tau_n^\alpha \big)\bigg) +\Prob(\Omega_n^c) +O(n^{-11}).
	\end{align*}
	Now we proceed to argue as in the fixed design case, except for $\mathscr{R}_n^f$ defined in (\ref{ineq:R_n}) is now replaced by
	\begin{align*}
	\mathscr{R}_n^r\equiv \mathcal{R}_n^r \cdot \tau_n^{\zeta^r}\bigvee  \omega_n (\tau_n \log n)^{1/2}\tau_n^\alpha,
	\end{align*}
	where $\zeta^r$ is defined in (\ref{def:zeta_random_design}). This completes the proof the case $x_0 \in (0,1)$ or $x_0=0$, $0<\rho\leq 1/(2\alpha+1)$, $\alpha<\infty$.
	
	For the remaining cases, we only consider $x_0 =0, 1/(2\alpha+1)<\rho<1$, $\alpha<\infty$ as other cases are simpler.  As $Q=0$, we no longer need to work on the event $\mathcal{E}_{n,1}$. Let $\tau_{1,n}^\ast\equiv \omega_n^2 n\Prob_n \bm{1}_{[0,x^\ast)}$ and ${\tau}_{2,n}^\ast\equiv \omega_n^2 n\Prob_n \bm{1}_{[x^\ast,1]}$. Then using Bernstein's inequality, it is easy to see that with probability at least $1-O(n^{-11})$, $\tau_{1,n}^\ast\geq 1-O(\omega_n\sqrt{\log n})$ and $\tau_{2,n}^\ast\leq 2n^\rho$ for $n$ large enough. Hence, for $\abs{t}\leq t_n$,
	\begin{align*}
	&\Prob\big(\omega_n^{-1}(\hat{f}_n(x^\ast)-f_0(x^\ast)) \leq t \big)\\
	&\leq \Prob\bigg(\max_{0< h_1\leq 1}\min_{0\leq h_2\leq (1-x_n)/x_n} \bigg[\omega_n \sum_{x^\ast-h_1r_n\leq X_i\leq x^\ast+h_2 r_n} \xi_i \nonumber\\
	&\quad\quad\quad\quad\quad\quad\quad -\big(\omega_n^2 n\Prob_n \bm{1}_{[x^\ast-h_1 r_n,x^\ast+h_2r_n]}\big)t\bigg]\leq O\big( \mathcal{R}_n^r \cdot \tau_n^{\zeta^r}\big)\bigg) \nonumber\\
	&\leq \Prob\bigg(\max_{ \substack{0< h_{1,n}\leq \tau_{n,1}^\ast, \\h_1 \in \omega_n^2 \mathbb{Z}} }\min_{ \substack{0\leq h_{2,n}\leq \tau_{n,2}^\ast,\\ h_{2,n} \in \omega_n^2 \mathbb{Z} } }\bigg[ \mathbb{B}(h_{2,n})-\mathbb{B}(-h_{1,n}) -t(h_{1,n}+h_{2,n}) \bigg]\\
	&\quad\quad\quad\quad\quad\quad\quad \leq O\big( \mathcal{R}_n^r \cdot \tau_n^{\zeta^r}\bigvee \omega_n \log n\big)\bigg)\\
	&\leq \Prob\bigg(\max_{ \substack{0< h_{1,n}\leq 1-O(\omega_n\sqrt{\log n})} }\min_{ \substack{0\leq h_{2,n}\leq 2n^\rho } }\bigg[ \mathbb{B}(h_{2,n})-\mathbb{B}(-h_{1,n}) -t(h_{1,n}+h_{2,n})\bigg] \\
	&\quad\quad\quad\quad\quad\quad\quad  \leq O\big( \mathcal{R}_n^r \cdot \tau_n^{\zeta^r}\bigvee \omega_n \log n\big)\bigg) +O(n^{-11})\\
	&\leq \Prob\bigg(\max_{ \substack{0< h_{1,n}\leq 1} }\min_{ \substack{0\leq h_{2,n}\leq 2n^\rho } }\bigg[ \mathbb{B}(h_{2,n})-\mathbb{B}(-h_{1,n}) -t(h_{1,n}+h_{2,n})\bigg] \\
	&\quad\quad\quad\quad\quad\quad\quad  \leq O\big( \mathcal{R}_n^r \cdot \tau_n^{\zeta^r}\bigvee \omega_n (\log n \vee t_n\sqrt{\log n})\big)\bigg) +O(n^{-11}).
	\end{align*}
	Here the last line follows by noting that  for $\abs{t}\leq t_n$, with probability at least $1-O(n^{-11})$,
	\begin{align*}
	&\max_{ \substack{0< h_{1,n}\leq 1-O(\omega_n\sqrt{\log n})}} \big(\mathbb{B}(-h_{1,n})-th_{1,n}\big) \\
	&\equald \big(1-O(\omega_n\sqrt{\log n})\big)^{1/2} \max_{0< h_{1,n}\leq 1} \bigg( \mathbb{B}(-h_{1,n})-(1-O(\omega_n\sqrt{\log n}))^{1/2} t h_{1,n}\bigg)\\
	& = \big(1-O(\omega_n\sqrt{\log n})\big)^{1/2} \max_{0< h_{1,n}\leq 1} \big(\mathbb{B}(-h_{1,n})-th_{1,n}\big)+ O\big(\omega_n \sqrt{\log n}\cdot t_n\big)\\
	& = \max_{0< h_{1,n}\leq 1} \big(\mathbb{B}(-h_{1,n})-th_{1,n}\big)+ O(\omega_n \sqrt{\log n})\cdot O(\sqrt{\log n})+ O\big(\omega_n \sqrt{\log n}\cdot t_n\big).
	\end{align*}
	Now we may proceed as in the fixed design case to conclude.
\end{proof}

\section{Concluding remarks and open questions}\label{section:final_remark}

In this paper we developed a new approach of proving Berry-Esseen bounds for Chernoff-type non-standard limit theorems in the isotonic regression model, by combining problem-specific localization techniques and an anti-concentration inequality for the supremum of a Brownian motion with a Lipschitz drift. The scope of the techniques applies to various known (or near-known) Chernoff-type non-standard asymptotics in isotonic regression allowing (i) general local smoothness conditions on the regression function, (ii) limit theorems both for interior points and points approaching the boundary, and (iii) both fixed and random design covariates.

Below we sketch two further open questions.

\begin{question}
	Prove a matching lower bound for the cube-root rate {(in the canonical case $\alpha=1$)} in the Berry-Esseen bound (\ref{eqn:isotonic_Berry_Esseen}).
\end{question}

As demonstrated in the simulation (Figure \ref{fig:sim1}), the oracle perspective (cf. Proposition \ref{prop:oracle_CLT}) is quite informative in that the cube-root rate in (\ref{eqn:isotonic_Berry_Esseen}) cannot be improved when the errors are i.i.d. Rademacher random variables. \cite{hall1984reversing} used Stein's method to prove a lower bound of order $n^{-1/2}$, for the Berry-Esseen bound for the central limit theorem for the sample mean  in the worst-case scenario. Unfortunately, the least squares estimator (\ref{def:isotonic_LSE}) is a highly non-linear and non-smooth functional of the samples in the isotonic regression model (\ref{model}), and therefore the connection between the Stein's method and the Berry-Esseen bound for the non-standard limit {theorem} (\ref{eqn:isotonic_Berry_Esseen}) remains largely unknown. New techniques seem necessary for proving a lower bound for (\ref{eqn:isotonic_Berry_Esseen}).

\begin{question}
	Prove a Berry-Esseen bound for the non-standard limit {theorem} for the block estimator of a multi-dimensional isotonic regression function (cf. \cite{han2020limit}).
\end{question}

Recently \cite{han2020limit} established a non-standard limit {theorem} for the {so-called} block estimator $\hat{f}_n$ (cf. \cite{fokianos2017integrated}) for a $d$-dimensional isotonic regression function $f_0$ on $[0,1]^d$ (i.e. $f_0(x)\leq f_0(y)$ if $x_k\leq y_k, 1\leq k\leq d$). In particular, suppose $x_0 \in (0,1)^d$ and $\partial_k f_0(x_0)>0$ for $1\leq k\leq d$, the errors $\xi_i$'s are i.i.d. mean-zero with variance $\sigma^2$, and the design points $\{X_i\}$ are of a fixed balanced design (see \cite{han2020limit} for a precise definition) or a random design with uniform distribution on $[0,1]^d$. Then \cite{han2020limit} showed that
\begin{align*}\label{eqn:limit_distribution_han_zhang}
&(n/\sigma^2)^{1/(d+2)}\big(\hat{f}_n(x_0)-f_0(x_0)\big) \limd \bigg(\prod_{k=1}^d \big(\partial_k f_0(x_0)/2\big)\bigg)^{1/(d+2)}\cdot \mathbb{D}_{(1,\ldots,1)},\nonumber
\end{align*}
where $\mathbb{D}_{(1,\ldots,1)}$ is a fairly complicated random variable generalizing the Chernoff distribution $\mathbb{D}_1$; a detailed description can be found in \cite{han2020limit}. We believe the techniques developed in this paper will be useful in establishing a Berry-Essen bound for the above limit {theorem}. However, the anti-concentration problem associated with $\mathbb{D}_{(1,\ldots,1)}$ in the multi-dimensional regression setting seems substantially more challenging than the univariate problem studied in this paper.

\appendix

	\section{Proof of auxiliary lemmas}

	\subsection{Proof of Lemma \ref{lem:BM_linear_drift}}

	\begin{proof}[Proof of Lemma \ref{lem:BM_linear_drift}]
		By \cite[formula (1.1.4), pp. 197]{borodin1996handbook}, noting that $\mathrm{Ercf}(z)=2(1-\Phi(\sqrt{2}z))$, we have for any $y\geq 0$,
		\begin{align*}
		\Prob ( M_{\mu} \ge y ) & = \frac{1}{2}\mathrm{Ercf}\bigg(\frac{y-\mu}{\sqrt{2}}\bigg)+\frac{1}{2}e^{2\mu y}\cdot \mathrm{Ercf}\bigg(\frac{y+\mu}{\sqrt{2}}\bigg)\\
		& = 1-\Phi\big(y-\mu)+e^{2\mu y}\big(1-\Phi\big(y+\mu)\big).
		\end{align*}
		Differentiating the above display with respect to $y$ yields (\ref{eqn:density_drifted_BM}), upon using $e^{2\mu y}\varphi(y+\mu) = \varphi(y-\mu)$. Alternatively, (\ref{eqn:density_drifted_BM}) can be derived using \cite[formula (1.9)]{shepp1979joint},
		\begin{align*}
		\Prob\big(M_\mu \in \d{y}, \mathbb{B}_\mu(1) \in \d{x}\big) = (2\pi)^{-1/2} 2(2y-x) e^{-\frac{(2y-x)^2}{2}}e^{\mu x-\frac{\mu^2}{2}}\bm{1}_{y\geq 0,y\geq x} \d{x}\d{y},
		\end{align*}
		which follows from the change of variables (or Cameron-Martin) formula for Gaussian measures (cf.  \cite[Theorem 2.6.13]{gine2015mathematical}).
		Hence for $y\geq 0$, $p_{M_\mu}(y)$ can be evaluated by integrating out $x$:
		\begin{align*}
		p_{M_\mu}(y) & = \frac{2}{\sqrt{2\pi}} \int_{-\infty}^y (2y-x) e^{-\frac{(2y-x)^2}{2}} e^{\mu x - \frac{\mu^2}{2}}\ \d{x} \\
		&= \frac{2}{\sqrt{2\pi}} \int_y^\infty t e^{-\frac{t^2}{2}} e^{\mu(2y-t)-\frac{\mu^2}{2} }\ \d{t}\\
		& = \frac{2 e^{2\mu y}}{\sqrt{2\pi}} \int_y^{\infty} t e^{-\frac{(t+\mu)^2}{2}}\ \d{t} = \frac{2 e^{2\mu y}}{\sqrt{2\pi}}  \bigg[\int_{y+\mu}^\infty ve^{-\frac{v^2}{2}}\ \d{v} -\mu \int_{y+\mu}^\infty e^{-\frac{v^2}{2}}\ \d{v}  \bigg]\\
		& = \frac{2 e^{2\mu y}}{\sqrt{2\pi}}\big[e^{-\frac{ (y+\mu)^2}{2}}-\sqrt{2\pi}\mu \big(1-\Phi(y+\mu)\big)\big],
		\end{align*}
		which agrees with (\ref{eqn:density_drifted_BM}). Since $p_{M_{\mu}}$ is discontinuous at $0$, $p_{M_\mu}(0)$ is understood as the right limit: $p_{M_{\mu}}(0)\equiv\lim_{y\to 0+} p_{M_{\mu}}(y)$. Finally, note that for $y+\mu\leq 1$, $e^{2\mu y}\big(1-\Phi\big(y+\mu)\big) \leq e^{2\mu(1-\mu)}(1-\Phi(1))\bm{1}_{\mu\geq 0}+\bm{1}_{\mu<0}\leq e^{1/2}$, while for $y+\mu>1$, $e^{2\mu y}\big(1-\Phi\big(y+\mu)\big)\leq (1/\sqrt{2\pi})e^{2\mu y-(y+\mu)^2/2}=(1/\sqrt{2\pi})e^{-(y-\mu)^2/2}\leq 1/\sqrt{2\pi}$. This implies that $\pnorm{p_{M_{\mu}}}{\infty}\lesssim (\mu \vee 1)$.
	\end{proof}

	\subsection{Proof of Lemma \ref{lem:bias_calculation}}

	\begin{proof}[Proof of Lemma \ref{lem:bias_calculation}]
		$\alpha=\infty$ is the trivial case, so we only consider $\alpha<\infty$. First consider fixed design with $x_0 \in (0,1)$. Then {for $x_0 \in \{X_i\}$,}
		\begin{align*}
		& \sum_{x_0-h_1r_n\leq X_i\leq x_0+h_2 r_n} \big(f_0(X_i)-f_0(x_0)\big)\\
		& =  \sum_{x_0-h_1r_n\leq X_i\leq x_0+h_2 r_n} \bigg[\frac{f_0^{(\alpha)}(x_0)}{\alpha!}(X_i-x_0)^{\alpha} +(1+o(1))\frac{f_0^{(\alpha^\ast)}(x_0)}{\alpha^\ast!}(X_i-x_0)^{\alpha^\ast}\bm{1}_{\alpha^\ast<\infty}\bigg]  \\
		& =  \sum_{-h_1 r_n \cdot \Lambda_0 n \leq  m\leq h_2 r_n \cdot \Lambda_0 n} \bigg[\frac{f_0^{(\alpha)}(x_0)}{\alpha!} \bigg(\frac{m}{\Lambda_0 n}\bigg)^{\alpha}  +(1+o(1)) \frac{f_0^{(\alpha^\ast)}(x_0)}{\alpha^\ast!} \bigg(\frac{m}{\Lambda_0 n}\bigg)^{\alpha^\ast}\bm{1}_{\alpha^\ast<\infty}\bigg] \\
		& = \frac{f_0^{(\alpha)}(x_0)}{(\alpha+1)!} \Lambda_0^{-\alpha} n^{-\alpha} \bigg[\floor{h_2r_n\cdot \Lambda_0 n}^{\alpha+1} +(-1)^{\alpha} \floor{h_1r_n\cdot \Lambda_0 n}^{\alpha+1}\\
		&\quad\quad +O\bigg(\big((h_1\vee h_2) r_n\Lambda_0 n\big)^{\alpha}\bigg)\bigg] +  n^{-\alpha^\ast}  O\bigg(\big((h_1\vee h_2) r_n\Lambda_0 n\big)^{\alpha^\ast+1} \bm{1}_{\alpha^\ast<\infty}\bigg)\\
		& = n r_n \bigg[\frac{f_0^{(\alpha)}(x_0)}{(\alpha+1)!}\cdot \big(h_2^{\alpha+1}-(-h_1)^{\alpha+1}\big)\cdot  \Lambda_0 r_n^{\alpha} + O \bigg(  \tau_n^\alpha r_n^{\alpha} (n r_n)^{-1} \bigvee  \tau_n ^{\alpha^\ast+1} r_n^{\alpha^\ast}  \bm{1}_{\alpha^\ast<\infty}\bigg)\bigg].
		\end{align*}
		{For general $x_0$ (not necessarily a design point), $X_i-x_0= m/(\Lambda_0 n)+O(n^{-1})$ for integers $m$ in the range $-h_1 r_n \cdot \Lambda_0 n \leq  m\leq h_2 r_n \cdot \Lambda_0 n$, so an extra term of order at most $O(nr_n\cdot \tau_n/n)$ would be contributed in the above summation.}
		
		For fixed design with $x_0 = 0$, {we work with $x_n \in \{X_i\}$ (the other case can be handled in similar way as above). Then}
		\begin{align*}
		& \sum_{x_n-h_1r_n\leq X_i\leq x_n+h_2 r_n} \big(f_0(X_i)-f_0(x_n)\big)\\
		& =\sum_{-h_1 r_n \cdot  \Lambda_0 n \leq  m\leq h_2 r_n \cdot \Lambda_0 n} \bigg[\sum_{\ell=1}^{\alpha^\ast-1}\frac{f_0^{(\ell)}(x_n)}{\ell!} \bigg(\frac{m}{ \Lambda_0 n}\bigg)^{\ell}  +(1+o(1)) \frac{f_0^{(\alpha^\ast)}(x_n)}{\alpha^\ast!} \bigg(\frac{m}{ \Lambda_0 n}\bigg)^{\alpha^\ast}\bm{1}_{\alpha^\ast<\infty}\bigg]\\
		& =  \sum_{\ell=1}^{\alpha} \frac{f_0^{(\alpha)}(0)}{(\alpha-\ell)!\ell!}\cdot \big(x_n^{\alpha-\ell}+O(x_n^{\alpha^\ast-\ell}\bm{1}_{\alpha^\ast<\infty})\big) \Lambda_0^{-\ell} n^{-\ell} \sum_{-h_1 r_n \cdot \Lambda_0 n \leq  m\leq h_2 r_n \cdot \Lambda_0  n} m^\ell\\
		&\quad\quad + f_0^{(\alpha^\ast)}(0)\bm{1}_{\alpha^\ast<\infty}\sum_{\ell=\alpha+1}^{\alpha^\ast} \frac{1+o(1)}{(\alpha^\ast-\ell)!\ell!}\cdot x_n^{\alpha^\ast-\ell} \Lambda_0^{-\ell} n^{-\ell} \sum_{-h_1 r_n \cdot \Lambda_0 n \leq  m\leq h_2 r_n \cdot \Lambda_0 n} m^\ell\\
		& =  \sum_{\ell=1}^{\alpha} \frac{f_0^{(\alpha)}(0)}{(\alpha-\ell)!(\ell+1)!}\cdot \big(x_n^{\alpha-\ell}+O(x_n^{\alpha^\ast-\ell}\bm{1}_{\alpha^\ast<\infty})\big) \Lambda_0^{-\ell}  n^{-\ell} \\
		&\quad\quad\quad \times \bigg\{\floor{h_2r_n \cdot \Lambda_0 n}^{\ell+1}-\big(-\floor{h_1r_n \cdot \Lambda_0 n}\big)^{\ell+1}+O\left(((h_1\vee h_2) r_n n)^\ell\right)\bigg\}\\
		&\quad\quad + O\bigg( \max_{\alpha+1\leq \ell \leq \alpha^\ast}x_n^{\alpha^\ast-\ell}n^{-\ell}(h_1\vee h_2)^{\ell+1} (r_n n)^{\ell+1} \bm{1}_{\alpha^\ast<\infty}\bigg)\\
		& = n r_n\bigg[ \sum_{\ell=1}^{\alpha} \frac{f_0^{(\alpha)}(0)}{(\alpha-\ell)!(\ell+1)!}\cdot \big(h_2^{\ell+1}-(-h_1)^{\ell+1} \big)\cdot \Lambda_0 x_n^{\alpha-\ell} r_n^{\ell}\\
		&\quad\quad + O\bigg(\max_{1\leq \ell \leq \alpha^\ast} \tau_n^{\ell+1} x_n^{\alpha^\ast-\ell} r_n^{\ell} \bm{1}_{\alpha^\ast<\infty} \bigvee \max_{1\leq \ell \leq \alpha } \tau_n^\ell x_n^{\alpha-\ell} r_n^\ell (n r_n)^{-1} \bigg)\bigg].
		\end{align*}
		Next consider random design with $x_0 \in (0,1)$. It is easy to see by Lemma \ref{lem:local_maximal_ineq} that for any $\ell\geq 1$,
		\begin{align*}
		& \E \sup_{0\leq h_i\leq \tau_n, i=1,2} \bigabs{n(\Prob_n-P)\big((X-x_0)^\ell\bm{1}_{[x_0-h_1 r_n,x_0+h_2r_n]}\big)}\lesssim \sqrt{n\cdot (\tau_n r_n)^{2\ell+1} \log n},  \\
		&\sup_{0\leq h_i\leq \tau_n, i=1,2}\mathrm{Var}\big( (X-x_0)^{\ell}\bm{1}_{[x_0-h_1 r_n,x_0+h_2r_n]}\big) \lesssim (\tau_n r_n)^{2\ell+1}.
		\end{align*}
		By Talagrand's concentration inequality (cf. Lemma \ref{lem:talagrand_conc_ineq}), there exists some constant $K>0$ such that for any $x\geq 0$,
		\begin{align*}
		&\Prob\bigg(K^{-1}\sup_{0\leq h_i\leq \tau_n, i=1,2} \bigabs{n(\Prob_n-P)\big( (X-x_0)^\ell \bm{1}_{[x_0-h_1 r_n,x_0+h_2r_n]}\big)}\\
		&\quad\quad \geq \sqrt{n(\tau_n r_n)^{2\ell+1} (\log n+x)}+x\bigg)\leq e^{-x}.
		\end{align*}
		Hence with probability at least $1-O(n^{-11})$, it holds that uniformly in $h_1,h_2\leq \tau_n$ 
		\begin{align*}
		& \sum_{x_0-h_1r_n\leq X_i\leq x_0+h_2 r_n} \big(f_0(X_i)-f_0(x_0)\big)\\
		& = \frac{f_0^{(\alpha)}(x_0)}{\alpha!} n \Prob_n (X-x_0)^{\alpha} \bm{1}_{[x_0-h_1r_n,x_0+h_2r_n]}\\
		&\quad\quad +(1+o(1))\bm{1}_{\alpha^\ast<\infty}\frac{f_0^{(\alpha^\ast)}(x_0)}{\alpha^\ast!} n \Prob_n (X-x_0)^{\alpha^\ast} \bm{1}_{[x_0-h_1r_n,x_0+h_2r_n]} \\
		& = \frac{f_0^{(\alpha)}(x_0)}{\alpha!} \bigg[n  P (X-x_0)^{\alpha} \bm{1}_{[x_0-h_1r_n,x_0+h_2r_n]} + O\bigg(\sqrt{n(\tau_n r_n)^{2\alpha+1}\log n}\vee \log n\bigg)\bigg]\\
		&\quad \quad +\bm{1}_{\alpha^\ast<\infty} \cdot O\bigg[n  P (X-x_0)^{\alpha^\ast} \bm{1}_{[x_0-h_1r_n,x_0+h_2r_n]}  + O\bigg(\sqrt{n(\tau_n r_n)^{2\alpha^\ast+1}\log n}\bigvee \log n\bigg)\bigg]\\
		& = nr_n\bigg[\frac{f_0^{(\alpha)}(x_0)}{(\alpha+1)!} \big(h_2^{\alpha+1}-(-h_1)^{\alpha+1}\big) \cdot \Lambda_0  r_n^{\alpha} \\
		&\quad\quad + O\bigg(\tau_n^{\alpha^\ast+1}r_n^{\alpha^\ast}\bm{1}_{\alpha^\ast<\infty}\bigvee  \tau_n^{\alpha+\beta+1} r_n^{\alpha+\beta}\bm{1}_{\beta<\infty}\bigvee \sqrt{\tau_n^{2\alpha+1} r_n^{2\alpha} \frac{\log n}{nr_n} }\bigvee \frac{\log n}{nr_n}\bigg)\bigg].
		\end{align*}
		Here we used that for all $\ell\geq 1$, 
		\begin{align*}
		&P (X-x_0)^{\ell} \bm{1}_{[x_0-h_1r_n,x_0+h_2r_n]}\\
		& = \int_{x_0}^{x_0+h_2r_n} (x-x_0)^\ell \pi(x)\ \d{x} + \int_{x_0-h_1 r_n}^{x_0} (x-x_0)^\ell \pi(x)\ \d{x}\\
		& = \frac{\Lambda_0}{\ell+1} \big(h_2^{\ell+1}-(-h_1)^{\ell+1}\big)r_n^{\ell+1} + O\big(((h_1\vee h_2)r_n)^{\ell+\beta+1}\bm{1}_{\beta<\infty}\big).
		\end{align*}
		For random design with $x_0 = 0$, with probability at least $1-O(n^{-11})$, we have uniformly in $h_1,h_2\leq \tau_n$,
		\begin{align*}
		& \sum_{x_n-h_1r_n\leq X_i\leq x_n+h_2 r_n} \big(f_0(X_i)-f_0(x_n)\big)\\
		& = \sum_{\ell=1}^{\alpha^\ast-1} \frac{f_0^{(\ell)}(x_n)}{\ell!} \cdot n \Prob_n (X-x_n)^\ell \bm{1}_{[x_n-h_1 r_n,x_n+h_2r_n]} \\
		&\quad\quad + (1+o(1)) \bm{1}_{\alpha^\ast<\infty}\frac{f_0^{(\alpha^\ast)}(x_n)}{\alpha^\ast!} \cdot n \Prob_n (X-x_n)^{\alpha^\ast} \bm{1}_{[x_n-h_1 r_n,x_n+h_2r_n]}\\
		& = \sum_{\ell=1}^{\alpha} \frac{f_0^{(\alpha)}(0)}{(\alpha-\ell)!\ell!} \big(x_n^{\alpha-\ell}+O(x_n^{\alpha^\ast-\ell}\bm{1}_{\alpha^\ast<\infty})\big)\\
		&\quad\quad\quad \times \bigg[n  P (X-x_0)^{\ell} \bm{1}_{[x_0-h_1r_n,x_0+h_2r_n]} + O\bigg(\sqrt{n(\tau_n r_n)^{2\ell+1}\log n}\bigvee \log n\bigg)\bigg]\\
		&\quad\quad + f_0^{(\alpha^\ast)}(0)\bm{1}_{\alpha^\ast<\infty}\sum_{\ell=\alpha+1}^{\alpha^\ast} \frac{1+o(1)}{(\alpha^\ast-\ell)!\ell!} x_n^{\alpha^\ast-\ell}\\
		&\quad\quad\quad \times \bigg[n  P (X-x_0)^{\ell} \bm{1}_{[x_0-h_1r_n,x_0+h_2r_n]} + O\bigg(\sqrt{n(\tau_n r_n)^{2\ell+1}\log n}\bigvee \log n\bigg)\bigg]\\
		& = \sum_{\ell=1}^{\alpha} \frac{f_0^{(\alpha)}(0)}{(\alpha-\ell)!(\ell+1)!} \Lambda_0 \big(h_2^{\ell+1}-(-h_1)^{\ell+1}\big)x_n^{\alpha-\ell} n r_n^{\ell+1}\\
		& \quad\quad +O\bigg(\max_{1\leq \ell \leq \alpha} x_n^{\alpha-\ell} \bigg\{ n(\tau_n r_n)^{\ell+\beta+1}\bm{1}_{\beta<\infty} \bigvee \sqrt{n(\tau_n r_n)^{2\ell+1}\log n}\bigvee \log n\bigg\}\bigg)\\
		&\quad\quad + O\bigg(\max_{1\leq \ell \leq \alpha^\ast} x_n^{\alpha^\ast-\ell} \bm{1}_{\alpha^\ast<\infty} \bigg\{ n (\tau_nr_n)^{\ell+1}\bigvee \sqrt{n(\tau_n r_n)^{2\ell+1}\log n}\bigvee \log n\bigg\}\bigg)\\
		& = nr_n \bigg[ \sum_{\ell=1}^{\alpha} \frac{f_0^{(\alpha)}(0)}{(\alpha-\ell)!(\ell+1)!} \big(h_2^{\ell+1}-(-h_1)^{\ell+1}\big)\cdot \Lambda_0 x_n^{\alpha-\ell} r_n^{\ell}\\
		&\quad\quad  +O\bigg(\max_{1\leq \ell \leq \alpha} x_n^{\alpha-\ell} \bigg\{ \tau_n^{\ell+\beta+1} r_n^{\ell+\beta}\bm{1}_{\beta<\infty} \bigvee \sqrt{\tau_n^{2\ell+1} r_n^{2\ell} \frac{\log n}{nr_n}}\bigvee \frac{\log n}{nr_n}\bigg\}\bigg)  \\
		&\quad\quad + O\bigg(\max_{1\leq \ell \leq \alpha^\ast} x_n^{\alpha^\ast-\ell} \bm{1}_{\alpha^\ast<\infty} \bigg\{ \tau_n^{\ell+1} r_n^{\ell}\bigvee \sqrt{\tau_n^{2\ell+1} r_n^{2\ell} \frac{\log n}{nr_n}}\bigvee \frac{\log n}{nr_n}\bigg\}\bigg) \bigg]\\
		& = nr_n \bigg[ \sum_{\ell=1}^{\alpha} \frac{f_0^{(\alpha)}(0)}{(\alpha-\ell)!(\ell+1)!} \big(h_2^{\ell+1}-(-h_1)^{\ell+1}\big)\cdot \Lambda_0 x_n^{\alpha-\ell} r_n^{\ell}\\
		&\quad\quad  +O\bigg( \max_{1\leq \ell\leq \alpha^\ast} \tau_n^{\ell+1}x_n^{\alpha^\ast-\ell} r_n^\ell \bm{1}_{\alpha^\ast<\infty} \bigvee \max_{1\leq \ell \leq \alpha} \tau_n^{\ell+\beta+1} x_n^{\alpha-\ell} r_n^{\ell+\beta}\bm{1}_{\beta<\infty}\\
		&\quad\quad\quad \bigvee \max_{1\leq \ell \leq \alpha^\ast} (x_n^{\alpha-\ell}\bm{1}_{1\leq \ell \leq \alpha}+x_n^{\alpha^\ast-\ell}\bm{1}_{\alpha^\ast<\infty})\bigg\{\sqrt{\tau_n^{2\ell+1} r_n^{2\ell} \frac{\log n}{nr_n}}\bigvee \frac{\log n}{nr_n}\bigg\}\bigg)\bigg],
		\end{align*}
		as desired.
	\end{proof}

	\subsection{Proof of Lemma \ref{lem:tail_bound_noise}}

	\begin{proof}[Proof of Lemma \ref{lem:tail_bound_noise}]
		Let $\tilde{x}_n \equiv x^\ast+r_n$. First consider fixed design. Let $\ell_n$ be the smallest integer such that $\tilde{x}_n-2^{\ell_n}r_n<0$. Then
		\begin{align*}
		&\Prob\bigg(\sup_{h\geq 0}\abs{ \bar{\xi}|_{[x^\ast-hr_n,x^\ast+r_n]}}>t\omega_n\bigg)=\Prob\bigg(\sup_{h\geq 1}\abs{ \bar{\xi}|_{[\tilde{x}_n-hr_n, \tilde{x}_n]}}>t\omega_n\bigg) \\
		&\leq \sum_{\ell =0}^{\ell_n}  \Prob\bigg(\sup_{2^{\ell}\leq h< 2^{\ell+1}}\abs{ \bar{\xi}|_{[\tilde{x}_n-hr_n,\tilde{x}_n]}}>t\omega_n\bigg)\\
		&\leq \sum_{\ell =0}^{\ell_n}  \Prob\bigg(\sup_{2^{\ell}\leq h< 2^{\ell+1}}\biggabs{\sum_{i=1}^n \xi_i \bm{1}_{X_i \in [\tilde{x}_n-hr_n,\tilde{x}_n]} }>t\omega_n \floor{2^\ell r_n\cdot \Lambda_0 n}\bigg).
		\end{align*}
		By L\'evy's maximal inequality (cf. \cite[Theorem 1.1.5]{de2012decoupling}), each probability in the above summation can be bounded, up to an absolute constant, by
		\begin{align*}
		&\Prob\bigg( \biggabs{\sum_{i=1}^n \xi_i \bm{1}_{X_i \in [\tilde{x}_n-2^{\ell+1}r_n,\tilde{x}_n]} } > K_1\cdot \sqrt{\floor{2^{\ell+1} r_n\cdot \Lambda_0 n }} \cdot \big(2^\ell r_n n\big)^{1/2} t \omega_n\bigg)\\
		&\leq K_2 \bigg[\exp\big(- 2^\ell t^2/K_2\big)+\exp\big(-(n r_n)^{1/2} 2^\ell t/K_2\big)\bigg].
		\end{align*}
		Here we used the following facts: (i) for centered sub-exponential random variables $\xi_1,\ldots,\xi_m$,  $\Prob\big(\abs{\sum_{i=1}^m \xi_i}>\sqrt{m}u\big)\leq Ke^{-(u^2\wedge \sqrt{m} u)/K}$ holds for $u\geq 0$ (cf. \cite[Proposition 3.1.8]{gine2015mathematical}), and (ii) $(nr_n)^{1/2}=\omega_n^{-1}$. The claim for the fixed design case now follows by summing up the probabilities. 
		
		For the random design case, without loss of generality we work with $P$ being the uniform distribution on $[0,1]$.  First note that by applying (essentially) \cite[Lemma 10]{han2020limit} with $L_n\equiv nr_n/\log n$, with probability at least $1-O(n^{-11})$, we have uniformly in $\ell$,
		\begin{align*}
		\frac{\Prob_n \bm{1}_{[\tilde{x}_n-2^\ell r_n,\tilde{x}_n]}}{P \bm{1}_{[\tilde{x}_n-2^\ell r_n,\tilde{x}_n]}} = 1 + O(L_n^{-1/2}) = 1+O(\omega_n\sqrt{\log n}).
		\end{align*}
		Equivalently, the event 
		\begin{align*}
		\mathcal{E}_n\equiv \bigg\{\Prob_n \bm{1}_{[\tilde{x}_n-2^\ell r_n,\tilde{x}_n]} = 2^\ell r_n \big(1+O(\omega_n\sqrt{\log n})\big): \ell \geq 1 \bigg\}
		\end{align*}
		satisfies $\Prob(\mathcal{E}_n^c) = O(n^{-11})$. Hence, up to an additive term of order $O(n^{-11})$, we only need to control
		\begin{align*}
		& \sum_{\ell=0}^{\ell_n}\Prob\bigg( \bigg \{ \sup_{2^{\ell}\leq h< 2^{\ell+1}}\biggabs{\sum_{i=1}^n \xi_i \bm{1}_{X_i \in [\tilde{x}_n-hr_n,\tilde{x}_n]} }>t\omega_n  n\Prob_n\bm{1}_{[\tilde{x}_n-2^\ell r_n,\tilde{x}_n]} \bigg \} \cap \mathcal{E}_{n} \bigg) \\
		&\lesssim \sum_{\ell=0}^{\ell_n}\E \bigg [ \Prob\bigg(\biggabs{\sum_{i=1}^n \xi_i \bm{1}_{X_i \in [\tilde{x}_n-2^{\ell+1}r_n,\tilde{x}_n]} }\gtrsim t\omega_n \cdot \Prob_n\bm{1}_{[\tilde{x}_n-2^\ell r_n,\tilde{x}_n]}\bigg\lvert \{X_i\}\bigg) \bm{1}_{\mathcal{E}_n}\bigg ] \\
		&=\sum_{\ell=0}^{\ell_n}\E \bigg [ \Prob\bigg(\biggabs{\sum_{i=1}^n \xi_i \bm{1}_{X_i \in [\tilde{x}_n-2^{\ell+1}r_n,\tilde{x}_n]} }  \nonumber \\
		&\quad\quad\quad  \gtrsim \sqrt{n\Prob_n\bm{1}_{[\tilde{x}_n-2^{\ell+1} r_n,\tilde{x}_n]}}\cdot  t\omega_n \cdot \frac{n\Prob_n\bm{1}_{[\tilde{x}_n-2^\ell r_n,\tilde{x}_n]}}{\sqrt{n\Prob_n\bm{1}_{[\tilde{x}_n-2^{\ell+1} r_n,\tilde{x}_n]}}}\bigg\lvert \{X_i\}\bigg)\bm{1}_{\mathcal{E}_n}\bigg ] \nonumber\\
		&\leq K_3 \sum_{\ell=0}^{\ell_n} \E \exp \bigg(-K_3^{-1}\min\bigg\{t^2 n \omega_n^2 \frac{ \big(\Prob_n\bm{1}_{[\tilde{x}_n-2^\ell r_n,\tilde{x}_n]}\big)^2}{\Prob_n\bm{1}_{[\tilde{x}_n-2^{\ell+1} r_n,\tilde{x}_n]}}, t n\omega_n \Prob_n \bm{1}_{[\tilde{x}_n-2^\ell r_n,\tilde{x}_n]}\bigg\} \bm{1}_{\mathcal{E}_n}\bigg)\nonumber \\
		&\lesssim \sum_{\ell=0}^{\ell_n} \exp\bigg(-K_4^{-1}\min\Big\{ t^2 n\omega_n^2 2^\ell r_n (1+O(\omega_n \sqrt{\log n})),  tn\omega_n 2^\ell r_n (1+O(\omega_n \sqrt{\log n}))\Big\}\bigg) \\
		&= \sum_{\ell=0}^{\ell_n} \exp\big(-K_5^{-1}\min\big\{ 2^\ell t^2 , (nr_n)^{1/2}2^\ell t\big\}\big).
		\end{align*}
		The claim now follows by summing up the probabilities.
	\end{proof}

	\subsection{Proof of Lemma \ref{lem:h_tilde_well_defined}}

	The proof of Lemma \ref{lem:h_tilde_well_defined} relies on the following technical lemma, which will be also used in the proof of Theorem \ref{thm:berry_esseen_isoreg} in Section \ref{section:proof_berry_esseen_isotonic} for the case with $x_0 =0, 1/(2\alpha+1)<\rho<1$, $\alpha<\infty$. 
	
	\begin{lemma}\label{lem:anti_concentration_oneside}
		There exists some constant $K>0$ such that for any $\epsilon>0$,
		\begin{align*}
		\Prob\bigg(\sup_{h_1 \in (0,1]} \inf_{h_2 \in [0,\infty)} \frac{\mathbb{B}(h_2)-\mathbb{B}(-h_1)}{h_1+h_2}\geq -\epsilon\bigg)\leq K \epsilon^{1/2}\log_+^{1/4}(1/\epsilon).
		\end{align*}
	\end{lemma}
	\begin{proof}
		Let $M_\epsilon=\max\{1, \epsilon^{-1}\log_+^{1/2}(1/\epsilon)\}$. Note that the probability in question can be bounded by
		\begin{align*}
		&\Prob\bigg(\exists h_1 \in (0,1],\forall h_2 \in [0,M_\epsilon], \mathbb{B}(h_2)+\epsilon h_2 \geq \mathbb{B}(-h_1)- \epsilon h_1\bigg)\\
		&\leq \Prob\bigg(\exists h_1 \in (0,1],\forall h_2 \in [0,M_\epsilon], \mathbb{B}(h_2)\geq \mathbb{B}(-h_1)-(M_\epsilon+1) \epsilon\bigg)\\
		& = \Prob\bigg( \inf_{h_2 \in [0,M_\epsilon]} \mathbb{B}(h_2)\geq \inf_{h_1 \in (0,1]} \mathbb{B}(-h_1)-(M_\epsilon+1)\epsilon\bigg).
		\end{align*}
		By the reflection principle of a  Brownian motion, we have 
		\[
		\left (\inf_{h_2 \in [0,M_\epsilon]} \mathbb{B}(h_2), \inf_{h_1 \in (0,1]} \mathbb{B}(-h_1) \right) \stackrel{d}{=} (-\sqrt{M_\epsilon}\cdot \abs{Z_1},-\abs{Z_2}),
		\]
		where $Z_1,Z_2$ are independent standard normal random variables. Hence the above display further equals
		\begin{align*}
		&\Prob\big(-\sqrt{M_\epsilon} \abs{Z_1}\geq -\abs{Z_2}-(M_\epsilon+1)\epsilon\big)\\
		&\leq \Prob\bigg(\abs{Z_1}\leq \big(200\log_+(1/\epsilon)/M_\epsilon\big)^{1/2}+2\sqrt{M_\epsilon}\cdot \epsilon\bigg)+\Prob\big(\abs{Z_2}\geq (200 \log_+(1/\epsilon))^{1/2}\big)\\
		&\leq K \epsilon^{1/2}\log_+^{1/4}(1/\epsilon)+ O(\epsilon^{100}),
		\end{align*}
		as desired.
	\end{proof}
	
	\begin{proof}[Proof of Lemma \ref{lem:h_tilde_well_defined}]
		We only consider the case for $\tilde{h}_2$ with $H_2 = [0,\infty)$, and for notational simplicity we set $\sigma=1$ and $\Lambda_0=1$. Geometrically, $\tilde{h}_2$ is the first touch point of $\mathbb{B}_{1,1,Q}$ and its global LCM on $H_2$, so it is well-defined on  the event $\cup_{n=1}^\infty \cap_{M=n}^\infty \mathcal{E}_M$, where $\mathcal{E}_M\equiv \{\sup_{h_1 \in H_1}\inf_{h_2 \in H_2} \mathbb{B}_{1,1,Q}(h_1,h_2) = \sup_{h_1 \in H_1}\inf_{h_2 \in [0,M]} \mathbb{B}_{1,1,Q}(h_1,h_2) \}$. 
		
		First consider the cases $x_0 \in (0,1)$ or $x_0 =0$ with $\rho \in (0,1/(2\alpha+1]$. In this case $Q$ is a non-vanishing polynomial of degree at least $2$. Then on the event $\mathcal{E}_M^c$, 
		\begin{align*}
		&\sup_{h_1 \in H_1}\inf_{h_2 \in H_2} \mathbb{B}_{1,1,Q}(h_1,h_2)  = \sup_{h_1 \in H_1}\inf_{h_2>M} \mathbb{B}_{1,1,Q}(h_1,h_2) \\
		&\geq \inf_{h_2>M} \mathbb{B}_{1,1,Q}(1,h_2) = \inf_{h_2 >M} \frac{\mathbb{B}(h_2)-\mathbb{B}(-1)+Q(h_2)-Q(-1)}{1+h_2}\\
		&\geq O(M)- \sup_{h>M} \frac{ \abs{\mathbb{B}(h)}}{1+h}-\frac{\abs{\mathbb{B}(-1)}}{1+M} \equald O(M) - \frac{Y_1}{M^{1/2}}-\frac{Y_2}{M+1},
		\end{align*}
		where $Y_1\equiv \sup_{h>1}\abs{\mathbb{B}(h)/h}$ and $Y_2\equiv \abs{\mathbb{B}(-1)}$ have sub-Gaussian tails. Hence for $M$ large, on the intersection of $\mathcal{E}_M^c$ and an event with probability at least $1- Ke^{-M^2/K}$,
		\begin{align*}
		\sup_{h_1 \in H_1}\inf_{h_2 \in H_2} \mathbb{B}_{1,1,Q}(h_1,h_2) \geq O(M)-\sqrt{M}-\frac{M}{M+1} \geq O(M).
		\end{align*}
		Since the random variable $\sup_{h_1 \in H_1}\inf_{h_2 \in H_2} \mathbb{B}_{1,1,Q}(h_1,h_2)$ has sub-Gaussian tails (using a similar proof to that of Lemma \ref{lem:large_deviation_0} above), we see that  $\Prob(\mathcal{E}_M^c)\leq Ke^{-M^2/K}$. Using Borel-Cantelli yields that $\Prob(\cup_{n=1}^\infty \cap_{M=n}^\infty \mathcal{E}_M) =1$. 
		
		Next consider the case $x_0 = 0$ with $\rho \in (1/(2\alpha+1),1)$. In this case $Q\equiv 0$. Then on the event $\mathcal{E}_M^c$, 
		\begin{align*}
		&\sup_{h_1 \in (0,1]}\inf_{h_2 \in [0,\infty)} \mathbb{B}_{1,1,0}(h_1,h_2)  \geq \inf_{h_2>M}  \frac{\mathbb{B}(h_2)-\mathbb{B}(-1)}{h_2+1}\equald -\frac{Y_1}{M^{1/2}}-\frac{Y_2}{M+1},
		\end{align*}
		where $Y_1,Y_2$ are defined as above. This means that on the intersection of $\mathcal{E}_M^c$ and an event with probability at least $1-M^{-100}$, 
		\begin{align*}
		\sup_{h_1 \in (0,1]}\inf_{h_2 \in [0,\infty)} \mathbb{B}_{1,1,0}(h_1,h_2) \geq - K (\log M/M)^{1/2},
		\end{align*}
		which occurs with probability at most $O( M^{-1/4} \log^{1/2}M)$ according to Lemma \ref{lem:anti_concentration_oneside}. Hence $P(\mathcal{E}_M^c)\leq O(M^{-1/4} \log^{1/2}M)$. Summing $M$ through a geometric sequence and using Borel-Cantelli yields the claim.
	\end{proof}
	
	\section{Technical tools}
	
	This appendix collects some technical tools used in the proofs.
	%
	%The following Dudley's entropy integral bound for sub-Gaussian processes and Gaussian concentration inequality will be also used. 
	The following Dudley's entropy integral bound can be found in \cite[Theorem 2.3.7]{gine2015mathematical}.
	\begin{lemma}[Entropy integral bound]\label{lem:dudley_entropy_integral}
		Let $(T,d)$ be a pseudometric space, and $(X_t)_{t \in T}$ be a separable sub-Gaussian process such that $X_{t_0}=0$ for some $t_0 \in T$. Then
		\begin{align*}
		\E \sup_{t \in T} \abs{X_t} \leq C\int_0^{\mathrm{diam}(T)} \sqrt{\log \mathcal{N}(\epsilon,T,d)}\ \d{\epsilon},
		\end{align*}
		where $C>0$ is a universal constant.
	\end{lemma}
	
	The following Gaussian concentration inequality can be found in \cite[Theorem 2.5.8]{gine2015mathematical}.
	
	\begin{lemma}[Gaussian concentration inequality]\label{lem:Gaussian_concentration}
		Let $(T,d)$ be a pseudometric space, and $(X_t)_{t \in T}$ be a separable mean-zero Gaussian process with $\sup_{t \in T} | X_{t} | < \infty$ a.s.  Then, with $\sigma^2\equiv \sup_{t \in T} \mathrm{Var}(X_t)$, for any $u >0$,
		\begin{align*}
		\Prob\left( \left | \sup_{t \in T} \abs{X_t} - \E \sup_{t \in T} \abs{X_t} \right|> u  \right)\le 2e^{-u^{2}/(2\sigma^{2})}. 
		\end{align*}
	\end{lemma}

	Talagrand's concentration inequality \cite{talagrand1996new} for the empirical process in the form given by Bousquet \cite{bousquet2003concentration}, is recorded as follows, cf. \cite[Theorem 3.3.9]{gine2015mathematical}.
	
	\begin{lemma}[Talagrand's concentration inequality]\label{lem:talagrand_conc_ineq}
		Let $\mathcal{F}$ be a countable class of real-valued measurable functions such that $\sup_{f \in \mathcal{F}} \pnorm{f}{\infty}\leq b$ and $X_1,\ldots,X_n$  be i.i.d. random variables with law $P$. Then there exists some absolute constant  $K>1$ such that
		\begin{align*}
		\Prob\bigg( K^{-1} \sup_{f \in \mathcal{F}}\abs{n(\Prob_n-P) f} \geq \E\sup_{f \in \mathcal{F}}\abs{n(\Prob_n-P) f} +\sqrt{ n\sigma^2 x}+b x \bigg)\leq e^{-x},
		\end{align*}
		where $\sigma^2\equiv \sup_{f \in \mathcal{F}} \mathrm{Var}_P f$ and $\Prob_n$ denotes the empirical distribution of $X_{1},\dots,X_{n}$. 
	\end{lemma}
	
	Talagrand's inequality is coupled with the following local maximal inequality for the empirical process due to \cite{gine2006concentration,van2011local}.
	Denote the uniform entropy integral by 
	\[
	J(\delta,\mathcal{F},L_2) \equiv   \int_0^\delta  \sup_Q\sqrt{1+\log \mathcal{N}(\epsilon\pnorm{F}{Q,2},\mathcal{F},L_2(Q))}\ \d{\epsilon},
	\]
	where the supremum is taken over all finitely discrete probability measures.
	
	\begin{lemma}
		[Local maximal inequality]
		\label{lem:local_maximal_ineq}
		Let $\mathcal{F}$ be a countable class of real-valued measurable functions such that $\sup_{f \in \mathcal{F}} \pnorm{f}{\infty}\leq 1$, and $X_1,\ldots,X_n$ be i.i.d. random variables with law $P$. Then with $\mathcal{F}(\delta)\equiv \{f \in \mathcal{F}:Pf^2<\delta^2\}$, 
		\begin{align*}
		\E \sup_{f \in \mathcal{F}(\delta)}\bigabs{(\Prob_n-P)(f)}  \lesssim n^{-1/2} J(\delta,\mathcal{F},L_2)\bigg(1+\frac{J(\delta,\mathcal{F},L_2)}{\sqrt{n} \delta^2 \pnorm{F}{P,2}}\bigg)\pnorm{F}{P,2}.
		\end{align*}
	\end{lemma}

\section*{Acknowledgments}
The authors would like to thank Jon Wellner for pointing out several references. They also would like to thank the anonymous
	referees, an Associate Editor, and the Editor for their constructive comments
	that improve the quality of this paper.

\bibliographystyle{amsalpha}
\bibliography{mybib}

\end{document}